\documentclass[10pt]{amsart}
\usepackage[leqno]{amsmath}
\usepackage{amssymb,latexsym,soul,cite,amsthm,color,enumitem,graphicx,mathtools,microtype,accents, esint, tikz}
\usepackage[colorlinks=true,urlcolor=cobalt,citecolor=cobalt,linkcolor=cobalt,linktocpage,pdfpagelabels,bookmarksnumbered,bookmarksopen]{hyperref}
\definecolor{cobalt}{rgb}{0.0, 0.28, 0.67}
\usepackage[english]{babel}
\usepackage[left=2.5cm,right=2.5cm,top=2.5cm,bottom=2.5cm]{geometry}

\numberwithin{equation}{section}

\newtheorem{theorem}{Theorem}[section]
\theoremstyle{plain}
\newtheorem{lemma}[theorem]{Lemma}
\theoremstyle{plain}
\newtheorem{proposition}[theorem]{Proposition}
\theoremstyle{plain}
\newtheorem{corollary}[theorem]{Corollary}
\newtheorem{definition}[theorem]{Definition}
\theoremstyle{definition}
\newtheorem{remark}[theorem]{Remark}

\newcommand{\N}{{\mathbb N}}

\newcommand{\R}{{\mathbb R}}
\newcommand{\eps}{\varepsilon}
\newcommand{\beq}{\begin{equation}}
\newcommand{\eeq}{\end{equation}}
\renewcommand{\le}{\leqslant}
\renewcommand{\ge}{\geqslant}

\def\Xint#1{\mathchoice
{\XXint\displaystyle\textstyle{#1}}%
{\XXint\textstyle\scriptstyle{#1}}%
{\XXint\scriptstyle\scriptscriptstyle{#1}}%
{\XXint\scriptscriptstyle\scriptscriptstyle{#1}}%
\!\int}
\def\XXint#1#2#3{{\setbox0=\hbox{$#1{#2#3}{\int}$ }
\vcenter{\hbox{$#2#3$ }}\kern-.6\wd0}}
\def\fint{\Xint-}

\def\Yint#1{\mathchoice
    {\YYint\displaystyle\textstyle{#1}}%
    {\YYint\textstyle\scriptstyle{#1}}%
    {\YYint\scriptstyle\scriptscriptstyle{#1}}%
    {\YYint\scriptscriptstyle\scriptscriptstyle{#1}}%
      \!\iint}
\def\YYint#1#2#3{{\setbox0=\hbox{$#1{#2#3}{\iint}$}
    \vcenter{\hbox{$#2#3$}}\kern-.51\wd0}}
\def\longdash{{-}\mkern-3.5mu{-}} 

\def\fiint{\Yint\longdash}

\makeatletter
\newcommand{\leqnomode}{\tagsleft@true}
\newcommand{\reqnomode}{\tagsleft@false}
\makeatother

\newenvironment{enumroman}{\begin{enumerate}

}{\end{enumerate}}

\title[Singular Fractional $p$-Laplacian]{Integral Harnack estimates and the rate of extinction of singular fractional diffusion}

\author[F.M.\ Cassanello, S.\ Ciani, A.\ Iannizzotto]{Filippo Maria Cassanello, Simone Ciani, Antonio Iannizzotto}

\address[F.M.\ Cassanello, A.\ Iannizzotto]{Dipartimento di Matematica e Informatica
\newline\indent
Universit\`a degli Studi di Cagliari
\newline\indent
Via Ospedale 72, 09124 Cagliari, Italy}
\email{filippom.cassanello@unica.it}
\email{antonio.iannizzotto@unica.it}

\address[S.\ Ciani]{Dipartimento di Matematica
\newline\indent
Universit\`a di Bologna
\newline\indent
Piazza di Porta San Donato 5, Bologna, Italy}
\email{simone.ciani3@unibo.it}

\subjclass[2020]{35K67, 35B65, 35K92, 35Q35..}
\keywords{Integral Harnack-type inequality, $L^1$-$L^1$ Harnack tipe inequality, Fractional $p$-Laplacian, Rate of Extinction, Decay.}

\begin{document}

\begin{abstract}
We prove several integral Harnack-type inequalities for local weak solutions of parabolic equations with measurable and bounded coefficients, describing singular $s$-fractional $p$-Laplacian diffusion. Then we apply such estimates to evaluate the decay rate of the local mass and supremum of the solutions as they approach a possible extinction time. Yet we show consistency of our general decay estimates by studying the extinction phenomenon for weak solutions of the Cauchy-Dirichlet problem, by means of an approximation procedure that carefully avoids the use of an integrable time derivative.
\end{abstract}

\maketitle

\section{Introduction and main results}\label{sec1}

\noindent \subsection{Heuristics} \label{heuristics} When describing the flow of a non-Newtonian fluid in a simple situation (as in a pipe), the momentum balance law written for a power-like stress tensor can prescribe a dissipation of energy that distinguishes between {\it dilatant} fluids, which starting still, stay immobile until a time $T^*$ and {\it pseudoplastic} fluids, that become immobile after a finite time $T^*$ has passed (see \cite{ADS} Ch.IV Section 7.6). We are here interested in the latter phenomenon (as opposite to the former one), which we rephrase under the more general principle of extinction of a diffusion process after a finite time. As anticipated, this principle is a consequence of the dissipation of energy involved in the evolutive process, that in general can be supplied by a particular source or, as in our case, by the unbalance between the energy of the propagation (parabolic energy terms with power-growth $\approx 2$) and the one of the diffusion (elliptic energy terms with a slower power-growth $\approx p<2$). We refer to the classic books \cite{ADS}, \cite{AS} for a presentation of energy methods for the study of the localization of the solutions to parabolic equations.
\vskip2pt 
\noindent
When the aim is the description of materials with memory or media that exhibit long-range elastic or plastic deformation, these models account for the fact that stresses or strains at one point in a material can be influenced by other regions over a nonlocal range, rather than just the local neighborhood (see for instance \cite{Phys1}, \cite{Caffa1}, \cite{Caffa-Annals}, \cite{Phys3}, \cite{Phys2}, \cite{Vaz-appl} and references therein) and the diffusion is termed {\em nonlocal}. It is the precise scope of this work to address the study of the interplay between local and nonlocal effects caused by the phenomenon of extinction and the regularity properties of weak solutions of diffusion processes described, in particular, by the fractional $p$-Laplacian equation
\begin{equation} \label{eq0}
u_t (x,t)= \int_{\R^N} \frac{|u(y,t)-u(x,t)|^{p-2}(u(y,t)-u(x,t))}{|x-y|^{N+sp}} \mu(x,y,t) \,dy\,,\qquad x \in \Omega, \  t\in [0,T],
\end{equation} \noindent
where $p\in(1,2)$, $s\in(01)$, and $\mu$ is a measurable bounded function, prescribed by the anisotropy of the medium and reflecting the impossibility to measure the properties described by the model without interfering with themselves (see \cite{DBGV-mono} discussion at 3.1 Ch.I). 

\subsection{Framing of the Topic and Main Results}\label{framing} We consider the following parabolic nonlinear, nonlocal equation
\beq\label{eq}
u_t+\mathcal{L}_Ku = 0,
\eeq
in the cylindrical set $\Omega_T=\Omega \times [0,T]$, with $\Omega \subset \R^N$ open ($N\ge 2$). The diffusion operator $\mathcal{L}_K$ is formally defined by
\[\mathcal{L}_Ku(x,t) = 2\lim_{\eps\to 0^+}\int_{B_\eps^c(x)}|u(x,t)-u(y,t)|^{p-2}(u(x,t)-u(y,t))K(x,y,t)\,dy,\]
where $p\in(1,2)$, $s\in(0,1)$, and $K:\R^N\times\R^N\times(0,T)\to\R$ is a measurable function satisfying, for constants $0<C_1\le C_2$, the following structural properties almost everywhere:
\begin{itemize}[leftmargin=1cm]
\item[$(K_1)$] $K(x,y,t)=K(y,x,t)$; \vskip0.2cm \noindent
\item[$(K_2)$] $C_1\le K(x,y,t)|x-y|^{N+ps}\le C_2$ \,.
\end{itemize} We remark that when  $C_1=C_2=1$, the operator $\mathcal{L}_K$ reduces to the prototype $s$-fractional $p$-Laplacian 
\[(-\Delta)_p^s u(x,t) = 2\lim_{\eps\to 0^+}\int_{B_\eps^c(x)}\frac{|u(x,t)-u(y,t)|^{p-2}(u(x,t)-u(y,t))}{|x-y|^{N+ps}}\,dy.\]
For the precise definition of solution adopted we refer to Section \ref{sec2}. For some related results in the elliptic framework, see \cite{CDI}, \cite{DIV} and the survey paper \cite{I}. Our choice of $p\in(1,2)$ qualifies the diffusion operator as {\em singular}, since when $u(x,t)=u(y,t)$ the elliptic term of the diffusion dominates the process. The former quality of the operator has significant consequences on the properties of the solutions. Indeed, consider the associated Cauchy-Dirichlet problem:
\beq\label{cdp}
\begin{cases}
u_t+\mathcal{L}_Ku = 0 & \text{in $\Omega_T$} \\
u = 0 & \text{in $\Omega^c\times(0,T)$} \\
u(\cdot,0) = u_0 & \text{in $\R^N$},
\end{cases}
\eeq
with $\Omega$ bounded and initial datum $u_0\in W^{s,p}_0(\Omega)$, i.e., $u_0\in W^{s,p}(\R^N)$ and $u_0=0$ in $\Omega^c$ (see Section \ref{sec2} for details). We will prove that (global) weak solutions to \eqref{cdp} extinguish within a time $T_*$ that depends on some $L^q$ norm of $u_0$ (see Theorem \ref{fte} for the precise statement), while local\footnote{Meaning with this, weak solutions that are irrespective of any boundary or initial datum, see Section \ref{sec2} for the precise definition.} weak solutions satisfy an integral Harnack-type inequality such as, for fixed $\rho, t>0$,
\[ {\gamma}^{-1} \sup_{0<\tau<t} \int_{B_{\rho}} u(x,\tau)\, d x \leq \inf_{0<\tau<t} \int_{B_{2\rho}} u(x,\tau)\, dx + \bigg(\frac{t}{\rho^{\lambda_1}} \bigg)^{\frac{1}{2-p}} +\mathcal{T},\qquad \qquad  \gamma>1,\]
where the second term on the right-hand side takes into account the possible global effects in time, while the term $\mathcal{T}$ involves the long-range values of the solution in space, through the quantity 
\beq\label{tail}
{\rm Tail}\big(u,x_0,\rho,t_1,t_2\big) = \sup_{t_1<\tau<t_2}\,\Big[\rho^{ps}\int_{B_\rho^c(x_0)}\frac{|u(x,\tau)|^{p-1}}{|x-x_0|^{N+ps}}\,dx\Big]^\frac{1}{p-1},
\eeq
for fixed $0<t_1 \leq t_2\leq T$, $x_0 \in \Omega$, that we refer to as the {\em nonlocal tail} of $u:\R^N\times(0,T)\to\R$ (see for instance~\cite{DKP14, DKP16} for the origin of the term, and \cite{Liao1} for an alternative definition of tail and its consequences).
\vskip2pt
\noindent
Both the property of extinction and this integral Harnack-type inequality are a consequence of the fact that the operator is {\em singular}, paralleling the description of the pseudoplastic fluids of Section \ref{heuristics}. We address the previous integral Harnack inequality, following the terminology of \cite{DB}, as {\em an $L^1$-$L^1$ Harnack inequality}, since it is an Harnack-type estimate for the function $t \rightarrow \|u(\cdot, t)\|_{L^1(B_{\rho})}$. Another peculiarity of the range $1<p<2$ is the fact that local weak solutions are not necessarily locally bounded (see \cite{DB} Ch. V for a comparison with the $p$-Laplacian case). A byproduct of our analysis shows with a quantitative estimate that, if $p,s$ satisfy the following relation
\begin{equation}\label{critical}
   \frac{2N}{N+2s}=:p_c<p<2,
\end{equation} \noindent then local weak solutions to \eqref{eq} are locally bounded, provided that the tail terms are adequately controlled (see Remark \ref{RMK-Bundedness} below for the details). Therefore we address the exponent $p_{c}$ as {\em critical}, and we will say that the equation \eqref{eq} lies in the {\em subcritical regime} if $1<p<p_c$, and in the {\em supercritical regime} if $p_c<p<2$, respectively. Connecting the $L^1$-$L^1$ Harnack inequality with the estimates for the local boundedness (Theorem \ref{lr}), we show some very useful $L^1$-$L^{\infty}$ Harnack-type estimates (Theorem \ref{linf}) that can be used directly to check the rate of decay toward extinction, and show that the behaviour of local weak solutions is, roughly speaking, dictated by the rule
\[\gamma^{-1} \|u(\cdot, \tau)\|_{L^{\infty}(B_{\rho})} \leq  \bigg( \inf_{[0,\tau]}\|u(\cdot, \tau)\|_{L^{1}(B_{2\rho})} \bigg)^{\frac{ps}{\lambda_1}} \tau^{-\frac{N}{\lambda_1}}+ \bigg( \frac{\tau}{\rho^{ps}}\bigg)^{\frac{1}{2-p}} +\hat{\mathcal{T}},\]
where again $\hat{\mathcal{T}}$ is a perturbation term involving the long-range values of the solution in space and $\lambda_1= N(p-2)+ps$ is the fractional Barenblatt number (see \cite{Vaz1} Thm 1.1).

\vskip2pt
\noindent
We now present our main results. For all $r\ge 1$ we define 
\[\lambda_r = N(p-2)+rps,\]
and we note that $p_c=2N/(N+2s)$ implies $\lambda_2>0$ for all $p>p_c$. In order to ease notation, say that a constant $\gamma$ depends only on the data if it depends only on $\{N,p,s,C_1,C_2\}$, where $C_1,C_2$ are the constants appearing in $(K_2)$.
\vskip2pt
\noindent 
Our first result states that, if $u$ is a locally bounded, local weak solution of \eqref{eq}, then the supremum of $u$ in a cylinder is controlled by the average of $u^r$ in a larger cylinder plus a perturbative term that depends on the tail and on the ratio $t/\rho^{ps}$, which accounts for long-range behavior with respect to time.

\begin{theorem}\label{lr}
{\rm ($L^r$-$L^\infty$ estimate)} Let $r\ge 1$ be such that $\lambda_r>0$, with $1\leq r<2$ when $p>p_c$. If $u$ is a locally bounded solution of \eqref{eq}, non-negative in $B_{4\rho}(x_0)\times(0,t)\subset\Omega_T$, then there exists $\gamma>0$ depending on the data and $r$, s.t.
\begin{equation*}
\begin{aligned}
 \sup_{B_{\rho/2}(x_0)\times(t/2,t)}\,u \le &\gamma\Big[\fiint_{B_\rho(x_0)\times(0,t)}u^r(x,\tau)\,dx\,d\tau\Big]^\frac{ps}{\lambda_r}\Big(\frac{t}{\rho^{ps}}\Big)^{-\frac{N}{\lambda_r}}+ \\
&+ \gamma\Big(\frac{t}{\rho^{ps}}\Big)^\frac{1}{2-p}\max\Big\{1,\,\Big(\frac{t}{\rho^{ps}}\Big)^\frac{1-p}{2-p}{\rm Tail}\Big(u_+,x_0,\frac{\rho}{2},0,t\Big)^{p-1}\Big\}.
\end{aligned}
\end{equation*}
Moreover, $\gamma$ blows up when $\lambda_r$ vanishes.
\end{theorem}
\noindent
For ease of presentation we prefer to state Theorem \ref{lr} for locally bounded solutions that are locally non-negative. Nevertheless, these two requirements are redundant, see Remarks \ref{RMK-Bundedness}, \ref{RMK-non-negative} for a more detailed analysis.
\vskip2pt
\noindent
Our second result compares the supremum of the integral of $u$ in a ball, over a time interval, to the infimum of the integral of $u$ in a larger ball.

\begin{theorem}\label{l1}
{\rm ($L^1$-$L^1$ estimate)} Let $u$ be a solution of \eqref{eq}, non-negative in $B_{4\rho}(x_0)\times(0,t)\subset\Omega_T$. Then, there exists $\gamma>0$ depending on the data s.t.\
\begin{align*}  \sup_{0<\tau<t}\,\int_{B_\rho(x_0)}u(x,\tau)\,dx \le& \gamma\Big[\inf_{0<\tau<t}\,\int_{B_{2\rho}(x_0)}u(x,\tau)\,dx\Big] \\
&+ \gamma\Big(\frac{t}{\rho^{\lambda_1}}\Big)^\frac{1}{2-p}\max\Big\{1,\,\Big(\frac{t}{\rho^{ps}}\Big)^\frac{1-p}{2-p}{\rm Tail}\Big(u_-,x_0,\frac{\rho}{2},0,t\Big)^{p-1}\Big\} \\
&+ \gamma\Big(\frac{t}{\rho^{\lambda_1}}\Big)^\frac{1}{2-p}\max\Big\{1,\,\Big(\frac{t}{\rho^{ps}}\Big)^\frac{1-p}{2-p}{\rm Tail}\Big(u_-,x_0,\frac{\rho}{2},0,t\Big)^{p-1}\Big\}^\frac{p-1}{2-p}.
\end{align*}
\end{theorem}

\noindent
Note that Theorem \ref{l1} does not require that the solution $u$ is locally bounded, as it only involves integrals of $u$; moreover, it is valid for every $1<p<2$. From Theorems \ref{lr} and \ref{l1} above, for the restricted supercritical regime
\[\frac{2N}{N+s}<p<2\]
(which ensures $\lambda_1>0$), we derive the following further result that relates the supremum of $u$ in a cylinder to the infimum of the integral of $u$ over a larger time interval.

\begin{theorem}\label{linf}
{\rm ($L^1$-$L^\infty$ estimate)} Let $\lambda_1>0$, let $u$ be a locally bounded solution of \eqref{eq} non-negative in $B_{4\rho}(x_0)\times(0,t)\subset\Omega_T$. Then, there exists $\gamma>0$ depending on the data s.t.\
\begin{align*}
\sup_{B_{\rho/2}(x_0)\times(t/2,t)}\,u &\le \gamma\Big[\inf_{0<\tau<t}\,\int_{B_{2\rho}(x_0)}u(x,\tau)\,dx\Big]^\frac{ps}{\lambda_1}t^{-\frac{N}{\lambda_1}} \\
&+ \gamma\Big(\frac{t}{\rho^{ps}}\Big)^\frac{1}{2-p}\max\Big\{1,\,\Big(\frac{t}{\rho^{ps}}\Big)^\frac{1-p}{2-p}{\rm Tail}\Big(u_+,x_0,\frac{\rho}{2},0,t\Big)^{p-1}\Big\} \\
&+ \gamma\Big(\frac{t}{\rho^{ps}}\Big)^\frac{1}{2-p}\max\Big\{1,\,\Big(\frac{t}{\rho^{ps}}\Big)^\frac{1-p}{2-p}{\rm Tail}\Big(u_-,x_0,\frac{\rho}{2},0,t\Big)^{p-1}\Big\}^\frac{ps}{\lambda_1} \\
&+ \gamma\Big(\frac{t}{\rho^{ps}}\Big)^\frac{1}{2-p}\max\Big\{1,\,\Big(\frac{t}{\rho^{ps}}\Big)^\frac{1-p}{2-p}{\rm Tail}\Big(u_-,x_0,\frac{\rho}{2},0,t\Big)^{p-1}\Big\}^\frac{(p-1)ps}{(2-p)\lambda_1}.
\end{align*}
\end{theorem}

\noindent
The following estimate relates the $L^r$-norm of a solution $u$, spanning a time interval $(0,t)$, to its initial value at time $0$, provided the $L^r$-norm of the positive part $u_+$ concentrates in a ball.

\begin{theorem}\label{back}
{\rm (Backward $L^r$-$L^r$ estimate)} Let $u$ be a locally bounded solution of \eqref{eq}, non-negative in $B_{4\rho}(x_0)\times(0,t)\subset\Omega_T$. Then, there exists $\gamma>0$ depending on the data, s.t. \
\[\sup_{0<\tau<t}\,\int_{B_\rho(x_0)}u^r(x,\tau)\,dx \le \gamma \max\Big\{\sup_{0<\tau<t}\,\int_{B_\rho^c(x_0)}u_+^r(x,\tau)\,dx, \ \int_{B_{2\rho}(x_0)}u^r(x,0)\,dx+\Big(\frac{t^r}{\rho^{\lambda_r}}\Big)^\frac{1}{r}\Big\}.\]
\end{theorem}
\noindent The quantitative estimate of Theorem \ref{back} has to be read as void when the right hand side is unbounded. \vskip0.1cm 
\noindent Besides free solutions of \eqref{eq}, we also consider the associated Cauchy-Dirichlet problem \eqref{cdp}, under the assumptions of Section \ref{framing}. Our first result on problem \eqref{cdp} is a qualitative one, ensuring a finite extinction time for globally non-negative  solutions.

\begin{theorem}\label{fte}
{\rm (Extinction time)} Let $\Omega$ be a bounded open set, $u_0\in W^{s,p}_0(\Omega)\cap L^\infty(\R^N)$ s.t.\ $u_0\ge 0$ in $\Omega$, and $u\ge 0$ be a bounded weak solution of \eqref{cdp}. Then, there exists $T_*\in(0,\infty)$ s.t.\ $u(\cdot,t)=0$ in $\R^N$, for all $t\ge T_*$. In addition, there exists a constant $\gamma_*>0$ depending only on $\{N,p,s\}$ such that 
\begin{enumroman}
\item\label{fte1} if $1<p<p_c$, then \
\[T_* = \frac{\gamma_*}{C_1}\|u_0\|_{L^q(\Omega)}^{2-p}\,, \quad \quad q=\frac{N(2-p)}{ps};\]
\item\label{fte2} if $p_c\le p<2$, then \
\[T_* = \frac{\gamma_*}{C_1}\|u_0\|_{L^2(\Omega)}^{2-p}|\Omega|^\frac{\lambda_2}{2N}.\]
\end{enumroman}
\end{theorem}

\noindent
The quantitative counterpart of Theorem \ref{fte}, following from Theorems \ref{l1} and \ref{linf} above, provides us with an estimate of the decay rate (see Figure \ref{fig:FINALE}) as $t$ approaches the extinction time $T_*$, as follows.

\begin{theorem}\label{dec}
{\rm (Decay Rate)} Let $u\ge 0$ be a locally bounded local weak solution of \eqref{eq} in $\Omega_T= \Omega \times [0,T]$ and $T_*\in(0,T)$ an extinction time for $u$. Let us assume $B_{4\rho}(x_0)\times(0,t)\subset\Omega\times(0,T_*)$. Then, there exists a constant $\gamma>0$ depending only on the data s.t. \vskip0.1cm \noindent 
\begin{enumroman}
\item\label{dec1} for all $1<p<2$ the local mass decays as \ 
\[\int_{B_{\rho}(x_0)}u(x,t)\,dx \le \gamma\Big(\frac{T_*-t}{\rho^{\lambda_1}}\Big)^\frac{1}{2-p};\]
\item\label{dec2} if $\lambda_1>0$, then \ 
\[\sup_{B_{\rho/2}(x_0)}u\Big(\cdot,\frac{t+T_*}{2}\Big) \le \gamma\Big(\frac{T_*-t}{\rho^{ps}}\Big)^\frac{1}{2-p}\max\Big\{1,\,\Big(\frac{T_*-t}{\rho^{ps}}\Big)^\frac{1-p}{2-p}{\rm Tail}\Big(u,x_0,\frac{\rho}{2},t,T_*\Big)^{p-1}\Big\}.\]
\end{enumroman}
\end{theorem}

\noindent
Note that the $L^\infty$-bound in \ref{dec2} is a stronger bound than the $L^1$-bound of \ref{dec1} but the latter also holds for $p<2N/(N+s)$ and it does not suffer from the behavior of $u_+$ in $B_\rho(x_0) \setminus B_{\rho/2}(x_0)$.

\begin{figure}
     \centering
     \includegraphics[scale=0.3]{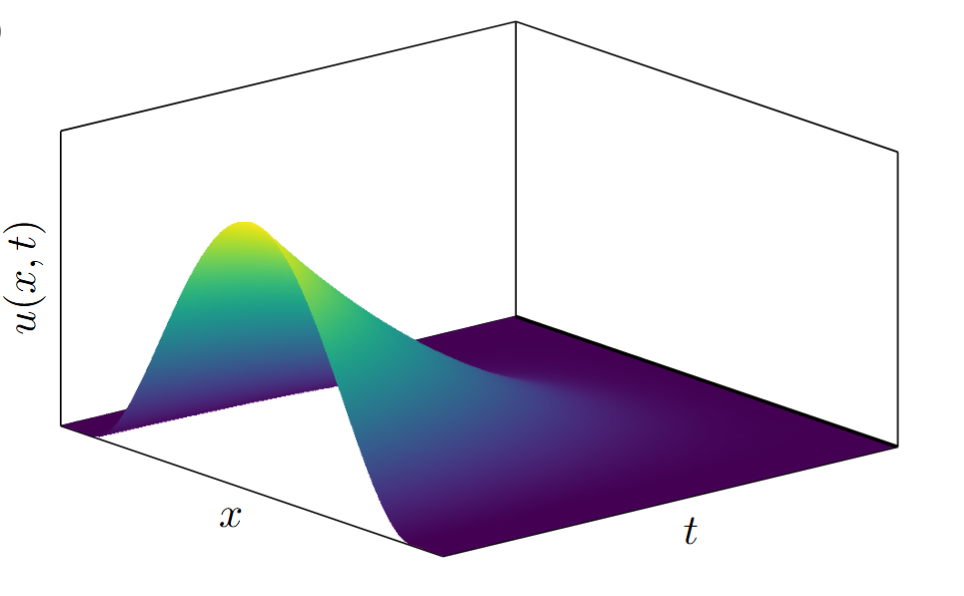}
     \caption{Illustration of the extinction decay for a local weak solution to \eqref{cdp}.}
     \label{fig:FINALE}
 \end{figure}

\subsection{Novelty and Significance} 
To the best of our knowledge, $L^1$-$L^1$ Harnack-type estimates are first found in \cite{DBH-sing}, with the aim of studying the existence of initial traces for very weak solutions to the singular $p$-Laplacian prototype. A couple of years later they were used in \cite{DiB-Kw} for both the cases of the $p$-Laplacian equation and the porous medium one, with the aim of evaluating the rate of extinction of non-negative solutions. The method of \cite{DBH-sing} for the $p$-Laplacian has been reported in \cite{DB}, Chap. VII for solutions to the prototype singular equation ($1<p<2$). A proof for $p$-Laplacian type equations with full quasilinear structure can be found first in the paper \cite{Annali} and then in the monograph \cite{DBGV-mono}. See also \cite{BoIaVa} and \cite{DBGV-mono} for a later treatment of the sub-critical case.
\vskip2pt
\noindent
In the case of fractional $p$-Laplacian diffusion, such Harnack-type estimates are new. For a different notion of solution and in the whole space, the author in \cite{Vaz1} has proved some global Harnack bounds. Existence theory for different definitions of solutions than ours can be found in  \cite{Mazon} and \cite{Vaz-CDP}, see also the book  \cite{Mazon-book}. There the notion of solution involves the existence of a Sobolev time-derivative, that we have carefully avoided here through the approximation procedures of Appendix \ref{app}. The main difference lies in considering a general measurable and bounded kernel as in $(K_1)$, $(K_2)$. The discontinuity of the field into the divergence term of \eqref{eq} results in a lack of regularity of the time derivative $u_t$, see \cite{LadSolUra} Sec.\ 13 of Ch. III (pages 224-233). The same lack of regularity in time pertains the possible weak solutions to the Cauchy-Dirichlet problem, whose analysis of the extinction phenomenon has been performed in \cite{AABP}, Theorem 5.1. There the authors point out a particular feature of the nature of fractional nonlinear diffusion, that is, it does not necessarily meet finite speed of propagation in $\R^N$ when $p>2$. On the other hand, similarly to the classical $p$-Laplacian, when $p>2$ local boundedness is inherent in the definition of local weak solution, see \cite{Strom1} (and later \cite{Brasco} for the prototype equation, see also \cite{AP}).
\vskip2pt
\noindent
Here, {\emph{en passant}}, we prove that the same is valid for local weak solutions in the super-critical range (see Remark \ref{RMK-Bundedness}), see also Theorem 1.1 of \cite{koreani} for a study on the best tail condition and \cite{Prasad-Tewary} in the case with no phase ($a=0$) for a whole satisfying picture in the special case of parabolic minima with a boundary datum. An analysis of the $L^1$-$L^{\infty}$ estimates with the finer tail condition of \cite{Kass} and \cite{Liao1} can be carried from ours with few effort. This is indeed the aim of a future project, where these estimate will be necessary for a different aim. Here we directly employ the definition \eqref{tail}, since we aim to the evaluation of the extinction decay of $u$ as power of time, i.e. $(T^*-t)^{1/(2-p)}$, disregarding of the optimality of the long-range effects. Concerning the study of the extinction of solutions, the works \cite{Vaz1}, \cite{Vazquez-Again} (see Section 8) investigate the decay in the whole space through an appropriate comparison principle. For a similar definition of solution, the authors in \cite{Collier-Hauer} study the decay rate towards extinction for an equation with double-nonlinearity in time. What diversifies the extinction rates of Theorem \ref{dec} is the concept of solution (see Definition \ref{lws}), and the circumvention of any comparison principle. Indeed, the formulation of the decay rate of Theorem \ref{dec} doesn't actually require $u$ to satisfy any Cauchy-Dirichlet conditions. 

\subsection{The Method} We rely on the energy estimates of \cite{L} to apply the iterative De Giorgi scheme to prove Theorem \ref{lr}. Then the $L^1$-$L^1$ Harnack-type estimates are found here, following the approach of \cite{DBGV-mono} within the refined technique of \cite{CH}, \cite{EIS}, but with a precise nonlocal method appropriately devised, see the proof of Lemma \ref{tst} for details. On the other hand, the $L^r$-$L^{\infty}$ bounds of Theorem \ref{lr} are more similar to the classic \cite{DB}, but again the tail terms impose an either/or alternative that is more in the spirit of Lemma 3.1 of \cite{L} (see also Theorem 1.8 of \cite{Kass} in the linear case with the finest tail condition). Finally, the backwards $L^r$- estimates cannot run the argument of \cite{DBGV-mono} and had to adapt the original idea of \cite{DBH-sing} with a proper condition on the long-range effects of the solution. The extinction phenomenon is reminiscent of the classic energy method (see for instance \cite{DB} Ch. VII section 2, or \cite{ADS}), while the lack of time-derivative called for a precise time-mollification in the spirit of \cite{KL}. 

\subsection{Future Applications} The $L^1$-$L^1$ estimate is very useful to transport uniformly measure information in time. It can be used, among other possibilities, to give a simpler proof of the H\"older continuity of solutions (see \cite{CV}) using the main tools developed in \cite{L}. At the same time $L^1$-$L^1$ Harnack-type inequalities are fundamental in the study of the initial traces, as in \cite{DBH-sing} (see \cite{DB} for a more comprehensive overview), while backward $L^r$ estimates as the ones of Theorem \ref{back} can be used to show that a precise local integrability of the initial datum $u_0$ (i.e., $u_0 \in L^r_{\rm loc}(\R^N)$) yields directly locally bounded and locally H\"older-continuous solutions by Ascoli-Arzelà's principle, see Remark 3.6 or \cite{DiB-Kw}. When an oscillation estimate is at stake (as in \cite{L}), it is again possible to study the decay rate to extinction, in a different local geometry than ours, see Corollary 1.4 of \cite{BDGLS} or \cite{Kim-Lee} for an approach with the extension trick of \cite{Caffarelli-Silvestre}.

\subsection{Structure of the paper} In Section \ref{sec2} we recall the basic definitions and some preliminary results, including the parabolic fractional embeddings, energy estimates and various iteration Lemmas; in Section \ref{sec3} we prove Theorem \ref{lr}; in Section \ref{sec4} we prove Theorems \ref{l1} and \ref{linf}; in Section \ref{sec5} we prove Theorem \ref{back}; in Section \ref{sec6} we prove Theorems \ref{fte} and \ref{dec}; finally, in Appendix \ref{app} we deal with the technical issue of time mollifications.

\subsection{Notation} Throughout the paper, for all $U\subset\R^N$ we will denote by $|U|$ the $N$-dimensional Lebesgue measure of $U$, $U^c=\R^N\setminus U$. By $B_\rho(x)$ we will denote the open ball of radius $\rho$ centered at $x$. Writings like $u\le v$ in $U$ will mean that $u(x)\le v(x)$ for a.e.\ $x\in U$. By $u_+$ (resp., $u_-$) we will denote the positive (resp., negative) part of $u$. By $\inf_U u$ (resp., $\sup_U u$) we will denote the essential infimum (resp., supremum) of $u$ in $U$. Most important, $\gamma$ will denote several positive constants, only depending on the data $N,p,s,\Omega,C_1,C_2$ of the problem, except when explicitly noted.

\section{Preliminaries}\label{sec2}

\noindent
This section includes some preliminary definitions which provide a rigorous framework for our problems, and some technical results that will be used in the proofs of our main results. We begin by recalling that, for any open set $\Omega\subseteq\R^N$ ($N\ge 2$), a function $u:\Omega\to\R$ is said to belong to the fractional Sobolev space $W^{s,p}(\Omega)$ $0<s<1$, $1<p<2$) if $u\in L^p(\Omega)$ and
\[\iint_{\Omega\times\Omega}\frac{|u(x)-u(y)|^p}{|x-y|^{N+ps}}\,dx\,dy < \infty.\]
Clearly, if $\Omega$ is unbounded, we say that $u\in W^{s,p}_{\rm loc}(\Omega)$ if $u\in W^{s,p}(\Omega')$ for all $\Omega'\Subset\Omega$. Also, if $\Omega$ is bounded, we say that $u\in W^{s,p}_0(\Omega)$ if $u\in W^{s,p}(\R^N)$ and $u=0$ a.e.\ in $\Omega^c$. We recall Sobolev's embedding for fractional order spaces: if $\Omega$ is bounded, then there exists $\gamma>0$ depending on $N,p,s$ s.t.\ for all $u\in W^{s,p}_0(\Omega)$
\beq\label{sob}
\Big[\int_\Omega |u(x)|^{p^*_s}\,dx\Big]^\frac{1}{p^*_s} \le \gamma\Big[\iint_{\R^N\times\R^N}\frac{|u(x)-u(y)|^p}{|x-y|^{N+ps}}\,dx\,dy\Big]^\frac{1}{p}, \ p^*_s = \frac{Np}{N-ps}
\eeq
(recall that $N>ps$). For evolutive equations, we will always identify $u(\cdot,t)$ with $u(t)$.
\vskip2pt
\noindent
Now let $\Omega\subset\R^N$ be an open set, $T\in(0,\infty]$. We define solutions of our equation as in \cite[Subsection 1.2.2]{L}:

\begin{definition}\label{lws}
We say that $u\in C(0,T,L^2_{\rm loc}(\Omega))\cap L^p_{\rm loc}(0,T,W^{s,p}_{\rm loc}(\Omega))$ is a (local, weak) solution of \eqref{eq} if for all $\Omega'\Subset\Omega$ and all $0<t_1<t_2<T$ the following conditions are satisfied:
\begin{enumroman}
\item\label{lws1} for all $\varphi\in W^{1,2}_{\rm loc}(0,T,L^2(\Omega'))\cap L^p_{\rm loc}(0,T,W^{s,p}_0(\Omega'))$
\begin{align*}
0 &= \int_{\Omega'}u(x,\tau)\varphi(x,\tau)\,dx\Big|_{t_1}^{t_2}-\int_{t_1}^{t_2}\int_{\Omega'}u(x,\tau)\varphi_\tau(x,\tau)\,dx\,d\tau \\
&+ \int_{t_1}^{t_2}\iint_{\R^N\times\R^N}|u(x,\tau)-u(y,\tau)|^{p-2}(u(x,\tau)-u(y,\tau))(\varphi(x,\tau)-\varphi(y,\tau))K(x,y,\tau)\,dx\,dy\,d\tau;
\end{align*}
\item\label{lws2} $\displaystyle\sup_{t_1<\tau<t_2}\,\int_{\R^N}\frac{|u(x,\tau)|^{p-1}}{(1+|x|)^{N+ps}}\,dx < \infty$.
\end{enumroman}
\end{definition}

\noindent
Solutions of the Cauchy-Dirichlet problem \eqref{cdp} are defined as follows:

\begin{definition}\label{scp}
Let $\Omega\subset\R^N$ be a bounded open set, $u_0\in W^{s,p}_0(\Omega)$. We say that $u\in C(0,T,L^2(\Omega))\cap L^p(0,T,W^{s,p}_0(\Omega))$ is a (weak) solution of \eqref{cdp} if
\begin{enumroman}
\item\label{scp1} for all $\varphi\in W^{1,2}(0,T,L^2(\Omega))\cap L^p(0,T,W^{s,p}_0(\Omega))$
\begin{align*}
0 &= \int_\Omega u(x,\tau)\varphi(x,\tau)\,dx\Big|_0^T-\int_0^T\int_\Omega u(x,\tau)\varphi_\tau(x,\tau)\,dx\,d\tau \\
&+ \int_0^T\iint_{\R^N\times\R^N}|u(x,\tau)-u(y,\tau)|^{p-2}(u(x,\tau)-u(y,\tau))(\varphi(x,\tau)-\varphi(y,\tau))K(x,y,\tau)\,dx\,dy\,d\tau;
\end{align*}
\item\label{scp2} $u(\cdot,0)=u_0$ in $\R^N$.
\end{enumroman}
\end{definition}

\noindent
Note that Definition \ref{scp} implies Definition \ref{lws}, as condition \ref{lws2} of the latter follows from $u(\cdot,\tau)\in W^{s,p}_0(\Omega)$. We keep the definition of nonlocal tails as in \eqref{tail}. We will next recall some basic properties of solutions, beginning from an embedding inequality from \cite[Proposition A.3]{L}:

\begin{proposition}\label{emb}
{\rm (Embedding)} Let $x_0\in\R^N$, $0<\rho_1<\rho_2<\rho$, $0<t_1<t_2$ be s.t.\ $B_\rho(x_0)\times(t_1,t_2)\subset\Omega_T$ and $u\in L^p(t_1,t_2,W^{s,p}(B_\rho(x_0)))\cap L^\infty(t_1,t_2,L^2(B_\rho(x_0)))$ be s.t.\ ${\rm supp}(u(\cdot,\tau))\subset B_{\rho_1}(x_0)$ for a.e.\ $\tau\in(t_1,t_2)$. Then, there exists $\gamma>0$ depending on $N,p,s$ s.t.\
\begin{align*}
& \int_{t_1}^{t_2}\int_{B_{\rho_2}(x_0)}|u(x,\tau)|^\frac{(N+2s)p}{N}\,dx\,d\tau \le \gamma\Big[\sup_{t_1<\tau<t_2}\,\fint_{B_{\rho_2}(x_0)}u^2(x,\tau)\,dx\Big]^\frac{ps}{N} \\
&\cdot \Big[\rho_2^{ps}\int_{t_1}^{t_2}\iint_{B_{\rho_2}(x_0)\times B_{\rho_2}(x_0)}\frac{|u(x,\tau)-u(y,\tau)|^p}{|x-y|^{N+ps}}\,dx\,dy\,d\tau+\Big(\frac{\rho_2}{\rho_2-\rho_1}\Big)^{N+ps}\int_{t_1}^{t_2}\int_{B_{\rho_2}(x_0)}|u(x,\tau)|^p\,dx\,d\tau\Big].
\end{align*}
\end{proposition}

\noindent
The following energy estimate from \cite[Corollary 2.1]{L} will be one of our primary tools:

\begin{proposition}\label{ee}
{\rm (Energy estimate)} Let $u$ be a solution of \eqref{eq}, $x_0\in\R^N$, $0<\rho_1<\rho_2$, $0<t_1<t_2<t$ be s.t.\ $B_{\rho_2}(x_0)\times(0,t)\subset \Omega_T$, $k>0$, $w=u-k$. Then, there exists $\gamma>0$ depending on the data s.t.\
\begin{align*}
& \sup_{t-t_1<\tau<t}\,\int_{B_{\rho_1}(x_0)}w_\pm^2(x,\tau)\,dx+\int_{t-t_1}^t\iint_{B_{\rho_1}(x_0)\times B_{\rho_1}(x_0)}\frac{|w_\pm(x,\tau)-w_\pm(y,\tau)|^p}{|x-y|^{N+ps}}\,dx\,dy\,d\tau \\
&\le \frac{\gamma}{t_2-t_1}\int_{t-t_2}^t\int_{B_{\rho_2}(x_0)}w_\pm^2(x,\tau)\,dx\,d\tau+\frac{\gamma\rho_2^{p-ps}}{(\rho_2-\rho_1)^p}\int_{t-t_2}^t\int_{B_{\rho_2}(x_0)}w_\pm^p(x,\tau)\,dx\,d\tau \\
&+ \frac{\gamma\rho_2^N}{(\rho_2-\rho_1)^{N+ps}}\Big[\int_{t-t_2}^t\int_{B_{\rho_2}(x_0)}w_\pm(x,\tau)\,dx\,d\tau\Big]\,{\rm Tail}(w_\pm,x_0,\rho_2,t-t_2,t)^{p-1}.
\end{align*}
\end{proposition}

\noindent
We also recall two useful technical results involving recursive relations. The first is \cite[Lemma 5.1]{DBGV-mono}:

\begin{lemma}\label{fc}
{\rm (Fast convergence)} Let $(Y_n)$ be a sequence of non-negative real numbers, $b,C>1$, $\eta>0$ s.t.\
\begin{enumroman}
\item\label{fc1} $Y_{n+1}\le Cb^nY_n^{1+\eta}$ for all $n\in\N$;
\item\label{fc2} $Y_0\le C^{-\frac{1}{\eta}}b^{-\frac{1}{\eta^2}}$.
\end{enumroman}
Then, $Y_n\to 0$ as $n\to\infty$.
\end{lemma}

\noindent
A somewhat dual property is the following \cite[Lemma 5.2]{DBGV-mono}:

\begin{lemma}\label{int}
{\rm (Interpolation)} Let $(Y_n)$ be a bounded sequence of non-negative real numbers, $b>1$, $C>0$, $\eta\in(0,1)$ s.t.\ $Y_n \le Cb^nY_{n+1}^{1-\eta}$ for all $n\in\N$. Then,
\[Y_0 \le \big(2Cb^\frac{1-\eta}{\eta}\big)^\frac{1}{\eta}.\]
\end{lemma}

\noindent
Finally, we recall some elementary inequalities from \cite[Lemma 2.7]{AABP}:

\begin{lemma}\label{alg}
Let $p\in(1,2)$, $a,b\ge 0$. Then:
\begin{enumroman}
\item\label{alg1} for all $q>1$ there exists $\gamma>0$ s.t.\
\[\Big|a^\frac{p+q-2}{p}-b^\frac{p+q-2}{p}\Big|^p \le \gamma|a-b|^{p-2}(a-b)(a^{q-1}-b^{q-1});\]
\item\label{alg2} for all $q\ge 2$ there exists $\gamma>0$ s.t.\
\[(a+b)^{q-2}|a-b|^p \le \gamma \Big|a^\frac{p+q-2}{p}-b^\frac{p+q-2}{p}\Big|^p.\]
\end{enumroman}
\end{lemma}

\noindent
We will also use the following weighted Young's inequality: for all $q>1$, there exists $\gamma>0$ s.t.\ for all $a,b\ge 0$, $\eps\in(0,1)$
\beq\label{young}
ab \le \eps a^q+\frac{\gamma}{\eps^{q-1}}b^{q'}.
\eeq
We note the following consequence of \eqref{young} (see \cite[eq.\ (2.4)]{L}): there exists $\gamma>0$ s.t.\ for all $a\ge b\ge 0$, $\eps\in(0,1)$
\beq\label{liao}
a^p-b^p \le \eps a^p+\frac{\gamma}{\eps^{p-1}}(a-b)^p.
\eeq

\section{$L^r$-$L^\infty$ estimate}\label{sec3}

\noindent
In this section we display the proof of Theorem \ref{lr}, which is divided into several steps. 
To ease notation we set, for $\sigma \in (0,1)$,  the perturbative term 
\[P_{\sigma} = \Big(\frac{t}{\rho^{ps}}\Big)^\frac{1}{2-p}\max\Big\{1,\,\Big(\frac{t}{\rho^{ps}}\Big)^\frac{1-p}{2-p}{\rm Tail}\Big(u_+,x_0,\sigma \rho,0,t\Big)^{p-1}\Big\}\,.\]
We first establish two preliminary lemmas, one for the supercritical regime $p>p_c$ and the other for the subcritical regime $p\le p_c$. As it will appear clear, in the first case the local boundedness is a consequence of the definition of local weak solution. We state our result for nonnegative solutions, in order to ease the reading, though minor standard modifications clarify that the result is still valid for signed super or sub-solutions, see Remark \ref{RMK-non-negative}.

\begin{lemma}\label{sup}
Let $u$ be a local weak solution of \eqref{eq}, non-negative in $B_{4\rho}(x_0)\times(0,t)\subset\Omega_T$. Let $p_c<p<2$, $r\in[1,2]$ s.t.\ $\lambda_r>0$, $\sigma\in(0,1)$, $t_0\in(0,t]$. Then, there exist $\alpha,\gamma>0$ depending on the data s.t.\
\[\sup_{B_{\sigma\rho}(x_0)\times(t-\sigma t_0,t)}\,u \le \frac{\gamma}{[\sigma(1-\sigma)]^\alpha}\Big[\sup_{B_\rho(x_0)\times(t-t_0,t)}\,u^\frac{ps(2-r)}{\lambda_2}\Big]\Big[\fiint_{B_\rho(x_0)\times(t-t_0,t)}u^r(x,\tau)\,dx\,d\tau\Big]^\frac{ps}{\lambda_2}\Big(\frac{t_0}{\rho^{ps}}\Big)^{-\frac{N}{\lambda_2}}+P_{\sigma}.\]
\end{lemma}
\begin{proof}
Let $k>0$ be a number to be determined later and perform a decreasing iteration by setting for all $n\in\N$
\[\rho_n = \rho\Big(\sigma+\frac{1-\sigma}{2^n}\Big), \ t_n = t_0\Big(\sigma+\frac{1-\sigma}{2^n}\Big), \ k_n = k\Big(1-\frac{1}{2^n}\Big).\]
In addition we set
\[\hat\rho_n = \frac{3\rho_n+\rho_{n+1}}{4}, \ \tilde\rho_n = \frac{\rho_n+\rho_{n+1}}{2}, \ \check\rho_n = \frac{\rho_n+3\rho_{n+1}}{4},\]
so we have a shrinking geometric configuration and increasing levels.
Indeed, if we define the balls
\[B_n = B_{\rho_n}(x_0), \ \hat B_n = B_{\hat\rho_n}(x_0), \ \tilde B_n = B_{\tilde\rho_n}(x_0), \ \check B_n = B_{\check\rho_n}(x_0),\]
and the cylinders $Q_n=B_n\times(t-t_n,t)$, then \[B_n\supset\hat B_n\supset\tilde B_n\supset\check B_n\supset B_{n+1}, \quad \text{and} \quad  Q_n\supset Q_{n+1},\] with initial cylinder $Q_0=B_\rho(x_0)\times(t-t_0,t)$ and limit cylinder $Q_\infty=B_{\sigma\rho}(x_0)\times(t-\sigma t_0,t)$.
\vskip4pt
\noindent
Next, we set $w_n=(u-k_n)_+$ and define the related energy
\[E_n = \sup_{t-t_{n+1}<\tau<t}\,\int_{\tilde B_n}w_{n+1}^2(x,\tau)\,dx+\int_{t-t_{n+1}}^t\iint_{\tilde B_n\times\tilde B_n}\frac{|w_{n+1}(x,\tau)-w_{n+1}(y,\tau)|^p}{|x-y|^{N+ps}}\,dx\,dy\,d\tau.\]
For any pair of exponents $1\le q_1\le q_2$ and each $n\in\N$ the following Chebychev-type inequality holds true:
\begin{align}\label{sup1}
\iint_{Q_n}w_{n+1}^{q_1}(x,\tau)\,dx\,d\tau &\le \int_{t-t_n}^t\int_{B_n\cap\{u>k_{n+1}\}}\frac{(u(x,\tau)-k_n)^{q_2}}{(k_{n+1}-k_n)^{q_2-q_1}}\,dx\,d\tau \\
\nonumber &\le \frac{\gamma 2^{(q_2-q_1)n}}{k^{q_2-q_1}}\,\iint_{Q_n}w_n^{q_2}(x,\tau)\,dx\,d\tau.
\end{align}
We estimate $E_n$ by applying Proposition \ref{ee} with radii $0<\tilde\rho_n<\rho_n<\rho$, times $0<t_{n+1}<t_n<t$, and level $k_{n+1}$:
\begin{align*}
E_n &\le \frac{\gamma}{t_n-t_{n+1}}\iint_{Q_n}w_{n+1}^2(x,\tau)\,dx\,d\tau+\frac{\gamma\rho_n^{p-ps}}{(\rho_n-\tilde\rho_n)^p}\iint_{Q_n}w_{n+1}^p(x,\tau)\,dx\,d\tau \\
&+ \frac{\gamma\rho_n^N}{(\rho_n-\tilde\rho_n)^{N+ps}}\Big[\iint_{Q_n}w_{n+1}(x,\tau)\,dx\,d\tau\Big]\,{\rm Tail}(w_{n+1},x_0,\rho_n,t-t_n,t)^{p-1} \\
&= I_1+I_2+I_3.
\end{align*}
We estimate separately each term above. For $I_1$ we use the definition of $(t_n)$ and \eqref{sup1} with $q_1=q_2=2$:
\begin{align*}
I_1 &\le \frac{\gamma 2^n}{(1-\sigma)t_0}\iint_{Q_n}w_{n+1}^2(x,\tau)\,dx\,d\tau \\
&\le \frac{\gamma 2^n}{(1-\sigma)t_0}\iint_{Q_n}w_n^2(x,\tau)\,dx\,d\tau.
\end{align*}
For $I_2$ we use the definition of $(\rho_n)$, $(\tilde\rho_n)$, and \eqref{sup1} with $q_1=p$, $q_2=2$:
\begin{align*}
I_2 &\le \frac{\gamma 2^{pn}}{(1-\sigma)^p\rho^{ps}}\iint_{Q_n}w_{n+1}^p(x,\tau)\,dx\,d\tau \\
&\le \frac{\gamma 2^{2n}}{(1-\sigma)^p\rho^{ps}k^{2-p}}\iint_{Q_n}w_n^2(x,\tau)\,dx\,d\tau.
\end{align*}
Finally, for $I_3$ we use \eqref{sup1} with $q_1=1$, $q_2=2$, and we estimate the tail term recalling that $w_{n+1}\le u_+$:
\begin{align*}
I_3 &\le \frac{\gamma 2^{(N+ps)n}}{(1-\sigma)^{N+ps}\rho^{ps}}\Big[\iint_{Q_n}w_{n+1}(x,\tau)\,dx\,d\tau\Big]\Big[\sup_{t-t_{n+1}<\tau<t}\,\rho_n^{ps}\int_{B_n^c}\frac{w_{n+1}^{p-1}(x,\tau)}{|x-x_0|^{N+ps}}\,dx\Big] \\
&\le \frac{\gamma 2^{(N+ps+1)n}}{(1-\sigma)^{N+ps}\rho^{ps}k}\Big[\iint_{Q_n}w_n^2(x,\tau)\,dx\,d\tau\Big]\Big(\frac{\sigma\rho}{\rho_n}\Big)^{-ps}\Big[\sup_{0<\tau<t}\,(\sigma \rho)^{ps}\int_{B_{\sigma \rho}^c(x_0)}\frac{u_+^{p-1}(x,\tau)}{|x-x_0|^{N+ps}}\,dx\Big] \\
&\le \frac{\gamma 2^{(N+ps+1)n}}{(1-\sigma)^{N+ps}(\sigma\rho)^{ps}k}\Big[\iint_{Q_n}w_n^2(x,\tau)\,dx\,d\tau\Big]{\rm Tail}\Big(u_+,x_0,\sigma \rho,0,t\Big)^{p-1}.
\end{align*}
Plugging such estimates into the previous inequality, we find $\alpha,\beta,\gamma>1$ depending on the data s.t.\
\[E_n \le \frac{\gamma 2^{\beta n}}{[\sigma(1-\sigma)]^\alpha}\Big[\frac{1}{t_0}+\frac{1}{\rho^{ps}k^{2-p}}+\frac{1}{\rho^{ps}k}{\rm Tail}\Big(u_+,x_0,\sigma \rho,0,t\Big)^{p-1}\Big]\iint_{Q_n}w_n^2(x,\tau)\,dx\,d\tau.\]
Now we introduce our first condition on $k$:
\beq\label{sup2}
k > P_{\sigma}.
\eeq
So, from the previous inequality and \eqref{sup2} we have for all $n\in\N$
\beq\label{sup3}
E_n \le \frac{\gamma 2^{\beta n}}{[\sigma(1-\sigma)]^\alpha t_0}\iint_{Q_n}w_n^2(x,\tau)\,dx\,d\tau.
\eeq
We aim at an iterative estimate for the sequence
\[Y_n = \fiint_{Q_n}w_n^2(x,\tau)\,dx\,d\tau.\]
Set for all $n\in\N$
\[M_n = \Big|Q_n\cap\{u>k_n\}\Big|^\frac{\lambda_2}{p(N+2s)},\]
which has a positive exponent due to the choice of $p$. To proceed, we fix a cutoff function $\xi_n\in C^\infty_c(\check B_n)$ s.t.\ $\xi_n=1$ in $B_{n+1}$ and satisfies in $\R^N$
\[0 \le \xi_n \le 1, \ |\nabla\xi_n| \le \frac{\gamma 2^n}{(1-\sigma)\rho}.\]
We apply Proposition \ref{emb} to the function $w_{n+1}\xi_n$, with radii $0<\check\rho_n<\tilde\rho_n$ and times $0<t-t_{n+1}<t$,
\begin{align*}
& \int_{t-t_{n+1}}^t\int_{\tilde B_n}|w_{n+1}(x,\tau)\xi_n(x)|^\frac{p(N+2s)}{N}\,dx\,d\tau \le \gamma\Big[\sup_{t-t_{n+1}<\tau<t}\,\fint_{\tilde B_n}|w_{n+1}(x,\tau)\xi_n(x)|^2\,dx\Big]^\frac{ps}{N} \\
&\cdot \Big[\tilde\rho_n^{ps}\int_{t-t_{n+1}}^t\iint_{\tilde B_n\times\tilde B_n}\frac{|w_{n+1}(x,\tau)\xi_n(x)-w_{n+1}(y,\tau)\xi_n(y)|^p}{|x-y|^{N+ps}}\,dx\,dy\,d\tau \\
&+ \Big(\frac{\tilde\rho_n}{\tilde\rho_n-\check\rho_n}\Big)^{N+ps}\int_{t-t_{n+1}}^t\int_{\tilde B_n}|w_{n+1}(x,\tau)\xi_n(x)|^p\,dx\,d\tau\Big].
\end{align*}
By H\"older's inequality and the previous embedding estimate we have
\begin{align}\label{sup4}
\iint_{Q_{n+1}}w_{n+1}^2(x,\tau)\,dx\,d\tau &\le \Big[\iint_{Q_{n+1}}w_{n+1}(x,\tau)^\frac{p(N+2s)}{N}\,dx\,d\tau\Big]^\frac{2N}{p(N+2s)}M_{n+1} \\
\nonumber &\le \Big[\int_{t-t_{n+1}}^t\int_{\tilde B_n}|w_{n+1}(x,\tau)\xi_n(x)|^\frac{p(N+2s)}{N}\,dx\,d\tau\Big]^\frac{2N}{p(N+2s)}M_{n+1} \\
\nonumber &\le \gamma\Big[\sup_{t-t_{n+1}<\tau<t}\,\fint_{\tilde B_n}|w_{n+1}(x,\tau)\xi_n(x)|^2\,dx\Big]^\frac{2s}{N+2s} \\
\nonumber &\cdot \Big[\rho^{ps}\int_{t-t_{n+1}}^t\iint_{\tilde B_n\times\tilde B_n}\frac{|w_{n+1}(x,\tau)\xi_n(x)-w_{n+1}(y,\tau)\xi_n(y)|^p}{|x-y|^{N+ps}}\,dx\,dy\,d\tau \\
\nonumber &+ \frac{2^{(N+ps)n}}{(1-\sigma)^{N+ps}}\int_{t-t_{n+1}}^t\int_{\tilde B_n}|w_{n+1}(x,\tau)\xi_n(x)|^p\,dx\,d\tau\Big]^\frac{2N}{p(N+2s)}M_{n+1}.
\end{align}
To estimate the 'gradient' term above, we split the integrand and use the properties of $\xi_n$:
\begin{align*}
& \int_{t-t_{n+1}}^t\iint_{\tilde B_n\times\tilde B_n}\frac{|w_{n+1}(x,\tau)\xi_n(x)-w_{n+1}(y,\tau)\xi_n(y)|^p}{|x-y|^{N+ps}}\,dx\,dy\,d\tau \\
&\le \gamma\int_{t-t_{n+1}}^t\iint_{\tilde B_n\times\tilde B_n}\frac{|w_{n+1}(x,\tau)-w_{n+1}(y,\tau)|^p}{|x-y|^{N+ps}}\xi_n^p(y)\,dx\,dy\,d\tau \\
&+ \gamma\int_{t-t_{n+1}}^t\iint_{\tilde B_n\times\tilde B_n}w_{n+1}^p(x,\tau)\frac{|\xi_n(x)-\xi_n(y)|^p}{|x-y|^{N+ps}}\,dx\,dy\,d\tau \\
&\le \gamma\int_{t-t_{n+1}}^t\iint_{\tilde B_n\times\tilde B_n}\frac{|w_{n+1}(x,\tau)-w_{n+1}(y,\tau)|^p}{|x-y|^{N+ps}}\,dx\,dy\,d\tau \\
&+ \frac{\gamma 2^{pn}}{(1-\sigma)^p\rho^p}\int_{t-t_{n+1}}^t\iint_{\tilde B_n\times\tilde B_n}\frac{w_{n+1}^p(x,\tau)}{|x-y|^{N+ps-p}}\,dx\,dy\,d\tau \\
&\le \gamma E_n+\frac{\gamma 2^{pn}}{(1-\sigma)^p\rho^p}\Big[\int_{t-t_{n+1}}^t\int_{\tilde B_n}w_{n+1}^p(x,\tau)\,dx\,d\tau\Big]\Big[\sup_{x\in\tilde B_n}\,\int_{\tilde B_n}\frac{dy}{|x-y|^{N+ps-p}}\Big] \\
&\le \gamma E_n+\frac{\gamma 2^{pn}}{(1-\sigma)^p\rho^{ps}}\int_{t-t_{n+1}}^t\int_{\tilde B_n}w_{n+1}^p(x,\tau)\,dx\,d\tau.
\end{align*}
Plug the last inequality into \eqref{sup4}, taking into account the mean value and the definition of $E_n$:
\begin{align*}
\iint_{Q_{n+1}}w_{n+1}^2(x,\tau)\,dx\,d\tau &\le \gamma\rho^{-\frac{2Ns}{N+2s}}\Big[\sup_{t-t_{n+1}<\tau<t}\,\int_{\tilde B_n}w_{n+1}^2(x,\tau)\,dx\Big]^\frac{2s}{N+2s} \\
&\cdot \Big[\rho^{ps}E_n+\frac{2^{\beta n}}{(1-\sigma)^\alpha}\int_{t-t_{n+1}}^t\int_{\tilde B_n}w_{n+1}^p(x,\tau)\,dx\,d\tau\Big]^\frac{2N}{p(N+2s)}M_{n+1} \\
&\le \frac{\gamma 2^{\beta n}}{(1-\sigma)^\alpha}\rho^{-\frac{2Ns}{N+2s}}E_n^\frac{2s}{N+2s}\Big[\rho^{ps}E_n+\int_{t-t_{n+1}}^t\int_{\tilde B_n}w_{n+1}^p(x,\tau)\,dx\,d\tau\Big]^\frac{2N}{p(N+2s)}M_{n+1},
\end{align*}
with different $\alpha,\beta,\gamma>1$ depending on the data. We use \eqref{sup3} on $E_n$, \eqref{sup1} with $q_1=p$ and $q_2=2$, and again \eqref{sup2} to estimate $k$:
\begin{align*}
\iint_{Q_{n+1}}w_{n+1}^2(x,\tau)\,dx\,d\tau &\le \frac{\gamma 2^{\beta n}}{[\sigma(1-\sigma)]^\alpha}\rho^{-\frac{2Ns}{N+2s}}t_0^{-\frac{2s}{N+2s}}\Big[\iint_{Q_n}w_n^2(x,\tau)\,dx\,d\tau\Big]^\frac{2s}{N+2s} \\
&\cdot \Big[\frac{\rho^{ps}}{t_0}\iint_{Q_n}w_n^2(x,\tau)\,dx\,d\tau+\iint_{Q_n}w_{n+1}^p(x,\tau)\,dx\,d\tau\Big]^\frac{2N}{p(N+2s)}M_{n+1} \\
&\le \frac{\gamma 2^{\beta n}}{[\sigma(1-\sigma)]^\alpha}\rho^{-\frac{2Ns}{N+2s}}t_0^{-\frac{2s}{N+2s}}\Big(\frac{\rho^{ps}}{t_0}+\frac{1}{k^{2-p}}\Big)^\frac{2N}{p(N+2s)}\Big[\iint_{Q_n}w_n^2(x,\tau)\,dx\,d\tau\Big]^\frac{2(N+ps)}{p(N+2s)}M_{n+1} \\
&\le \frac{\gamma 2^{\beta n}}{[\sigma(1-\sigma)]^\alpha}t_0^{-\frac{2(N+ps)}{p(N+2s)}}\Big[\iint_{Q_n}w_n^2(x,\tau)\,dx\,d\tau\Big]^\frac{2(N+ps)}{p(N+2s)}M_{n+1}.
\end{align*}
The inequality above is 'almost' the desired iterative estimate, due to the presence of the measure term $M_{n+1}$. The latter can be estimated as follows, via a version of \eqref{sup1} with exponents $q_1=0$, $q_2=2$:
\begin{align*}
\Big|Q_{n+1}\cap\{u>k_{n+1}\}\Big| &\le \int_{t-t_{n+1}}^t\int_{B_n\cap\{u>k_{n+1}\}}\frac{(u(x,\tau)-k_n)^2}{(k_{n+1}-k_n)^2}\,dx\,d\tau \\
&\le \frac{2^{2n+2}}{k^2}\iint_{Q_n}w_n^2(x,\tau)\,dx\,d\tau.
\end{align*}
Combining the last two inequalities, and recalling that $|Q_n|$, $|Q_{n+1}|$ are comparable, we get the iterative estimate
\begin{align*}
\fiint_{Q_{n+1}}w_{n+1}^2(x,\tau)\,dx\,d\tau &\le \frac{1}{|Q_{n+1}|}\frac{\gamma 2^{\beta n}}{[\sigma(1-\sigma)]^\alpha}t_0^{-\frac{2(N+ps)}{p(N+2s)}}k^{-\frac{2\lambda_2}{p(N+2s)}}\Big[\iint_{Q_n}w_n^2(x,\tau)\,dx\,d\tau\Big]^{1+\frac{2s}{N+2s}} \\
&\le \frac{\gamma 2^{\beta n}}{[\sigma(1-\sigma)]^\alpha}\Big(\frac{t_0}{\rho^{ps}}\Big)^{-\frac{2N}{p(N+2s)}}k^{-\frac{2\lambda_2}{p(N+2s)}}\Big[\fiint_{Q_n}w_n^2(x,\tau)\,dx\,d\tau\Big]^{1+\frac{2s}{N+2s}}.
\end{align*}
Our next step consists in applying Lemma \ref{fc} to the sequence $(Y_n)$. Set
\[C = \frac{\gamma}{[\sigma(1-\sigma)]^\alpha}\Big(\frac{t_0}{\rho^{ps}}\Big)^{-\frac{2N}{p(N+2s)}}k^{-\frac{2\lambda_2}{p(N+2s)}}, \ b = 2^\beta, \eta = \frac{2s}{N+2s}.\]
Clearly $C>0$, $b>1$, and $\eta\in(0,1)$ are independent of $n$. Our iterative estimate then rephrases as
\beq\label{sup5}
Y_{n+1} \le Cb^nY_n^{1+\eta}.
\eeq
Non we give a precise definition of $k$ by setting
\[k = \Big[\frac{\gamma 2^\frac{\beta(N+2s)}{2s}}{[\sigma(1-\sigma)]^\alpha}\Big]^\frac{p(N+2s)}{2\lambda_2}\Big[\fiint_{Q_0}u^2(x,\tau)\,dx\,d\tau\Big]^\frac{ps}{\lambda_2}\Big(\frac{t_0}{\rho^{ps}}\Big)^{-\frac{N}{\lambda_2}}+P_{\sigma}.\]
Such choice clearly complies with \eqref{sup2}. In addition, recalling that $u\ge 0$ in $Q_0$, we have
\begin{align}\label{sup6}
Y_0 &= \fiint_{Q_0}u^2(x,\tau)\,dx\,d\tau \\
\nonumber &\le \Big[\frac{\gamma}{[\sigma(1-\sigma)]^\alpha}\Big(\frac{t_0}{\rho^{ps}}\Big)^{-\frac{2N}{p(N+2s)}}k^{-\frac{2\lambda_2}{p(N+2s)}}\Big]^{-\frac{N+2s}{2s}} 2^{-\beta(\frac{N+2s}{2s})^2} \\
\nonumber &= C^{-\frac{1}{\eta}}b^{-\frac{1}{\eta^2}}.
\end{align}
Due to \eqref{sup5} and \eqref{sup6} we may apply Lemma \ref{fc}, yielding $Y_n\to 0$ as $n\to\infty$. Passing to the limit we have
\[\fiint_{Q_\infty}(u(x,\tau)-k)_+^2\,dx\,d\tau = 0,\]
implying $u\le k$ almost everywhere in $Q_\infty$. Thus, recalling our choice of $k$, we have for convenient $\alpha,\gamma>1$ depending on the data
\begin{equation} \label{bounded-solutions}
\sup_{Q_\infty}\,u \le \frac{\gamma}{[\sigma(1-\sigma)]^\alpha}\Big[\fiint_{Q_0}u^2(x,\tau)\,dx\,d\tau\Big]^\frac{ps}{\lambda_2}\Big(\frac{t_0}{\rho^{ps}}\Big)^{-\frac{N}{\lambda_2}}+P_{\sigma}\,. \end{equation}
Now, we use the local boundedness implied by \eqref{bounded-solutions} plus our assumption on $r$, that is equivalent to
\[\frac{N(2-p)}{ps} < r <2,\]
which is an admissible range since $p>p_c$, to  obtain
\begin{equation*}
\sup_{Q_\infty}\,u \le \frac{\gamma}{[\sigma(1-\sigma)]^\alpha}\Big[\sup_{Q_0}\,u^\frac{ps(2-r)}{\lambda_2}\Big]\Big[\fiint_{Q_0}u^r(x,\tau)\,dx\,d\tau\Big]^\frac{ps}{\lambda_2}\Big(\frac{t_0}{\rho^{ps}}\Big)^{-\frac{N}{\lambda_2}}+P_{\sigma}\,.
\end{equation*}
Recalling the definitions of $Q_\infty$ and $Q_0$, respectively, we conclude.
\end{proof}

\begin{remark}\label{RMK-Bundedness} Formula \eqref{bounded-solutions} alone shows that local weak solutions are locally bounded in the range $p_c<p\leq 2$. We observe that if $r>2$, then it is possible to bound the supremum of $u$ with its $L^r$ norm by a simple use of H\"older inequality.
    
\end{remark}
\noindent
We show a similar Lemma for the sub-critical regime, this time assuming $u$ to be locally bounded. 

\begin{lemma}\label{sub}
Let $u$ be a locally bounded local weak solution of \eqref{eq}, non-negative in $B_{4\rho}(x_0)\times(0,t)\subset\Omega_T$. Let $1<p\le p_c$  and $\lambda_r>0$. Let us fix $\sigma\in(0,1)$ and $t_0\in(0,t]$. Then, there exist $\alpha,\gamma>0$ depending on the data and $r$, s.t.\
\begin{align*}
\sup_{B_{\sigma\rho}(x_0)\times(t-\sigma t_0,t)}\,u &\le \frac{\gamma}{[\sigma(1-\sigma)]^\alpha}\Big[\sup_{B_\rho(x_0)\times(t-t_0,t)}u^\frac{Nr-p(N+2s)}{(r-2)(N+ps)}\Big]\Big[\fiint_{B_\rho(x_0)\times(t-t_0,t)}u^r(x,\tau)\,dx\,d\tau\Big]^\frac{ps}{(r-2)(N+ps)} \\
&\cdot \Big(\frac{t_0}{\rho^{ps}}\Big)^{-\frac{N}{(r-2)(N+ps)}}+P_{\sigma}.
\end{align*}
Moreover, we have that $\gamma$, $\alpha \to +\infty$ as $r \to 2$, while both constants remain stable when $r \to +\infty$.
\end{lemma}
\begin{proof}
First note that the assumption $\lambda_r>0$, along with $p\le p_c$, implies
\[r > \frac{N(2-p)}{ps} \ge 2.\]
Let $k>0$ be a number to be determined later and define sequences of cylinders and levels, functions $w_n$ and energies $E_n$, as in Lemma \ref{sup}. Arguing as in Lemma \ref{sup}, applying Proposition \ref{ee} and \eqref{sup1} with several pairs of exponents (here we use $r>2$), we find
\[E_n \le \frac{\gamma 2^{\beta n}}{[\sigma(1-\sigma)]^\alpha}\Big[\frac{1}{t_0k^{r-2}}+\frac{1}{\rho^{ps}k^{r-p}}+\frac{1}{\rho^{ps}k^{r-1}} {\rm Tail}\Big(u_+,x_0,\sigma \rho,0,t\Big)^{p-1}\Big]\iint_{Q_n}w_n^r(x,\tau)\,dx\,d\tau,\]
with $\alpha,\beta,\gamma>0$ depending on the data and $r$. Again if we assume
\beq\label{sub1}
k >P_{\sigma},
\eeq
then the $n$-th energy is homogeneously controlled by
\beq\label{sub2}
E_n \le \frac{\gamma 2^{\beta n}}{[\sigma(1-\sigma)]^\alpha t_0k^{r-2}}\iint_{Q_n}w_n^r(x,\tau)\,dx\,d\tau.
\eeq
This time we will derive an iterative estimate for the following sequence (with no mean value):
\[Y_n = \iint_{Q_n}w_n^r(x,\tau)\,dx\,d\tau.\]
Define the cutoff function $\xi_n$ as in Lemma \ref{sup}. Then, apply Proposition \ref{emb} to the function $w_{n+1}\xi_n$, the properties of $\xi_n$, \eqref{sub1}, and \eqref{sub2}:
\begin{align*}
 \int_{t-t_{n+1}}^t\int_{\tilde B_n}|w_{n+1}(x,\tau)&\xi_n(x)|^\frac{p(N+2s)}{N}\,dx\,d\tau \le \gamma\Big[\sup_{t-t_{n+1}<\tau<t}\,\fint_{\tilde B_n}|w_{n+1}(x,\tau)\xi_n(x)|^2\,dx\Big]^\frac{ps}{N} \\
&\cdot \Big[\tilde\rho_n^{ps}\int_{t-t_{n+1}}^t\iint_{\tilde B_n\times\tilde B_n}\frac{|w_{n+1}(x,\tau)\xi_n(x)-w_{n+1}(x,\tau)\xi_n(y)|^p}{|x-y|^{N+ps}}\,dx\,dy\,d\tau \\
&+ \Big(\frac{\tilde\rho_n}{\tilde\rho_n-\check\rho_n}\Big)^{N+ps}\int_{t-t_{n+1}}^t\int_{\tilde B_n}|w_{n+1}(x,\tau)\xi_n(x)|^p\,dx\,d\tau\Big] \\
&\le \frac{\gamma}{\tilde\rho_n^{ps}}E_n^\frac{ps}{N}\Big[\tilde\rho_n^{ps}E_n+\Big(\frac{\tilde\rho_n}{\tilde\rho_n-\check\rho_n}\Big)^{N+ps}\iint_{Q_n}w_{n+1}^p(x,\tau)\,dx\,d\tau\Big] \\
&\le \frac{\gamma 2^{psn}}{[\sigma(1-\sigma)]^{ps}\rho^{ps}}E_n^\frac{ps}{N}\Big[\rho^{ps}E_n+\frac{2^{(N+ps+r-p)n}}{[\sigma(1-\sigma)]^{N+ps}k^{r-p}}\iint_{Q_n}w_n^r(x,\tau)\,dx\,d\tau\Big] \\
&\le \frac{\gamma 2^{\beta n}}{[\sigma(1-\sigma)]^\alpha}\Big[\frac{1}{t_0k^{r-2}}\iint_{Q_n}w_n^r(x,\tau)\,dx\,d\tau\Big]^\frac{N+ps}{N},
\end{align*}
again with $\alpha,\beta,\gamma>0$ depending on the data and $r$. Further, noting that $p(N+2s)/N\le 2\le r$, recalling that $u$ is assumed locally bounded, and that $\xi_n=1$ in $B_{n+1}$, we get
\begin{align*}
\iint_{Q_{n+1}}w_{n+1}^r(x,\tau)\,dx\,d\tau &\le \Big[\int_{t-t_{n+1}}^t\int_{\tilde B_n}|w_{n+1}(x,\tau)\xi_n(x)|^\frac{p(N+2s)}{N}\,dx\,d\tau\Big]\Big[\sup_{Q_0}\,u\Big]^\frac{Nr-p(N+2s)}{N} \\
&\le \frac{\gamma 2^{\beta n}}{[\sigma(1-\sigma)]^\alpha}\Big[\frac{1}{t_0k^{r-2}}\iint_{Q_n}w_n^r(x,\tau)\,dx\,d\tau\Big]^\frac{N+ps}{N}\Big[\sup_{Q_0}\,u\Big]^\frac{Nr-p(N+2s)}{N},
\end{align*}
with the supremum of $u$ in $Q_0$ term replacing the measure multiplier of the corresponding estimate of Lemma \ref{sup} (here the assumption that $u$ is locally bounded is essential). We will now apply Lemma \ref{fc} to $(Y_n)$, this time with
\[C = \frac{\gamma}{[\sigma(1-\sigma)]^\alpha}\Big(\frac{1}{t_0k^{r-2}}\Big)^\frac{N+ps}{N}\Big[\sup_{Q_0}\,u\Big]^\frac{Nr-p(N+2s)}{N}, \ b = 2^\beta, \ \eta = \frac{ps}{N}.\]
Note that $C,b>1$ and $\eta\in(0,1)$ are independent of $n$. The previous estimate then rephrases as
\beq\label{sub3}
Y_{n+1} \le Cb^nY_n^{1+\eta}.
\eeq
Define $k$ by setting
\[k = \Big[\frac{\gamma 2^\frac{\beta N}{ps}}{[\sigma(1-\sigma)]^\alpha}\Big]^\frac{N}{(r-2)(N+ps)}\Big[\iint_{Q_0}u^r(x,\tau)\,dx\,d\tau\Big]^\frac{ps}{(r-2)(N+ps)}\Big[\sup_{Q_0}\,u\Big]^\frac{Nr-p(N+2s)}{(r-2)(N+ps)}t_0^{-\frac{1}{r-2}}+P_{\sigma},\]
complying in particular with \eqref{sub1}. Then we have
\begin{align}\label{sub4}
Y_0 &= \iint_{Q_0}u^r(x,\tau)\,dx\,d\tau \\
\nonumber &\le \Big[\frac{\gamma}{[\sigma(1-\sigma)]^\alpha}\Big(\frac{1}{t_0k^{r-2}}\Big)^\frac{N+ps}{N}\sup_{Q_0}\,u^\frac{Nr-p(N+2s)}{N}\Big]^{-\frac{N}{ps}}2^{-\beta(\frac{N}{ps})^2} \\
\nonumber &= C^{-\frac{1}{\eta}}b^{-\frac{1}{\eta^2}}.
\end{align}
Due to \eqref{sub3} and \eqref{sub4} we may apply Lemma \ref{fc} and find that $Y_n\to 0$ as $n\to\infty$. Passing to the limit, we get
\[\iint_{Q_\infty}(u(x,\tau)-k)_+^r\,dx\,d\tau = 0,\]
hence $u\le k$ a.e.\ in $Q_\infty$. Therefore, by the definition of $k$ we have
\begin{align*}
\sup_{Q_\infty}\,u &\le \frac{\gamma}{[\sigma(1-\sigma)]^\alpha}\Big[\sup_{Q_0}\,u^\frac{Nr-p(N+2s)}{(r-2)(N+ps)}\Big]\Big[\iint_{Q_0}u^r(x,\tau)\,dx\,d\tau\Big]^\frac{ps}{(r-2)(N+ps)}t_0^{-\frac{1}{r-2}}+P_{\sigma} \\
&\le \frac{\gamma}{[\sigma(1-\sigma)]^\alpha}\Big[\sup_{Q_0}\,u^\frac{Nr-p(N+2s)}{(r-2)(N+ps)}\Big]\Big[\fiint_{Q_0}u^r(x,\tau)\,dx\,d\tau\Big]^\frac{ps}{(r-2)(N+ps)}\Big(\frac{t_0}{\rho^{ps}}\Big)^{-\frac{N}{(r-2)(N+ps)}}+P_{\sigma},
\end{align*}
for some $\alpha,\gamma>0$ depending on the data and $r$. This concludes the proof.
\end{proof}

\noindent
We can now complete the proof of the main result:
\vskip4pt
\noindent
{\em Proof of Theorem \ref{lr} (conclusion).} The second part of the proof is based on a further iterative argument, with increasing cylinders. Fix $\theta\in(0,1)$ (to be determined later) and set $\rho'_0 = \theta\rho$, $t'_0 = \theta t$. Also, set for all $n\in\N$
\[\rho'_n = \rho\Big[\theta+(1-\theta)\sum_{i=1}^n\frac{1}{2^i}\Big], \ t'_n = t\Big[\theta+(1-\theta)\sum_{i=1}^n\frac{1}{2^i}\Big],\]
so that both $(\rho'_n)$ and $(t'_n)$ are increasing and $\rho'_n\to\rho$, $t'_n\to t$ as $n\to\infty$. Also, for all $n\in\N$ we set $Q'_n=B_{\rho'_n}(x_0)\times(t-t'_n,t)$ so $Q'_n\subset Q'_{n+1}$ and the limit cylinder is $Q'_\infty=B_\rho(x_0)\times(0,t)$. Now recall that $u\ge 0$ in $Q_\infty$ and $u$ is locally bounded, so we may define a bounded sequence of positive numbers by setting for all $n\in\N$
\[S_n = \sup_{Q'_n}\,u,\]
We are going to prove an iterative estimate on the sequence $(S_n)$, by considering two cases:
\begin{itemize}[leftmargin=1cm]
\item[$(a)$] If $p_c<p<2$, then recall that $r<2$. We apply Lemma \ref{sup} with radius $\rho'_{n+1}$, time $t_0=t'_{n+1}$, and
\[\sigma = \frac{\theta+(1-\theta)\sum_{i=1}^n 2^{-i}}{\theta+(1-\theta)\sum_{i=1}^{n+1}2^{-i}} \in (0,1),\]
so that $\sigma\rho'_{n+1}=\rho'_n$, $\sigma t'_{n+1}=t'_n$. We get for some $\gamma,\alpha>0$ depending on the data and $\theta$
\begin{align}\label{lr1}
\sup_{Q'_n}\,u &\le \frac{\gamma}{[\sigma(1-\sigma)]^\alpha}\Big[\sup_{Q'_{n+1}}\,u^\frac{ps(2-r)}{\lambda_2}\Big]\Big[\fiint_{Q'_{n+1}}u^r(x,\tau)\,dx\,d\tau\Big]^\frac{ps}{\lambda_2}\Big(\frac{t'_{n+1}}{(\rho'_{n+1})^{ps}}\Big)^{-\frac{N}{\lambda_2}} \\
\nonumber &+ \Big(\frac{t}{(\rho'_{n+1})^{ps}}\Big)^\frac{1}{2-p}\max\Big\{1,\,\Big(\frac{t}{(\rho'_{n+1})^{ps}}\Big)^\frac{1-p}{2-p}{\rm Tail}\Big(u_+,x_0,\rho'_n,0,t\Big)^{p-1}\Big\} \\
\nonumber &= J_1+J_2
\end{align}
Now we estimate $J_1$ and $J_2$, separately. First note that
\[\sigma^\alpha(1-\sigma)^\alpha \ge \frac{(1-\theta)^\alpha}{2^{\alpha(n+2)}}.\]
Also, we have for some $\gamma>0$ independent of $n$ the inequality
\[
\fiint_{Q'_{n+1}}u^r(x,\tau)\,dx\,d\tau = \frac{1}{|Q'_{n+1}|}\iint_{Q'_{n+1}}u^r(x,\tau)\,dx\,d\tau \leq \gamma\fiint_{Q'_\infty}u^r(x,\tau)\,dx\,d\tau.
\]
Therefore, since $\rho'_{n+1}$ and $t'_{n+1}$ are comparable to $\rho$ and $t$, respectively, we find
\[J_1 \le \frac{\gamma 2^{\alpha n}}{(1-\theta)^\alpha}\Big[\sup_{Q'_{n+1}}\,u^\frac{ps(2-r)}{\lambda_2}\Big]\Big[\fiint_{Q'_\infty}u^r(x,\tau)\,dx\,d\tau\Big]^\frac{ps}{\lambda_2}\Big(\frac{t}{\rho^{ps}}\Big)^{-\frac{N}{\lambda_2}}.\]
To estimate $J_2$ we note that $\rho'_n\ge\theta\rho$, hence the tail can be treated as
\[{\rm Tail}\Big(u_+,x_0,\rho'_n,0,t\Big)^{p-1} \le \theta^{-ps}{\rm Tail}\Big(u_+,x_0,\theta\rho,0,t\Big)^{p-1}.\]
Therefore, again by comparability, we have
\[J_2 \le \gamma \theta^{-ps} \Big(\frac{t}{\rho^{ps}}\Big)^\frac{1}{2-p}\max\Big\{1,\,\Big(\frac{t}{\rho^{ps}}\Big)^\frac{1-p}{2-p}{\rm Tail}\Big(u_+,x_0,\theta \rho,0,t\Big)^{p-1}\Big\} =: \gamma \theta^{-ps} P_{\theta}.\]
Plugging such estimates into \eqref{lr1} yields, for some $\alpha,\gamma>0$ depending on the data and $\theta$,
\[S_n \le \gamma 2^{\alpha n}S_{n+1}^\frac{ps(2-r)}{\lambda_2}\Big[\fiint_{Q'_\infty}u^r(x,\tau)\,dx\,d\tau\Big]^\frac{ps}{\lambda_2}\Big(\frac{t}{\rho^{ps}}\Big)^{-\frac{N}{\lambda_2}}+\gamma P_{\theta}.\]
Recall the bounds on $r$ seen in Lemma \ref{sup}. We can then apply Young's inequality \eqref{young} to the first and second terms, with exponents
\[q = \frac{\lambda_2}{ps(2-r)}, \qquad q' = \frac{\lambda_2}{\lambda_r}\]
and $\eps\in(0,1)$ (to be determined later), to find
\[S_n \le \eps S_{n+1}+\frac{\gamma 2^{\alpha n}}{\eps^\alpha}\Big[\fiint_{Q'_\infty}u^r(x,,\tau)\,dx\,d\tau\Big]^\frac{ps}{\lambda_r}\Big(\frac{t}{\rho^{ps}}\Big)^{-\frac{N}{\lambda_r}}+P_{\theta},\]
with $\alpha,\gamma>0$ depending on the data and $\theta$.
\item[$(b)$] If $1<p\le p_c$, then recall that $r>2$. We apply Lemma \ref{sub} with the same choice of $\rho$, $t_0$, $\sigma$ as in case $(a)$, and with analogous estimates we get
\[S_n \le \gamma 2^{\alpha n}S_{n+1}^\frac{Nr-p(N+2s)}{(r-2)(N+ps)}\Big[\fiint_{Q'_\infty}u^r(x,\tau)\,dx\,d\tau\Big]^\frac{ps}{(r-2)(N+ps)}\Big(\frac{t}{\rho^{ps}}\Big)^{-\frac{N}{(r-2)(N+ps)}}+\gamma P_{\theta}.\]
This time we apply \eqref{young} with exponents
\[q = \frac{(r-2)(N+ps)}{Nr-p(N+2s)}, \ q' = \frac{(r-2)(N+ps)}{\lambda_r},\]
and $\eps\in(0,1)$ (to be determined later), to find again
\[S_n \le \eps S_{n+1}+\frac{\gamma 2^{\alpha n}}{\eps^\alpha}\Big[\fiint_{Q'_\infty}u^r(x,,\tau)\,dx\,d\tau\Big]^\frac{ps}{\lambda_r}\Big(\frac{t}{\rho^{ps}}\Big)^{-\frac{N}{\lambda_r}}+P_{\theta},\]
with $\alpha,\gamma>0$ depending on the data, $r$, and $\theta$.
\end{itemize}
In both cases above, we have obtained the same iterative inequality for the sequence $(S_n)$. Iterating such estimate from $0$ to $n$, we have
\begin{align*}
S_0 &\le \eps S_1+\frac{\gamma}{\eps^\alpha}\Big[\fiint_{Q'_\infty}u^r(x,\tau)\,dx\,d\tau\Big]^\frac{ps}{\lambda_r}\Big(\frac{t}{\rho^{ps}}\Big)^{-\frac{N}{\lambda_r}}+\gamma P_{\theta} \\
&\le \eps^2 S_2+\frac{\gamma}{\eps^\alpha}(1+\eps 2^\alpha)\Big[\fiint_{Q'_\infty}u^r(x,\tau)\,dx\,d\tau\Big]^\frac{ps}{\lambda_r}\Big(\frac{t}{\rho^{ps}}\Big)^{-\frac{N}{\lambda_r}}+\gamma(1+\eps)P_{\theta}\ldots \\
&\le \eps^n S_n+\frac{\gamma}{\eps^\alpha}\sum_{i=0}^{n-1}(\eps 2^\alpha)^i\Big[\fiint_{Q'_\infty}u^r(x,\tau)\,dx\,d\tau\Big]^\frac{ps}{\lambda_r}\Big(\frac{t}{\rho^{ps}}\Big)^{-\frac{N}{\lambda_r}}+\gamma\sum_{i=0}^{n-1}\eps^iP_{\theta}.
\end{align*}
As soon as we take $0<\eps<2^{-\alpha}$, the series above are convergent. Thus, recalling that $(S_n)$ is bounded and passing to the limit as $n\to\infty$, we have for some $\gamma>0$ depending on the data, $r$, and $\theta$
\[\sup_{B_{\theta\rho}\times (t-\theta t,\, t)} u \, \, = S_0 \le \gamma\Big[\fiint_{B_\rho(x_0)\times(0,t)}u^r(x,\tau)\,dx\,d\tau\Big]^\frac{ps}{\lambda_r}\Big(\frac{t}{\rho^{ps}}\Big)^{-\frac{N}{\lambda_r}}+\gamma P_{\theta}.\]
Finally, choosing $\theta=1/2$ and recalling the definition of $P_{\theta}$, we find
\begin{align*}
\sup_{B_{\rho/2}(x_0)\times(t/2 ,\, t)}\,u &\le \gamma\Big[\fiint_{B_\rho(x_0)\times(0,t)}u^r(x,\tau)\,dx\,d\tau\Big]^\frac{ps}{\lambda_r}\Big(\frac{t}{\rho^{ps}}\Big)^{-\frac{N}{\lambda_r}} \\
&+ \gamma\Big(\frac{t}{\rho^{ps}}\Big)^\frac{1}{2-p}\max\Big\{1,\,\Big(\frac{t}{\rho^{ps}}\Big)^\frac{1-p}{2-p}{\rm Tail}\Big(u_+,x_0,\frac{\rho}{2},0,t\Big)^{p-1}\Big\},
\end{align*}
with $\gamma>0$ depending on the data and $r$, which concludes the proof. \qed

\begin{remark}\label{gammar}
We note that, in Theorem \ref{lr}, the functional dependence of the constant $\gamma$ on $r$ degenerates when $\lambda_r$ tends to $0$. However, we observe that in the supercritical regime 
\begin{equation*}
\lambda_r = N(p-2) +rps > 0 \qquad \forall \, r \geq 1\,.
\end{equation*}
Therefore, we have the following situation:
\begin{itemize}
\item{if $1<p < \frac{2N}{N+s}$, the constant $\gamma$ depends on $r$ and becomes unbounded when $\lambda_r$ vanishes;}
\vskip0.1cm \noindent 
\item{if $\frac{2N}{N+s}<p<2$ the constant $\gamma$ does not degenerate with any choice of $r$.}
\end{itemize} The functional independence of $\gamma$ on the integrability exponent $r$ in the supercritical range indeed reflects the fact that the solutions are automatically bounded, and therefore $r$-integrable. Since this property must be assumed in the subcritical regime, the estimate of Theorem \ref{lr} trivializes as soon as $\lambda_r$ vanishes. The critical regime $p= 2N/(N+s)$ is the threshold and needs a different inspection, in the spirit of \cite{CHSS}.
\noindent 
\end{remark}

\begin{remark}\label{RMK-non-negative}
The whole double iterative procedure can be performed, as usual, for any signed sub-solution, thereby showing the $L^r$-$L^{\infty}$ estimate for $u_+$. In case $u$ is a solution, then $-u$ is a sub-solution to a similar equation and a similar bound is available for the negative part of $u$. Summarizing, if $u$ is a locally bounded local weak solution to $\eqref{eq}$ and $\lambda_r>0$ (with $1<r\leq 2$ if $p>p_c$), then the $L^r$-$L^{\infty}$ estimate is valid for $|u|$,
\begin{align*}
 \sup_{B_{\rho/2}(x_0)\times(t/2,t)}\,|u| &\le \gamma\Big[\fiint_{B_\rho(x_0)\times(0,t)}|u|^r(x,\tau)\,dx\,d\tau\Big]^\frac{ps}{\lambda_r}\Big(\frac{t}{\rho^{ps}}\Big)^{-\frac{N}{\lambda_r}} \\
&+ \gamma\Big(\frac{t}{\rho^{ps}}\Big)^\frac{1}{2-p}\max\Big\{1,\,\Big(\frac{t}{\rho^{ps}}\Big)^\frac{1-p}{2-p}{\rm Tail}\Big(u,x_0,\frac{\rho}{2},0,t\Big)^{p-1}\Big\}.
\end{align*}
\end{remark}

\section{$L^1$-$L^1$ and $L^1$-$L^{\infty}$ estimates}\label{sec4}

\noindent
In this section we focus on estimates involving the $L^1$-norm of the solution, proving Theorems \ref{l1} and \ref{linf}. First we assume $u$ to be a (possibly unbounded) solution of \eqref{eq}, s.t.\ $u\ge 0$ in $B_{4\rho}(x_0)\times(0,t)\subset\Omega_T$. Set
\beq\label{pert}
P_\pm = \max\Big\{1,\Big(\frac{t}{\rho^{ps}}\Big)^\frac{1-p}{2-p}{\rm Tail}\Big(u_\pm,x_0,\frac{\rho}{2},0,t\Big)^{p-1}\Big\}.
\eeq
We need a technical lemma, yielding a special energy estimate for $u$. We shall use it at one step only (see \eqref{vol3} below), but it is a crucial one:

\begin{lemma}\label{tst}
Let $0<\sigma<\sigma'<1$, $\gamma'>0$, $\xi\in C^1_c(B_{\sigma'\rho}(x_0))$ s.t.\
\[\xi = 1 \ \text{in $B_{\sigma\rho}(x_0)$,} \ 0 \le \xi \le 1, \ |\nabla\xi| \le \frac{\gamma'}{(\sigma'-\sigma)\rho} \ \text{in $\R^N$.}\]
Then, there exists $\gamma>0$ depending on the data and $\gamma'$, s.t.\
\begin{align*}
& \int_0^t \tau^\frac{1}{p}\iint_{A_\tau}\frac{(u(x,\tau)-u(y,\tau))^p}{|x-y|^{N+ps}}\Big[u(x,\tau)+\Big(\frac{t}{\rho^{ps}}\Big)^\frac{1}{2-p}\Big]^{-\frac{2}{p}}\xi^p(x)\,dx\,dy\,d\tau \\
&\le \gamma\rho^s\max\Big\{\frac{1}{(\sigma'-\sigma)^p},\,\frac{1}{(1-\sigma')^{N+ps}}\Big\}\Big[\sup_{0<\tau<t}\,\int_{B_\rho(x_0)}u(x,\tau)\,dx+\Big(\frac{t}{\rho^{\lambda_1}}\Big)^\frac{1}{2-p}\Big]^\frac{2(p-1)}{p}\Big(\frac{t}{\rho^{\lambda_1}}\Big)^\frac{1}{p} P_-,
\end{align*}
where for all $\tau\in(0,t)$
\[A_\tau = \big\{(x,y)\in B_{\sigma'\rho}(x_0)\times B_{\sigma'\rho}(x_0):\,u(x,\tau)>u(x,\tau),\,\xi(x)>\xi(y)\big\}.\]
\end{lemma}
\begin{proof}
Set for brevity
\[B = B_\rho(x_0), \ \check B = B_{\sigma\rho}(x_0), \ \hat B = B_{\sigma'\rho}(x_0),\]
so that $\check B\subset\hat B\subset B$, and
\[\nu = \Big(\frac{t}{\rho^{ps}}\Big)^\frac{1}{2-p} > 0.\]
Define for all $(x,\tau)\in\R^N\times(0,T)$
\[\varphi(x,\tau) = -\tau^\frac{1}{p}(u(x,\tau)+\nu)^\frac{p-2}{p}\xi^p(x).\]
The function $\varphi$ is well defined since $\xi=0$ in $\hat B^c$, yet it is not an admissible test function in \eqref{eq} as it is not differentiable in $\tau$, in general. Nevertheless, using a convenient mollification procedure (the proof is postponed to Lemma \ref{mol}), we get the following inequality, with $\gamma>0$ depending on the data:
\begin{align}\label{tst1}
0 &\ge -\frac{pt^\frac{1}{p}}{2(p-1)}\int_B (u(x,t)+\nu)^\frac{2(p-1)}{p}\xi^p(x)\,dx+\frac{1}{2(p-1)}\int_0^t \tau^\frac{1-p}{p}\int_B(u(x,\tau)+\nu)^\frac{2(p-1)}{p}\xi^p(x)\,dx\,d\tau \\
\nonumber &+ \gamma\int_0^t\iint_{\R^N\times\R^N}\frac{|u(x,\tau)-u(y,\tau)|^{p-2}(u(x,\tau)-u(y,\tau))}{|x-y|^{N+ps}}\,(\varphi(x,\tau)-\varphi(y,\tau))\,dx\,dy\,d\tau \\
\nonumber &= I_1+I_2.
\end{align}
We first deal with the evolutive term $I_1$. Since the second integral of $I_1$ is positive, we dismiss it. Besides, on the first term we apply H\"older's inequality and $\xi\le 1$:
\begin{align*}
I_1 &\ge -\frac{pt^\frac{1}{p}}{2(p-1)}\int_B(u(x,t)+\nu)^\frac{2(p-1)}{p}\xi^p(x)\,dx \\
&\ge -\gamma t^\frac{1}{p}\Big[\int_B(u(x,t)+\nu)\,dx\Big]^\frac{2(p-1)}{p}|B|^\frac{2-p}{p} \\
&\ge -\gamma\rho^\frac{N(2-p)}{p} t^\frac{1}{p}\Big[\sup_{0<\tau<t}\,\int_B u(x,\tau)\,dx+\nu\rho^N\Big]^\frac{2(p-1)}{p}.
\end{align*}
Recalling the definitions of $\lambda_1$ and $\nu$, we find
\beq\label{tst2}
I_1 \ge -\gamma\rho^s\Big[\sup_{0<\tau<t}\,\int_B u(x,\tau)\,dx+\Big(\frac{t}{\rho^{\lambda_1}}\Big)^\frac{1}{2-p}\Big]^\frac{2(p-1)}{p}\Big(\frac{t}{\rho^{\lambda_1}}\Big)^\frac{1}{p}.
\eeq
Now we turn to the diffusive term $I_2$. For simplicity, we fix $\tau\in(0,t)$ and focus on the space integral only (via a sign change):
\[I_3 = \iint_{\R^N\times\R^N}\frac{|u(x,\tau)-u(y,\tau)|^{p-2}(u(x,\tau)-u(y,\tau))}{|x-y|^{N+ps}}\Big[(u(x,\tau)+\nu)^\frac{p-2}{p}\xi^p(x)-(u(y,\tau)+\nu)^\frac{p-2}{p}\xi^p(y)\Big]\,dx\,dy.\]
By symmetry, we may rephrase such integral by setting
\begin{align*}
I_3 &= 2\iint_{A^+_\tau}\frac{(u(x,\tau)-u(y,\tau))^{p-1}}{|x-y|^{N+ps}}\Big[(u(x,\tau)+\nu)^\frac{p-2}{p}\xi^p(x)-(u(y,\tau)+\nu)^\frac{p-2}{p}\xi^p(y)\Big]\,dx\,dy \\
&+ 2\iint_{A^-_\tau}\frac{(u(x,\tau)-u(y,\tau))^{p-1}}{|x-y|^{N+ps}}\Big[(u(x,\tau)+\nu)^\frac{p-2}{p}\xi^p(x)-(u(y,\tau)+\nu)^\frac{p-2}{p}\xi^p(y)\Big]\,dx\,dy \\
&= I_4+I_5,
\end{align*}
where we have set
\[A^+_\tau = \big\{(x,y)\in\R^N\times\R^N:\,u(x,\tau)>u(y,\tau),\, u(y,\tau)\ge 0\big\},\]
\[A^-_\tau = \big\{(x,y)\in\R^N\times\R^N:\,u(x,\tau)>u(y,\tau),\, u(y,\tau)< 0\big\}.\]
We first consider $I_4$. Note that for all $(x,y)\in A^+_\tau$ s.t.\ $\xi(x)\le\xi(y)$ the integrand is negative, so we may restrict ourselves to the subdomain where $\xi(x)>\xi(y)$ and split the integrand as follows:
\begin{align*}
I_4 &\le \gamma\iint_{A^+_\tau\cap\{\xi(x)>\xi(y)\}}\frac{(u(x,\tau)-u(y,\tau))^{p-1}}{|x-y|^{N+ps}}\Big[(u(x,\tau)+\nu)^\frac{p-2}{p}-(u(y,\tau)+\nu)^\frac{p-2}{p}\Big]\xi^p(x)\,dx\,dy \\
&+ \gamma\iint_{A^+_\tau\cap\{\xi(x)>\xi(y)\}}\frac{(u(x,\tau)-u(y,\tau))^{p-1}}{|x-y|^{N+ps}}(u(y,\tau)+\nu)^\frac{p-2}{p}(\xi^p(x)-\xi^p(y))\,dx\,dy \\
&= I_6+I_7.
\end{align*}
In order to deal with $I_6$, $I_7$, respectively, we need two pointwise inequalities holding for all $(x,y)\in A^+_\tau$ s.t.\ $\xi(x)>\xi(y)$. First, by concavity we have
\[(u(x,\tau)+\nu)^\frac{p-2}{p}-(u(y,\tau)+\nu)^\frac{p-2}{p} \le \frac{p-2}{p}(u(x,\tau)+\nu)^{-\frac{2}{p}}(u(x,\tau)-u(y,\tau)).\]
So, $I_6$ is estimated by a negative quantity as follows:
\beq\label{tst3}
I_6 \le -\gamma\iint_{A^+_\tau\cap\{\xi(x)>\xi(y)\}}\frac{(u(x,\tau)-u(y,\tau))^p}{|x-y|^{N+ps}}(u(x,\tau)+\nu)^{-\frac{2}{p}}\xi^p(x)\,dx\,dy.
\eeq
Besides, for all  $(x,y)\in A^+_\tau\cap\{\xi(x)>\xi(y)\}$ we apply inequality \eqref{liao} with $a=\xi(x)>\xi(y)=b$, $\delta\in(0,1)$ to be determined later, and $\eps\in(0,1)$ defined by
\begin{align*}
\eps &= \delta\frac{u(x,\tau)-u(y,\tau)}{u(y,\tau)+\nu}\Big[\frac{u(y,\tau)+\nu}{u(x,\tau)+\nu}\Big]^\frac{2}{p} \\
&\le \Big[\frac{u(y,\tau)+\nu}{u(x,\tau)+\nu}\Big]^\frac{2-p}{p} < 1,
\end{align*}
so that the following pointwise inequality holds with $\gamma>0$ independent of $\delta$:
\begin{align*}
\xi^p(x)-\xi^p(y) &\le \delta\frac{u(x,\tau)-u(y,\tau)}{u(y,\tau)+\nu}\Big[\frac{u(y,\tau)+\nu}{u(x,\tau)+\nu}\Big]^\frac{2}{p}\xi^p(x) \\
&+ \frac{\gamma}{\delta^{p-1}}\Big[\frac{u(y,\tau)+\nu}{u(x,\tau)-u(y,\tau)}\Big]^{p-1}\Big[\frac{u(x,\tau)+\nu}{u(y,\tau)+\nu}\Big]^\frac{2(p-1)}{p}(\xi(x)-\xi(y))^p.
\end{align*}
Plugging such inequality into the integrand of $I_7$, and recalling that $u(y,\tau)+\nu\ge\nu$, we have
\begin{align}\label{tst4}
I_7 &\le \gamma\delta\iint_{A^+_\tau\cap\{\xi(x)>\xi(y)\}}\frac{(u(x,\tau)-u(y,\tau))^p}{|x-y|^{N+ps}}(u(x,\tau)+\nu)^{-\frac{2}{p}}\xi^p(x)\,dx\,dy \\
\nonumber &+ \frac{\gamma}{\delta^{p-1}\nu^{2-p}}\iint_{A^+_\tau\cap\{\xi(x)>\xi(y)\}}(u(x,\tau)+\nu)^\frac{2(p-1)}{p}\frac{(\xi(x)-\xi(y))^p}{|x-y|^{N+ps}}\,dx\,dy.
\end{align}
Choosing $\delta\in(0,1)$ small enough, we can reabsorb the first term of \eqref{tst4} into \eqref{tst3}, so that for some $\gamma>0$ depending on the data
\begin{align}\label{tst5}
I_4 &\le -\gamma\iint_{A^+_\tau\cap\{\xi(x)>\xi(y)\}}\frac{(u(x,\tau)-u(y,\tau))^p}{|x-y|^{N+ps}}(u(x,\tau)+\nu)^{-\frac{2}{p}}\xi^p(x)\,dx\,dy \\
\nonumber &+ \frac{\gamma}{\nu^{2-p}}\iint_{A^+_\tau\cap\{\xi(x)>\xi(y)\}}(u(x,\tau)+\nu)^\frac{2(p-1)}{p}\frac{(\xi(x)-\xi(y))^p}{|x-y|^{N+ps}}\,dx\,dy \\
\nonumber &= I_8+I_9.
\end{align}
In the following lines, we leave $I_8$ alone and estimate $I_9$. First note that, for all $(x,y)\in A^+_\tau\cap\{\xi(x)>\xi(y)\}$, we have in particular $\xi(x)>0$ and hence $x\in\hat B$. Therefore, we may split the space integral as follows:
\begin{align*}
I_9 &\le \frac{\gamma}{\nu^{2-p}}\iint_{\hat B\times B}(u(x,\tau)+\nu)^\frac{2(p-1)}{p}\frac{(\xi(x)-\xi(y))^p}{|x-y|^{N+ps}}\,dx\,dy \\
&+ \frac{\gamma}{\nu^{2-p}}\iint_{\hat B\times B^c}(u(x,\tau)+\nu)^\frac{2(p-1)}{p}\frac{(\xi(x)-\xi(y))^p}{|x-y|^{N+ps}}\,dx\,dy \\
&= I_{10}+I_{11}.
\end{align*}
On $I_{10}$ we act using the gradient bound on $\xi$ and H\"older's inequality:
\begin{align}\label{tst6}
I_{10} &\le \frac{\gamma}{(\sigma'-\sigma)^p\rho^p\nu^{2-p}}\iint_{\hat B\times B}\frac{(u(x,\tau)+\nu)^\frac{2(p-1)}{p}}{|x-y|^{N+ps-p}}\,dx\,dy \\
\nonumber &\le \frac{\gamma}{(\sigma'-\sigma)^p\rho^p\nu^{2-p}}\Big[\int_{\hat B}(u(x,\tau)+\nu)^\frac{2(p-1)}{p}\,dx\Big]\Big[\sup_{x\in\hat B}\int_B\frac{dy}{|x-y|^{N+ps-p}}\,dy\Big] \\
\nonumber &\le \frac{\gamma}{(\sigma'-\sigma)^p\rho^p\nu^{2-p}}\Big[\int_{B}(u(x,\tau)+\nu)\,dx\Big]^\frac{2(p-1)}{p}|B|^\frac{2-p}{p}\Big[\int_{B_{2\rho}(0)}\frac{dz}{|z|^{N+ps-p}}\Big] \\
\nonumber &\le \frac{\gamma\rho^{s-ps-\frac{\lambda_1}{p}}}{(\sigma'-\sigma)^p\nu^{2-p}}\Big[\int_B u(x,\tau)\,dx+\nu\rho^N\Big]^\frac{2(p-1)}{p}.
\end{align}
To estimate $I_{11}$, we note that for all $x\in\hat B$, $y\in B^c$
\[|x-y| \ge |y-x_0|-|x-x_0| \ge (1-\sigma')|y-x_0|,\]
so by $0\le\xi\le 1$ and H\"older's inequality again we get
\begin{align}\label{tst7}
I_{11} &\le \frac{\gamma}{(1-\sigma')^{N+ps}\nu^{2-p}}\iint_{\hat B\times B^c}\frac{(u(x,\tau)+\nu)^\frac{2(p-1)}{p}}{|y-x_0|^{N+ps}}\,dx\,dy \\
\nonumber &\le \frac{\gamma}{(1-\sigma')^{N+ps}\nu^{2-p}}\Big[\int_B(u(x,\tau)+\nu)^\frac{2(p-1)}{p}\,dx\Big]\Big[\int_{B^c}\frac{dy}{|y-x_0|^{N+ps}}\Big] \\
\nonumber &\le \frac{\gamma}{(1-\sigma')^{N+ps}\rho^{ps}\nu^{2-p}}\Big[\int_B(u(x,\tau)+\nu)\,dx\Big]^\frac{2(p-1)}{p}|B|^\frac{2-p}{p} \\
\nonumber &\le \frac{\gamma\rho^{s-ps-\frac{\lambda_1}{p}}}{(1-\sigma')^{N+ps}\nu^{2-p}}\Big[\int_B u(x,\tau)\,dx+\nu\rho^N\Big]^\frac{2(p-1)}{p}.
\end{align}
By \eqref{tst6} and \eqref{tst7} we have
\[I_9 \le \frac{\gamma\rho^{s-ps-\frac{\lambda_1}{p}}}{\nu^{2-p}}\max\Big\{\frac{1}{(\sigma'-\sigma)^p},\,\frac{1}{(1-\sigma')^{N+ps}}\Big\}\Big[\int_B u(x,\tau)\,dx+\nu\rho^N\Big]^\frac{2(p-1)}{p}.\]
Going back to \eqref{tst5} and recalling the definition of $\nu$ we have
\begin{align}\label{tst8}
I_4 &\le -\gamma\iint_{A^+_\tau\cap\{\xi(x)>\xi(y)\}}\frac{(u(x,\tau)-u(y,\tau))^p}{|x-y|^{N+ps}}\Big[u(x,\tau)+\Big(\frac{t}{\rho^{ps}}\Big)^\frac{1}{2-p}\Big]^{-\frac{2}{p}}\xi^p(x)\,dx\,dy \\
\nonumber &+ \frac{\gamma\rho^{s-\frac{\lambda_1}{p}}}{t}\max\Big\{\frac{1}{(\sigma'-\sigma)^p},\,\frac{1}{(1-\sigma')^{N+ps}}\Big\}\Big[\int_B u(x,\tau)\,dx+\Big(\frac{t}{\rho^{\lambda_1}}\Big)^\frac{1}{2-p}\Big]^\frac{2(p-1)}{p}.
\end{align}
There remains to estimate $I_5$. First we observe that, for all $(x,y)\in A^-_\tau$ we have $u(y,\tau)< 0$ and hence by assumption $y\in B^c$. This in turn implies $\xi(y)=0$, so we may reduce the integrand. Also, by subadditivity we have
\[(u(x,\tau)-u(y,\tau))^{p-1} \le u^{p-1}(x,\tau)+u_-^{p-1}(y,\tau).\]
Therefore, we can estimate $I_5$ as follows, further reducing the integration domain by $\xi=0$ in $\hat B^c$:
\begin{align}\label{tst9}
I_5 &\le \gamma\iint_{A^-_\tau}\frac{u^{p-1}(x,\tau)+u_-^{p-1}(y,\tau)}{|x-y|^{N+ps}}(u(x,\tau)+\nu)^\frac{p-2}{p}\xi^p(x)\,dx\,dy \\
\nonumber &\le \frac{\gamma}{\nu^{2-p}}\iint_{\hat B\times B^c}\frac{(u(x,\tau)+\nu)^\frac{2(p-1)}{p}}{|x-y|^{N+ps}}\,dx\,dy+\frac{\gamma}{\nu}\iint_{\hat B\times B^c}\frac{(u(x,\tau)+\nu)^\frac{2(p-1)}{p}}{|x-y|^{N+ps}}u_-^{p-1}(y,\tau)\,dx\,dy \\
\nonumber &= I_{12}+I_{13}.
\end{align}
Now $I_{12}$ is estimated similarly to \eqref{tst7}:
\beq\label{tst10}
I_{12} \le \frac{\gamma\rho^{s-ps-\frac{\lambda_1}{p}}}{(1-\sigma')^{N+ps}\nu^{2-p}}\Big[\int_B u(x,\tau)\,dx+\nu\rho^N\Big]^\frac{2(p-1)}{p}.
\eeq
The estimate of $I_{13}$ is more delicate, as it involves a tail-type integral of $u_-$ (see \eqref{tail}). As usual, for all $(x,y)\in\hat B\times B^c$ we have
\[|x-y| \ge (1-\sigma')|y-x_0|,\]
hence we can separate the integrals as follows:
\begin{align}\label{tst11}
I_{13} &\le \frac{\gamma}{(1-\sigma')^{N+ps}\nu}\Big[\int_{\hat B}(u(x,\tau)+\nu)^\frac{2(p-1)}{p}\,dx\Big]\Big[\int_{B^c}\frac{u_-^{p-1}(y,\tau)}{|y-x_0|^{N+ps}}\,dy\Big] \\
\nonumber &\le \frac{\gamma}{(1-\sigma')^{N+ps}\rho^{ps}\nu}\Big[\int_B(u(x,\tau)+\nu)\,dx\Big]^\frac{2(p-1)}{p}|B|^\frac{2-p}{p}\Big[\Big(\frac{\rho}{2}\Big)^{ps}\int_{B^c}\frac{u_-^{p-1}(y,\tau)}{|y-x_0|^{N+ps}}\,dy\Big] \\
\nonumber &\le \frac{\gamma\rho^{s-ps-\frac{\lambda_1}{p}}}{(1-\sigma')^{N+ps}\nu}\Big[\int_B u(x,\tau)\,dx+\nu\rho^N\Big]^\frac{2(p-1)}{p}{\rm Tail}\Big(u_-,x_0,\frac{\rho}{2},0,t\Big)^{p-1}.
\end{align}
Plugging \eqref{tst10} and \eqref{tst11} back into \eqref{tst9}, and recalling the definition of $\nu$, we get
\beq\label{tst12}
I_5 \le \frac{\gamma\rho^{s-\frac{\lambda_1}{p}}}{(1-\sigma')^{N+ps}t}\Big[\int_B u(x,\tau)\,dx+\Big(\frac{t}{\rho^{\lambda_1}}\Big)^\frac{1}{2-p}\Big]^\frac{2(p-1)}{p}P_-,
\eeq
where $P_-\ge 1$ is defined by \eqref{pert}. Note that \eqref{tst8} and \eqref{tst12} provide us with homogeneous estimates of $I_4$ and $I_5$, respectively (except for the constants and the perturbation $P_-$)), so we have for all $\tau\in(0,t)$
\begin{align*}
I_3 &= I_4+I_5 \\
&\le -\gamma\iint_{A^+_\tau\cap\{\xi(x)>\xi(y)\}}\frac{(u(x,\tau)-u(y,\tau))^p}{|x-y|^{N+ps}}\Big[u(x,\tau)+\Big(\frac{t}{\rho^{ps}}\Big)^\frac{1}{2-p}\Big]^{-\frac{2}{p}}\xi^p(x)\,dx\,dy \\
&+ \frac{\gamma\rho^{s-\frac{\lambda_1}{p}}}{t}\max\Big\{\frac{1}{(\sigma'-\sigma)^p},\,\frac{1}{(1-\sigma')^{N+ps}}\Big\}\Big[\int_B u(x,\tau)\,dx+\Big(\frac{t}{\rho^{\lambda_1}}\Big)^\frac{1}{2-p}\Big]^\frac{2(p-1)}{p}P_-.
\end{align*}
A time integration yields
\begin{align}\label{tst13}
I_2 &\ge \gamma\int_0^t \tau^\frac{1}{p}\iint_{A^+_\tau\cap\{\xi(x)>\xi(y)\}}\frac{(u(x,\tau)-u(y,\tau))^p}{|x-y|^{N+ps}}\Big[u(x,\tau)+\Big(\frac{t}{\rho^{ps}}\Big)^\frac{1}{2-p}\Big]^{-\frac{2}{p}}\xi^p(x)\,dx\,dy\,d\tau \\
\nonumber &- \gamma\rho^s\max\Big\{\frac{1}{(\sigma'-\sigma)^p},\,\frac{1}{(1-\sigma')^{N+ps}}\Big\}\Big[\sup_{0<\tau<t}\,\int_B u(x,\tau)\,dx+\Big(\frac{t}{\rho^{\lambda_1}}\Big)^\frac{1}{2-p}\Big]^\frac{2(p-1)}{p}\Big(\frac{t}{\rho^{\lambda_1}}\Big)^\frac{1}{p}P_-.
\end{align}
Finally, we combine \eqref{tst1}, \eqref{tst2}, and \eqref{tst13} and reduce the space integration domain in the first term, to find
\begin{align*}
0 &\ge I_1+I_2 \\
&\ge \gamma\int_0^t \tau^\frac{1}{p}\iint_{A_\tau}\frac{(u(x,\tau)-u(y,\tau))^p}{|x-y|^{N+ps}}\Big[u(x,\tau)+\Big(\frac{t}{\rho^{ps}}\Big)^\frac{1}{2-p}\Big]^{-\frac{2}{p}}\xi^p(x)\,dx\,dy\,d\tau \\
&- \gamma\rho^s\max\Big\{\frac{1}{(\sigma'-\sigma)^p},\,\frac{1}{(1-\sigma')^{N+ps}}\Big\}\Big[\sup_{0<\tau<t}\,\int_B u(x,\tau)\,dx+\Big(\frac{t}{\rho^{\lambda_1}}\Big)^\frac{1}{2-p}\Big]^\frac{2(p-1)}{p}\Big(\frac{t}{\rho^{\lambda_1}}\Big)^\frac{1}{p}P_-,
\end{align*}
which yields the conclusion, with a different $\gamma>0$ depending on the data and on $\gamma'$.
\end{proof}

\noindent
We can now prove the $L^1-L^1$ estimate:
\vskip4pt
\noindent
{\em Proof of Theorem \ref{l1} (conclusion).} We perform an increasing iteration by setting for all $n\in\N$
\[\rho_n = \rho\sum_{i=0}^n\frac{1}{2^i}, \ \check\rho_n = \frac{3\rho_n+\rho_{n+1}}{4}, \ \hat\rho_n = \frac{\rho_n+3\rho_{n+1}}{4},\]
so that $\rho_n<\check\rho_n<\hat\rho_n<\rho_{n+1}$, with $\rho_0=\rho$ and limit $\rho_\infty=2\rho$ as $n\to\infty$. Also set
\[B_n = B_{\rho_n}(x_0), \ \check B_n = B_{\check\rho_n}(x_0), \ \hat B_n = B_{\hat\rho_n}(x_0),\]
so $B_n\subset\check B_n\subset\hat B_n\subset B_{n+1}$, and $B_\infty=B_{2\rho}(x_0)$. For all $n\in\N$ let $t_1\in[0,t]$ be s.t.\
\[\int_{B_n}u(x,t_1)\,dx = \sup_{0<\tau<t}\,\int_{B_n}u(x,\tau)\,dx = S_n,\]
noting as well that $(S_n)$ is a bounded sequence in $[0,+\infty)$ as $u\in C(0,T,L^1(B_{2\rho}(x_0))$ (see Definition \ref{lws}). Also, we can find $t_2\in[0,t]$, independent of $n$, s.t.\
\[\int_{B_\infty}u(x,t_2)\,dx = \inf_{0<\tau<t}\,\int_{B_\infty}u(x,\tau)\,dx = J_1.\]
Let us fix $n\in\N$. Without loss of generality, we shall assume henceforth that $0\le t_1\le t_2\le t$.  In addition, we pick $\xi_n\in C^\infty_c(\check B_n)$ s.t.\
\[\xi_n = 1 \ \text{in $B_n$,} \ 0 \le \xi_n \le 1, \ |\nabla\xi_n| \le \frac{\gamma 2^n}{\rho} \ \text{in $\R^N$.}\]
We use $\xi_n^{p+1}\in C^1_c(\check B_n)$ as a (stationary) test function for \eqref{eq} in the cylinder $\check B_n\times(t_1,t_2)$, to get
\begin{align*}
0 &= \int_{\check B_n}u(x,\tau)\xi_n^{p+1}(x)\,dx\Big|_{t_1}^{t_2} \\
&+ \int_{t_1}^{t_2}\iint_{\R^N\times\R^N}|u(x,\tau)-u(y,\tau)|^{p-2}(u(x,\tau)-u(y,\tau))(\xi_n^{p+1}(x)-\xi_n^{p+1}(y))K(x,y,t)\,dx\,dy\,d\tau.
\end{align*}
This in turn, by the properties of $\xi_n$, implies
\begin{align}\label{vol1}
& S_n \le \int_{\check B_n}u(x,t_1)\xi_n^{p+1}(x)\,dx \\
\nonumber &\le \int_{B_\infty}u(x,t_2)\,dx+\gamma\int_{t_1}^{t_2}\iint_{\R^N\times\R^N}\frac{|u(x,\tau)-u(y,\tau)|^{p-2}(u(x,\tau)-u(y,\tau))}{|x-y|^{N+ps}}(\xi_n^{p+1}(x)-\xi_n^{p+1}(y))\,dx\,dy\,d\tau \\
\nonumber &= J_1+J_2.
\end{align}
We focus on the diffusive term $J_2$. First, using symmetry and the properties of $\xi_n$, we split such term into three integrals:
\begin{align*}
J_2 &\le \gamma\int_{t_1}^{t_2}\iint_{\hat B_n\times\hat B_n}\frac{|u(x,\tau)-u(y,\tau)|^{p-2}(u(x,\tau)-u(y,\tau))}{|x-y|^{N+ps}}(\xi_n^{p+1}(x)-\xi_n^{p+1}(y))\,dx\,dy\,d\tau \\
&+ 2\gamma\int_{t_1}^{t_2}\iint_{\hat B_n\times\hat B_n^c}\frac{|u(x,\tau)-u(y,\tau)|^{p-2}(u(x,\tau)-u(y,\tau))}{|x-y|^{N+ps}}\xi_n^{p+1}(x)\,dx\,dy\,d\tau \\
&\le \gamma\int_{t_1}^{t_2}\iint_{\hat B_n\times\hat B_n^c}\frac{|u(x,\tau)-u(y,\tau)|^{p-2}(u(x,\tau)-u(y,\tau))}{|x-y|^{N+ps}}(\xi_n^{p+1}(x)-\xi_n^{p+1}(y))\,dx\,dy\,d\tau \\
&+ \gamma\int_{t_1}^{t_2}\iint_{\hat B_n\times\hat B_n^c}\frac{u^{p-1}(x,\tau)}{|x-y|^{N+ps}}\xi_n^{p+1}(x)\,dx\,dy\,d\tau+\gamma\int_{t_1}^{t_2}\iint_{\hat B_n\times\hat B_n^c}\frac{u_-^{p-1}(y,\tau)}{|x-y|^{N+ps}}\xi_n^{p+1}(x)\,dx\,dy\,d\tau \\
&= J_3+J_4+J_5.
\end{align*}
We are going to separately estimate $J_3$, $J_4$, and $J_5$. First let us deal with $J_3$. Note that, for all $\tau\in(0,t)$ and $(x,y)\in\hat B_n\times\hat B_n$, we have $u(y,\tau)\ge 0$, in addition the integrand in $J_3$ is positive iff the differences $u(x,\tau)-u(y,\tau)$, $\xi_n(x)-\xi_n(y)$ have the same sign. Therefore, exploiting also symmetry and expanding the time integration interval, we have
\[J_3 \le \gamma\int_0^t\iint_{A_{n,\tau}}\frac{(u(x,\tau)-u(y,\tau))^{p-1}}{|x-y|^{N+ps}}(\xi_n^{p+1}(x)-\xi_n^{p+1}(y))\,dx\,dy\,d\tau,\]
where for all $\tau\in(0,t)$ we have set
\[A_{n,\tau} = \big\{(x,y)\in\hat B_n\times\hat B_n:\,u(x,\tau)>u(y,\tau),\,\xi_n(x)>\xi_n(y)\big\}.\]
Next, note that for all $(x,y)\in A_{n,\tau}$
\[0 \le \xi_n^{p+1}(x)-\xi_n^{p+1}(y) \le \frac{\gamma 2^n\xi_n^p(x)}{\rho}|x-y|,\]
with $\gamma>0$ independent of $n$. We then set
\[\nu_n = \Big(\frac{t}{\rho_{n+1}^{ps}}\Big)^\frac{1}{2-p} > 0,\]
so by the previous estimate on the cut-off function and weighted H\"older's inequality we have
\begin{align}\label{vol2}
J_3 &\le \frac{\gamma 2^n}{\rho}\int_0^t\iint_{A_{n,\tau}}\frac{(u(x,\tau)-u(y,\tau))^{p-1}}{|x-y|^{N+ps-1}}\xi_n^p(x)\,dx\,dy\,d\tau \\
\nonumber &\le \frac{\gamma 2^n}{\rho}\int_0^t\iint_{A_{n,\tau}}\Big[(u(x,\tau)-u(y,\tau))^{p-1}(u(x,\tau)+\nu_n)^\frac{2(1-p)}{p^2}\tau^\frac{p-1}{p^2}\Big] \\
\nonumber &\cdot \Big[(u(x,\tau)+\nu_n)^\frac{2(p-1)}{p^2}\tau^\frac{1-p}{p^2}|x-y|\Big]\frac{\xi_n^p(x)}{|x-y|^{N+ps}}\,dx\,dy\,d\tau \\
\nonumber &\le \frac{\gamma 2^n}{\rho}\Big[\int_0^t \tau^\frac{1}{p}\iint_{A_{n,\tau}}\frac{(u(x,\tau)-u(y,\tau))^p}{|x-y|^{N+ps}}(u(x,\tau)+\nu_n)^{-\frac{2}{p}}\xi_n^p(x)\,dx\,dy\,d\tau\Big]^\frac{p-1}{p} \\
\nonumber &\cdot \Big[\int_0^t \tau^\frac{1-p}{p}\iint_{A_{n,\tau}}\frac{(u(x,\tau)+\nu_n)^\frac{2(p-1)}{p}}{|x-y|^{N+ps-p}}\xi_n^p(x)\,dx\,dy\,d\tau\Big]^\frac{1}{p} = \frac{\gamma 2^n}{\rho}J_6^\frac{p-1}{p}J_7^\frac{1}{p}.
\end{align}
To estimate $J_6$, we apply Lemma \ref{tst} with the following choices:
\[\rho = \rho_{n+1}, \ \sigma = \frac{\rho_n}{\rho_{n+1}}, \ \sigma' = \frac{\check\rho_n}{\rho_{n+1}}, \ \xi = \xi_n.\]
Note that $0<\sigma<\sigma'<1$ with differences estimated respectively by
\[\sigma'-\sigma = \frac{\rho_{n+1}-\rho_n}{4\rho_{n+1}} \ge \frac{1}{2^{n+4}}, \quad 1-\sigma' = \frac{3\rho_{n+1}-3\rho_n}{4\rho_{n+1}} \ge \frac{3}{2^{n+4}}.\]
Also we have $\xi_n\in C^1_c(\check B_n)$, $\xi_n=1$ in $B_n$, and in all of $\R^N$
\[0 \le \xi_n \le 1, \ |\nabla\xi_n| \le \frac{\gamma'}{(\sigma'-\sigma)\rho},\]
for some $\gamma'>0$ depending on the data. In addition, $A_{n,\tau}$ coincides with the domain $A_\tau$ of Lemma \ref{tst} for all $\tau\in(0,t)$. Therefore, recalling that $\rho_{n+1}$ and $\rho$ are comparable via numerical constants, we have
\begin{align}\label{vol3}
J_6 &\le \gamma\rho_{n+1}^s\max\big\{2^{pn},\,2^{(N+ps)n}\big\}\Big[\sup_{0<\tau<t}\,\int_{B_{n+1}}u(x,\tau)\,dx+\Big(\frac{t}{\rho_{n+1}^{\lambda_1}}\Big)^\frac{1}{2-p}\Big]^\frac{2(p-1)}{p}\Big(\frac{t}{\rho_{n+1}^{\lambda_1}}\Big)^\frac{1}{p} \\
\nonumber &\cdot \max\Big\{1,\,\Big(\frac{t}{\rho_{n+1}^{ps}}\Big)^\frac{1-p}{2-p}{\rm Tail}\Big(u_-,x_0,\frac{\rho_{n+1}}{2},0,t\Big)^{p-1}\Big\} \\
\nonumber &\le \gamma b^n\rho^s\Big[S_{n+1}+\Big(\frac{t}{\rho^{\lambda_1}}\Big)^\frac{1}{2-p}\Big]^\frac{2(p-1)}{p}\Big(\frac{t}{\rho^{\lambda_1}}\Big)^\frac{1}{p}P_-,
\end{align}
where $P_-$ is defined by \eqref{pert} and $b>1$ depends on the data. For $J_7$ we have the following estimate:
\begin{align}\label{vol4}
J_7 &\le \Big[\int_0^t \tau^\frac{1-p}{p}\,d\tau\Big]\Big[\sup_{0<\tau<t}\,\int_{B_{n+1}}(u(x,\tau)+\nu_n)^\frac{2(p-1)}{p}\,dx\Big]\Big[\sup_{x\in B_{n+1}}\,\int_{B_{n+1}}\frac{dy}{|x-y|^{N+ps-p}}\Big] \\
\nonumber &\le \gamma t^\frac{1}{p}\Big[\sup_{0<\tau<t}\,\int_{B_{n+1}}(u(x,\tau)+\nu_n)\,dx\Big]^\frac{2(p-1)}{p}|B_{n+1}|^\frac{2-p}{p}\Big[\int_{B_{4\rho(0)}}\frac{dz}{|z|^{N+ps-p}}\Big] \\
\nonumber &\le \gamma t^\frac{1}{p}\rho^{\frac{N(2-p)}{p}+p-ps}\Big[\sup_{0<\tau<t}\,\int_{B_{n+1}}u(x,\tau)\,dx+\nu_n\rho_{n+1}^N\Big]^\frac{2(p-1)}{p} \\
\nonumber &\le \gamma\rho^{p+s-ps}\Big[S_{n+1}+\Big(\frac{t}{\rho^{\lambda_1}}\Big)^\frac{1}{2-p}\Big]^\frac{2(p-1)}{p}\Big(\frac{t}{\rho^{\lambda_1}}\Big)^\frac{1}{p}.
\end{align}
Plugging \eqref{vol3} and \eqref{vol4} back into \eqref{vol2}, we find $\gamma,b>1$ depending on the data s.t.\
\[J_3 \le \gamma b^n\Big[S_{n+1}+\Big(\frac{t}{\rho^{\lambda_1}}\Big)^\frac{1}{2-p}\Big]^\frac{2(p-1)}{p}\Big(\frac{t}{\rho^{\lambda_1}}\Big)^\frac{1}{p}P_-^\frac{p-1}{p}\]
Next we apply Young's inequality \eqref{young} with $q=p/(2(p-1))>1$ and $\eps\in(0,1)$ to be determined later:
\beq\label{vol5}
J_3 \le \eps\Big[S_{n+1}+\Big(\frac{t}{\rho^{\lambda_1}}\Big)^\frac{1}{2-p}\Big]+\gamma_\eps b^n\Big(\frac{t}{\rho^{\lambda_1}}\Big)^\frac{1}{2-p}P_-^\frac{p-1}{2-p},
\eeq
with $\gamma_\eps\to\infty$ as $\eps\to 0$ and $b>1$ depending on the data. The integrals $J_4$, $J_5$ are easier to deal with. To estimate $J_4$, we first recall that $\xi_n=0$ in $\check B_n^c$ and $0\le\xi_n\le 1$ in all of $\R^N$, hence
\[J_4 \le \gamma\int_0^t\iint_{\check B_n\times\hat B_n^c}\frac{u^{p-1}(x,\tau)}{|x-y|^{N+ps}}\,dx\,dy\,d\tau.\]
Besides, we note that for all $x\in\check B_n$, $y\in\hat B_n^c$ we have
\[|x-y| \ge |y-x_0|-|x-x_0| \ge \frac{|y-x_0|}{\gamma 2^n}.\]
So, using both H\"older's and Young's inequalities \eqref{young} with $q=1/(p-1)>1$ and $\eps\in(0,1)$ to be determined, we get
\begin{align}\label{vol6}
J_4 &\le \gamma t\Big[\sup_{0<\tau<t}\,\int_{\check B_n}u^{p-1}(x,\tau)\,dx\Big]\Big[\sup_{x\in\check B_n}\,\int_{\hat B_n^c}\frac{dy}{|x-y|^{N+ps}}\Big] \\
\nonumber &\le \gamma 2^{(N+ps)n}t\Big[\sup_{0<\tau<t}\,\int_{B_{n+1}}u^{p-1}(x,\tau)\,dx\Big]\Big[\int_{B_\rho^c(0)}\frac{dz}{|z|^{N+ps}}\Big] \\
\nonumber &\le \frac{\gamma b^nt}{\rho^{ps}}\Big[\sup_{0<\tau<t}\,\int_{B_{n+1}}u(x,\tau)\,dx\Big]^{p-1}|B_{n+1}|^{2-p} \\
\nonumber &\le \frac{\gamma b^nt}{\rho^{\lambda_1}}S_{n+1}^{p-1} \le \eps S_{n+1}+\gamma_\eps b^n\Big(\frac{t}{\rho^{\lambda_1}}\Big)^\frac{1}{2-p},
\end{align}
again with $\gamma_\eps\to\infty$ as $\eps\to 0$ and $b>1$ depending on the data. The estimate of $J_5$ begins as above, but involves a further tail term:
\begin{align}\label{vol7}
J_5 &\le \gamma\int_0^t\iint_{\check B_n\times\hat B_n^c}\frac{u_-^{p-1}(y,\tau)}{|x-y|^{N+ps}}\,dx\,dy\,d\tau \\
\nonumber &\le \gamma 2^{(N+ps)n}t|\check B_n|\Big[\sup_{0<\tau<t}\,\int_{\hat B_n^c}\frac{u_-^{p-1}(y,\tau)}{|y-x_0|^{N+ps}}\,dy\Big] \\
\nonumber &\le \gamma b^n t\rho^{N-ps}{\rm Tail}\Big(u_-,x_0,\frac{\hat\rho_n}{2},0,t\Big)^{p-1} \\
\nonumber &\le \gamma b^n \Big(\frac{t}{\rho^{\lambda_1}}\Big)^\frac{1}{2-p}P_-,
\end{align}
where we used comparability of $\rho$ and $\hat\rho_n$, and as usual $b>1$ only depends on the data. Next we gather all estimates from \eqref{vol5}, \eqref{vol6}, and \eqref{vol7} to find the following estimate for $J_2$, holding for $\eps\in(0,1)$ to be determined, $\gamma_\eps\to\infty$ as $\eps\to 0$, and $b>1$ depending on the data:
\[J_2 \le \eps S_{n+1}+\gamma_\eps b^n\Big(\frac{t}{\rho^{\lambda_1}}\Big)^\frac{1}{2-p}\big(P_-+P_-^\frac{p-1}{2-p}\big),\]
where we have used $P_-\ge 1$, but we could not choose a unique exponent for $P_-$ since $(p-1)/(2-p)$ spans the whole interval $(0,\infty)$ for $p\in(1,2)$. Now \eqref{vol1} yields an iterative bound for the sequence $(S_n)$:
\begin{align}\label{vol8}
S_n &\le J_1+\eps S_{n+1}+\gamma_\eps b^n\Big(\frac{t}{\rho^{\lambda_1}}\Big)^\frac{1}{2-p}\big(P_-+P_-^\frac{p-1}{2-p}\big) \\
\nonumber &\le \eps S_{n+1}+\gamma_\eps b^n\Big[J_1+\Big(\frac{t}{\rho^{\lambda_1}}\Big)^\frac{1}{2-p}\big(P_-+P_-^\frac{p-1}{2-p}\big)\Big].
\end{align}
Iterating on \eqref{vol8} from $0$ to $n$, we find
\begin{align*}
S_0 &\le \eps S_1+\gamma\Big[J_1+\Big(\frac{t}{\rho^{\lambda_1}}\Big)^\frac{1}{2-p}\big(P_-+P_-^\frac{p-1}{2-p}\big)\Big] \\
&\le \eps^2 S_2+\gamma(1+\eps b)\Big[J_1+\Big(\frac{t}{\rho^{\lambda_1}}\Big)^\frac{1}{2-p}\big(P_-+P_-^\frac{p-1}{2-p}\big)\Big] \ \ldots \\
&\le \eps^n S_n+\gamma\sum_{i=0}^{n-1}(\eps b)^i\Big[J_1+\Big(\frac{t}{\rho^{\lambda_1}}\Big)^\frac{1}{2-p}\big(P_-+P_-^\frac{p-1}{2-p}\big)\Big].
\end{align*}
Choose now $\eps\in(0,1/b)$ so that the series on the right hand side converges, and recall that $(S_n)$ is a bounded sequence. So, letting $n\to\infty$, we get
\[S_0 \le \gamma\Big[J_1+\Big(\frac{t}{\rho^{\lambda_1}}\Big)^\frac{1}{2-p}\big(P_-+P_-^\frac{p-1}{2-p}\big)\Big],\]
which rephrases as
\begin{align*}
\sup_{0<\tau<t}\,\int_{B_\rho(x_0)}u(x,\tau)\,dx &\le \gamma\inf_{0<\tau<t}\,\int_{B_{2\rho}(x_0)}u(x,\tau)\,dx+\gamma\Big(\frac{t}{\rho^{\lambda_1}}\Big)^\frac{1}{2-p}\big(P_-+P_-^\frac{p-1}{2-p}\big),
\end{align*}
yielding the conclusion as soon as we recall the definition of $P_-$. \qed
\vskip4pt
\noindent
Combining Theorems \ref{lr} and \ref{l1}, it is not difficult to obtain a $L^1-L^\infty$ estimate, provided the solution $u$ is locally bounded in addition to the previous assumptions:
\vskip4pt
\noindent
{\em Proof of Theorem \ref{linf}.} First, from Theorem \ref{l1} we derive the following inequality involving the mean values of $u$ in different balls:
\begin{align}\label{linf1}
\sup_{0<\tau<t}\,\fint_{B_\rho(x_0)}u(x,\tau)\,dx &\le \gamma\Big[\inf_{0<\tau<t}\,\fint_{B_{2\rho}(x_0)}u(x,\tau)\,dx\Big]+\frac{\gamma}{\rho^N}\Big(\frac{t}{\rho^{\lambda_1}}\Big)^\frac{1}{2-p}\big(P_-+P_-^\frac{p-1}{2-p}\big) \\
\nonumber &\le \gamma\Big[\inf_{0<\tau<t}\,\fint_{B_{2\rho}(x_0)}u(x,\tau)\,dx\Big]+\gamma\Big(\frac{t}{\rho^{ps}}\Big)^\frac{1}{2-p}\big(P_-+P_-^\frac{p-1}{2-p}\big),
\end{align}
where $\gamma$ depends on the data. Besides, since $\lambda_1>0$ and $u$ is locally bounded, we may apply Theorem \ref{lr} with $r=1$:
\beq\label{linf2}
\sup_{B_{\rho/2}(x_0)\times(t/2,t)}\,u \le \gamma\Big[\fiint_{B_\rho(x_0)\times(0,t)}u(x,\tau)\,dx\,d\tau\Big]^\frac{ps}{\lambda_1}\Big(\frac{t}{\rho^{ps}}\Big)^{-\frac{N}{\lambda_1}}+\gamma\Big(\frac{t}{\rho^{ps}}\Big)^\frac{1}{2-p}P_+,
\eeq
with a possibly bigger $\gamma>0$ depending on the data. Concatenating \eqref{linf1} and \eqref{linf2} and changing the constant $\gamma$ conveniently, we find
\begin{align*}
\sup_{B_{\rho/2}(x_0)\times(t/2,t)}\,u &\le \gamma\Big[\sup_{0<\tau<t}\,\fint_{B_\rho(x_0)}u(x,\tau)\,dx\Big]^\frac{ps}{\lambda_1}\Big(\frac{t}{\rho^{ps}}\Big)^{-\frac{N}{\lambda_1}}+\gamma\Big(\frac{t}{\rho^{ps}}\Big)^\frac{1}{2-p}P_+ \\
&\le \gamma\Big[\inf_{0<\tau<t}\,\fint_{B_{2\rho}(x_0)}u(x,\tau)\,dx+\Big(\frac{t}{\rho^{ps}}\Big)^\frac{1}{2-p}\big(P_-+P_-^\frac{p-1}{2-p}\big)\Big]^\frac{ps}{\lambda_1}\Big(\frac{t}{\rho^{ps}}\Big)^{-\frac{N}{\lambda_1}}+\gamma\Big(\frac{t}{\rho^{ps}}\Big)^\frac{1}{2-p}P_+ \\
&\le \gamma\Big[\inf_{0<\tau<t}\,\fint_{B_{2\rho}(x_0)}u(x,\tau)\,dx\Big]^\frac{ps}{\lambda_1}\Big(\frac{t}{\rho^{ps}}\Big)^{-\frac{N}{\lambda_1}} \\
&+ \gamma\Big(\frac{t}{\rho^{ps}}\Big)^{\frac{ps}{(2-p)\lambda_1}-\frac{N}{\lambda_1}}\big(P_-+P_-^\frac{p-1}{2-p}\big)^\frac{ps}{\lambda_1}+\gamma\Big(\frac{t}{\rho^{ps}}\Big)^\frac{1}{2-p}P_+ \\
&\le \gamma\Big[\inf_{0<\tau<t}\,\int_{B_{2\rho}(x_0)}u(x,\tau)\,dx\Big]^\frac{ps}{\lambda_1}t^{-\frac{N}{\lambda_1}}+\gamma\Big(\frac{t}{\rho^{ps}}\Big)^\frac{1}{2-p}\big(P_++P_-^\frac{ps}{\lambda_1}+P_-^\frac{(p-1)ps}{(2-p)\lambda_1}\big),
\end{align*}
which concludes the proof. \qed

\noindent
It is worth pointing out the following special cases of Theorems \ref{l1} and \ref{linf} for {\em globally non-negative} solutions (for which $P_-=1$):

\begin{corollary}\label{pos}
{\rm (Globally non-negative solutions)} Let $u$ be a solution of \eqref{eq}, s.t.\ $u\ge 0$ in $\R^N\times(0,T)$, and $B_{4\rho}(x_0)\times(0,t)\subset\Omega_T$. There exists a constant $\gamma>0$ depending only on the data s.t.
\begin{enumroman}
\item\label{pos1} for all $1<p<2$ and $\lambda_1$ of any sign,
\[\sup_{0<\tau<t}\,\int_{B_\rho(x_0)}u(x,\tau)\,dx \le \gamma\Big[\inf_{0<\tau<t}\,\int_{B_{2\rho}(x_0)}u(x,\tau)\,dx\Big]+\gamma\Big(\frac{t}{\rho^{\lambda_1}}\Big)^\frac{1}{2-p};\]
\item\label{pos2} if $\lambda_1>0$ and $u$ is locally bounded, then 
\begin{align*}
\sup_{B_{\rho/2}(x_0)\times(t/2,t)}\,u &\le \gamma\Big[\inf_{0<\tau<t}\,\int_{B_{2\rho}(x_0)}u(x,\tau)\,dx\Big]^\frac{ps}{\lambda_1}t^{-\frac{N}{\lambda_1}} \\
&+ \gamma\Big(\frac{t}{\rho^{ps}}\Big)^\frac{1}{2-p}\max\Big\{1,\,\Big(\frac{t}{\rho^{ps}}\Big)^\frac{1-p}{2-p}{\rm Tail}\Big(u,x_0,\frac{\rho}{2},0,t\Big)^{p-1}\Big\}.
\end{align*}
\end{enumroman}
\end{corollary}

\section{Backward $L^r-L^r$ estimate}\label{sec5}

\noindent
In this section we prove Theorem \ref{back}. With this aim in mind, we assume $u$ to be a locally bounded solution of \eqref{eq} s.t.\ $u\ge 0$ in $B_{4\rho}(x_0)\times(0,t)\subset\Omega_T$, and we assume that the right-hand side is finite, i.e
\[
\sup_{0<\tau<t}\,\int_{B_\rho^c(x_0)}u_+^r(x,\tau)\,dx\, + \int_{B_{2\rho}(x_0)}u^r(x,0)\,dx+\Big(\frac{t^r}{\rho^{\lambda_r}}\Big)^\frac{1}{r}\, \, < \infty\,,
\]
otherwise there is nothing to prove. We argue in a dichotomic form, assuming
\beq\label{bal}
\sup_{0<\tau<t}\,\int_{B_\rho(x_0)}u^r(x,\tau)\,dx > \sup_{0<\tau<t}\,\int_{B_\rho^c(x_0)}u_+^r(x,\tau)\,dx.
\eeq
The crucial step is the following technical lemma:

\begin{lemma}\label{bl}
Let \eqref{bal} hold, $\sigma\in(0,1)$. Then, there exists $\gamma>0$ depending on the data and $r$, s.t.\
\begin{align*}
\sup_{0<\tau<t}\,\int_{B_\rho(x_0)}u^r(x,\tau)\,dx &\le \int_{B_{(1+\sigma)\rho}(x_0)}u^r(x,0)\,dx \\
&+ \frac{\gamma}{\sigma^{N+ps}}\Big[\sup_{0<\tau<t}\,\int_{B_{(1+\sigma)\rho}(x_0)}u^r(x,\tau)\,dx\Big]^\frac{p+r-2}{r}\Big(\frac{t^r}{\rho^{\lambda_r}}\Big)^\frac{1}{r}.
\end{align*}
\end{lemma}
\begin{proof}
Fix $t_0\in(0,t)$ and set for brevity
\[B = B_\rho(x_0), \ \tilde B = B_{(1+\sigma/2)\rho}(x_0), \ \hat B = B_{(1+\sigma)\rho}(x_0),\]
so that $B\subset\tilde B\subset\hat B$. Also, let $\xi\in C^\infty_c(\tilde B)$ be a cutoff function s.t.\ $\xi=1$ in $B$, $0\le\xi\le 1$, $|\nabla\xi|\le\gamma/(\sigma\rho)$ in $\R^N$. As in previous cases, the idea is to use as a test function in \eqref{eq}
\[\varphi(x,\tau) = u^{r-1}(x,\tau)\xi^p(x),\]
but this is prevented by several reasons (non-differentiability in time, possible singularity if $r<2$). Nevertheless, by applying a convenient mollification procedure (see Lemma \ref{mob} below), we can find $\gamma>1$ depending on the data and $r$, s.t.\
\begin{align}\label{bl1}
& 0 \ge \frac{1}{r}\int_{\tilde B}u^r(x,\tau)\xi^p(x)\,dx\Big|_0^{t_0} \\
\nonumber &+ \frac{1}{\gamma}\int_0^{t_0}\iint_{\R^N\times\R^N}\frac{|u(x,\tau)-u(y,\tau)|^{p-2}(u(x,\tau)-u(y,\tau))}{|x-y|^{N+ps}}\Big[u^{r-1}(x,\tau)\xi^p(x)-u^{r-1}(y,\tau)\xi^p(y)\Big]\,dx\,dy\,d\tau \\
\nonumber &= I_1+I_2
\end{align}
(note that multiplication by $\xi^p$ selects non-negative values of $u$ in all terms above). We first deal with the evolutive term $I_1$, by exploiting the properties of $\xi$:
\beq\label{bl2}
I_1 \ge \frac{1}{r}\int_B u^r(x,t_0)\,dx-\frac{1}{r}\int_{\tilde B}u^r(x,0)\,dx.
\eeq
We now turn to the diffusive term $I_2$. Set for all $\tau\in(0,t_0)$
\[A_\tau = \big\{(x,y)\in\R^N\times\R^N:\,u(x,\tau)>u(y,\tau)\big\},\]
\[A_\tau^+ = \big\{(x,y)\in A_\tau:\,\xi(x)<\xi(y)\big\}.\]
Note that for all $(x,y)\in A_\tau$ two cases may occur:
\begin{itemize}[leftmargin=1cm]
\item[$(a)$] if $\xi(x)\ge\xi(y)$, then
\[u^{r-1}(x,\tau)\xi^p(x) \ge u^{r-1}(y,\tau)\xi^p(y),\]
in particular the integrand of $I_2$ becomes non-negative;
\item[$(b)$] if $\xi(x)<\xi(y)$ (i.e., $(x,y)\in A_\tau^+$), then we must have $y\in\tilde B$ and
\[u(x,\tau) > u(y,\tau) \ge 0.\]
\end{itemize}
We fix $\tau\in(0,t_0)$ and use symmetry to reduce the space integration domain to $A_\tau$, then we recall $(a)$ above to dismiss positive contributions, and $(b)$ to separate $u^{r-1}$ from the corresponding multiplier $\xi^p$:
\begin{align*}
& \iint_{\R^N\times\R^N}\frac{|u(x,\tau)-u(y,\tau)|^{p-2}(u(x,\tau)-u(y,\tau))}{|x-y|^{N+ps}}\Big[u^{r-1}(x,\tau)\xi^p(x)-u^{r-1}(y,\tau)\xi^p(y)\Big]\,dx\,dy \\
&\ge 2\iint_{A_\tau^+}\frac{(u(x,\tau)-u(y,\tau))^{p-1}}{|x-y|^{N+ps}}\Big[u^{r-1}(x,\tau)\xi^p(x)-u^{r-1}(y,\tau)\xi^p(y)\Big]\,dx\,dy \\
&= 2\iint_{A_\tau^+}\frac{(u(x,\tau)-u(y,\tau))^{p-1}}{|x-y|^{N+ps}}u^{r-1}(x,\tau)(\xi^p(x)-\xi^p(y))\,dx\,dy \\
&+ 2\iint_{A_\tau^+}\frac{(u(x,\tau)-u(y,\tau))^{p-1}}{|x-y|^{N+ps}}(u^{r-1}(x,\tau)-u^{r-1}(y,\tau))\xi^p(y)\,dx\,dy \\
&= I_3+I_4.
\end{align*}
To estimate $I_3$ we use inequality \eqref{liao} with $a=\xi(y)>\xi(x)=b$, $\delta\in(0,1)$ to be determined, and
\[\eps = \delta\,\frac{u(x,\tau)-u(y,\tau)}{u(x,\tau)} \in (0,1),\]
so we get the following pointwise estimate for all $(x,y)\in A_\tau^+$:
\[\xi^p(y)-\xi^p(x) \le \delta\,\frac{u(x,\tau)-u(y,\tau)}{u(x,\tau)}\xi^p(y)+\frac{\gamma}{\delta^{p-1}}\frac{u^{p-1}(x,\tau)}{(u(x,\tau)-u(y,\tau))^{p-1}}(\xi(y)-\xi(x))^p.\]
Reversing the inequality above, and recalling that $u(x,\tau)>0$, we get
\begin{align}\label{bl3}
I_3 &\ge -\delta\iint_{A_\tau^+}\frac{(u(x,\tau)-u(y,\tau))^p}{|x-y|^{N+ps}}u^{r-2}(x,\tau)\xi^p(y)\,dx\,dy \\
\nonumber &- \frac{\gamma}{\delta^{p-1}}\iint_{A_\tau^+}u^{p+r-2}(x,\tau)\frac{(\xi(y)-\xi(x))^p}{|x-y|^{N+ps}}\,dx\,dy.
\end{align}
To estimate $I_4$, we distinguish two cases:
\begin{itemize}[leftmargin=1cm]
\item[$(a)$] if $r\ge 2$, then we apply Lemma \ref{alg}, concatenating \ref{alg1} and \ref{alg2}, with $q=r$, $a=u(x,\tau)>u(y,\tau)=b$, to get the pointwise estimate
\[(u(x,\tau)-u(y,\tau))^{p-1}(u^{r-1}(x,\tau)-u^{r-1}(y,\tau)) \ge \frac{1}{\gamma}(u(x,\tau)-u(y,\tau))^p(u(x,\tau)+u(y,\tau))^{r-2},\]
which in turn produces
\begin{align*}
I_4 &\ge \frac{1}{\gamma}\iint_{A_\tau^+}\frac{(u(x,\tau)-u(y,\tau))^p}{|x-y|^{N+ps}}(u(x,\tau)+u(y,\tau))^{r-2}\xi^p(y)\,dx\,dy \\
&\ge \frac{1}{\gamma}\iint_{A_\tau^+}\frac{(u(x,\tau)-u(y,\tau))^p}{|x-y|^{N+ps}}u^{r-2}(x,\tau)\xi^p(y)\,dx\,dy;
\end{align*}
\item[$(b)$] if $1<r<2$, then we apply Lagrange's rule:
\begin{align*}
I_4 &\ge \frac{1}{\gamma}\iint_{A_\tau^+}\frac{(u(x,\tau)-u(y,\tau))^p}{|x-y|^{N+ps}}\min\big\{u^{r-2}(x,\tau),\,u^{r-2}(y,\tau)\big\}\xi^p(y)\,dx\,dy \\
&= \frac{1}{\gamma}\iint_{A_\tau^+}\frac{(u(x,\tau)-u(y,\tau))^p}{|x-y|^{N+ps}}u^{r-2}(x,\tau)\xi^p(y)\,dx\,dy.
\end{align*}
\end{itemize}
Not that the estimates in $(a)$, $(b)$ above coincide, up to a different constant $\gamma>1$ depending on the data and $r$. Now choose $\delta\in(0,1)$ in \eqref{bl3} so small that the first term can be reabsorbed in the estimate of $I_4$. Subsequently, we dismiss positive contributions to get for all $\tau\in(0,t_0)$
\begin{align*}
& \iint_{\R^N\times\R^N}\frac{|u(x,\tau)-u(y,\tau)|^{p-2}(u(x,\tau)-u(y,\tau))}{|x-y|^{N+ps}}\Big[u^{r-1}(x,\tau)\xi^p(x)-u^{r-1}(y,\tau)\xi^p(y)\Big]\,dx\,dy \\
&\ge \Big(\frac{1}{\gamma}-\delta\Big)\iint_{A_\tau^+}\frac{(u(x,\tau)-u(y,\tau))^p}{|x-y|^{N+ps}}u^{r-2}(x,\tau)\xi^p(y)\,dx\,dy \\
&- \frac{\gamma}{\delta^{p-1}}\iint_{A_\tau^+}u^{p+r-2}(x,\tau)\frac{(\xi(y)-\xi(x))^p}{|x-y|^{N+ps}}\,dx\,dy \\
&\ge -\gamma\iint_{A_\tau^+}u^{p+r-2}(x,\tau)\frac{(\xi(y)-\xi(x))^p}{|x-y|^{N+ps}}\,dx\,dy.
\end{align*}
Next integrate in time and recall that $y\in\tilde B$ for all $(x,y)\in A_\tau^+$:
\begin{align}\label{bl4}
I_2 &\ge -\gamma\int_0^{t_0}\iint_{A_\tau^+}u^{p+r-2}(x,\tau)\frac{(\xi(y)-\xi(x))^p}{|x-y|^{N+ps}}\,dx\,dy\,d\tau \\
\nonumber &\ge -\gamma\int_0^{t_0}\iint_{\hat B\times\tilde B}u^{p+r-2}(x,\tau)\frac{(\xi(y)-\xi(x))^p}{|x-y|^{N+ps}}\,dx\,dy\,d\tau \\
\nonumber &-\gamma\int_0^{t_0}\iint_{\hat B^c\times\tilde B}u_+^{p+r-2}(x,\tau)\frac{(\xi(y)-\xi(x))^p}{|x-y|^{N+ps}}\,dx\,dy\,d\tau \\
\nonumber &= I_5+I_6.
\end{align}
To estimate $I_5$, we apply the properties of $\xi$:
\begin{align*}
I_5 &\ge -\frac{\gamma}{\sigma^p\rho^p}\int_0^{t_0}\iint_{\hat B	\times\tilde B}\frac{u^{p+r-2}(x,\tau)}{|x-y|^{N+ps}}\,dx\,dy\,d\tau \\
&\ge -\frac{\gamma t}{\sigma^p\rho^p}\Big[\sup_{0<\tau<t}\,\int_{\hat B}u^{p+r-2}(x,\tau)\,dx\Big]\Big[\sup_{x\in\hat B}\,\int_{\tilde B}\frac{dy}{|x-y|^{N+ps-p}}\Big] \\
&\ge -\frac{\gamma t}{\sigma^p\rho^p}\Big[\sup_{0<\tau<t}\,\int_{\hat B}u^r(x,\tau)\,dx\Big]^\frac{p+r-2}{r}|\hat B|^\frac{2-p}{r}\Big[\int_{B_{2\rho}(0)}\frac{dz}{|z|^{N+ps-p}}\Big] \\
&\ge -\frac{\gamma}{\sigma^p}\Big[\sup_{0<\tau<t}\,\int_{\hat B}u^r(x,\tau)\,dx\Big]^\frac{p+r-2}{r}\Big(\frac{t^r}{\rho^{\lambda_r}}\Big)^\frac{1}{r}.
\end{align*}
The estimate of $I_6$ is more delicate. First note that for all $(x,y)\in\hat B^c\times\tilde B$
\[|x-y| \ge |x-x_0|-|y-x_0| \ge \frac{\sigma}{2+2\sigma}|x-x_0|,\]
so by boundedness of $\xi$ and weighted H\"older's inequality we have
\begin{align*}
I_6 &\ge -\gamma\Big(\frac{2+2\sigma}{\sigma}\Big)^{N+ps}\int_0^t\iint_{\hat B^c\times\tilde B}\frac{u_+^{p+r-2}(x,\tau)}{|x-x_0|^{N+ps}}\,dx\,dy\,d\tau \\
&\ge -\frac{\gamma t}{\sigma^{N+ps}}|\tilde B|\Big[\sup_{0<\tau<t}\,\int_{\hat B^c}\frac{u_+^{p+r-2}(x,\tau)}{|x-x_0|^{N+ps}}\,dx\Big] \\
&\ge -\frac{\gamma\rho^N t}{\sigma^{N+ps}}\Big[\sup_{0<\tau<t}\,\int_{\hat B^c}u_+^r(x,\tau)\,dx\Big]^\frac{p+r-2}{r}\Big[\int_{\hat B^c}|x-x_0|^{-\frac{(N+ps)r}{2-p}}\,dx\Big]^\frac{2-p}{r} \\
&\ge -\frac{\gamma}{\sigma^{N+ps}}\Big[\sup_{0<\tau<t}\,\int_{B^c}u_+^r(x,\tau)\,dx\Big]^\frac{p+r-2}{r}\Big(\frac{t^r}{\rho^{\lambda_r}}\Big)^\frac{1}{r}.
\end{align*}
Now we recall \eqref{bal} to see that
\[I_6 \ge -\frac{\gamma}{\sigma^{N+ps}}\Big[\sup_{0<\tau<t}\,\int_{B}u^r(x,\tau)\,dx\Big]^\frac{p+r-2}{r}\Big(\frac{t^r}{\rho^{\lambda_r}}\Big)^\frac{1}{r}.\]
Plug these estimates into \eqref{bl4}, recalling that $\sigma^p>\sigma^{N+ps}$ and $B\subset\hat B$, to find the following estimate of the diffusive term:
\beq\label{bl5}
I_2 \ge -\frac{\gamma}{\sigma^{N+ps}}\Big[\sup_{0<\tau<t}\,\int_{\hat B}u^r(x,\tau)\,dx\Big]^\frac{p+r-2}{r}\Big(\frac{t^r}{\rho^{\lambda_r}}\Big)^\frac{1}{r}.
\eeq
Finally, we concatenate \eqref{bl1} with \eqref{bl2} and \eqref{bl5}:
\[\frac{1}{r}\int_B u^r(x,t_0)\,dx \le \frac{1}{r}\int_{\hat B}u^r(x,0)\,dx-\frac{\gamma}{\sigma^{N+ps}}\Big[\sup_{0<\tau<t}\,\int_{\hat B}u^r(x,\tau)\,dx\Big]^\frac{p+r-2}{r}\Big(\frac{t^r}{\rho^{\lambda_r}}\Big)^\frac{1}{r},\]
which yields the conclusion as soon as we multiply by $r$ and take the supremum over $t_0\in(0,t)$.
\end{proof}

\noindent
We can now complete the proof of the main result:
\vskip4pt
\noindent
{\em Proof of Theorem \ref{back} (conclusion).} We assume that \eqref{bal}, otherwise there is nothing to prove. We are going to perform an iteration, setting $\sigma_0=0$, $\rho_0=\rho$, and for all $n\ge 1$
\[\sigma_n = \sum_{i=1}^n\frac{1}{2^i}, \ \rho_n = (1+\sigma_n)\rho,\]
so that $(\sigma_n)$, $(\rho_n)$ are increasing with $\sigma_n\to 1$, $\rho_n\to 2\rho$. Besides, we set $B_n=B_{\rho_n}(x_0)$, hence $B_n\subset B_{n+1}$. We fix $n\ge 0$ and apply Lemma \ref{bl} with $\rho=\rho_n$ and
\[\sigma = \frac{\sigma_{n+1}-\sigma_n}{1+\sigma_n} \in (0,1),\]
implying $(1+\sigma)\rho=\rho_{n+1}$. Note that, by \eqref{bal} and the inclusion $B_0\subset B_n$, we have
\begin{align*}
\sup_{0<\tau<t}\,\int_{B_n}u^r(x,\tau)\,dx &\ge \sup_{0<\tau<t}\,\int_{B_0}u^r(x,\tau)\,dx \\
&\ge \sup_{0<\tau<t}\,\int_{B_0^c}u^r(x,\tau)\,dx \ge \sup_{0<\tau<t}\,\int_{B_n^c}u^r(x,\tau)\,dx.
\end{align*}
Starting from the inequality of Lemma \ref{bl}, and exploiting the definitions above, we get
\begin{align}\label{back1}
\sup_{0<\tau<t}\,\int_{B_n}u^r(x,\tau)\,dx &\le \int_{B_{n+1}}u^r(x,0)\,dx+\gamma\Big(\frac{1+\sigma_n}{\sigma_{n+1}-\sigma_n}\Big)^{N+ps}\Big[\sup_{0<\tau<t}\,\int_{B_{n+1}}u^r(x,\tau)\,dx\Big]^\frac{p+r-2}{r}\Big(\frac{t^r}{\rho_n^{\lambda_r}}\Big)^\frac{1}{r} \\
\nonumber &\le \Big[\int_{B_{n+1}}u^r(x,0)\,dx\Big]^\frac{p+r-2}{r}\Big[\int_{B_{2\rho}(x_0)}u^r(x,0)\,dx\Big]^\frac{2-p}{r} \\
\nonumber &+ \gamma 2^{(N+ps)n}\Big[\sup_{0<\tau<t}\,\int_{B_{n+1}}u^r(x,\tau)\,dx\Big]^\frac{p+r-2}{r}\Big(\frac{t^r}{\rho^{\lambda_r}}\Big)^\frac{1}{r} \\
\nonumber &\le \gamma 2^{(N+ps)n}\Big[\sup_{0<\tau<t}\,\int_{B_{n+1}}u^r(x,\tau)\,dx\Big]^\frac{p+r-2}{r}\Big\{\Big[\int_{B_{2\rho}(x_0)}u^r(x,0)\,dx\Big]^\frac{2-p}{r}+\Big(\frac{t^r}{\rho^{\lambda_r}}\Big)^\frac{1}{r}\Big\},
\end{align}
with $\gamma>0$ independent of $n$. We are going to apply Lemma \ref{int} to the sequence
\[Y_n = \sup_{0<\tau<t}\,\int_{B_n}u^r(x,\tau)\,dx,\]
which is a bounded sequence of non-negative real numbers by assumption $u\in L^r(\R^N\times(0,T))$. We also set
\[b = 2^{N+ps}, \ \eta = \frac{2-p}{r}, \ C= \gamma\Big\{\Big[\int_{B_{2\rho}(x_0)}u^r(x,0)\,dx\Big]^\frac{2-p}{r}+\Big(\frac{t^r}{\rho^{\lambda_r}}\Big)^\frac{1}{r}\Big\},\]
so $b>1$, $0<\eta<1$, $C>0$, and \eqref{back1} implies for all $n\in\N$
\[Y_n \le Cb^nY_{n+1}^{1-\eta}.\]
Therefore, we have
\[Y_0 \le \big(2Cb^\frac{1-\eta}{\eta}\big)^\frac{1}{\eta} \le \gamma\int_{B_{2\rho}(x_0)}u^r(x,0)\,dx+\gamma\Big(\frac{t^r}{\rho^{\lambda_r}}\Big)^\frac{1}{2-p},\]
for some $\gamma>0$ depending on the data and $r$. In view of \eqref{bal}, we conclude. \qed

\section{Extinction time and decay estimates}\label{sec6}

\noindent
In this section we deal with solutions of problem \eqref{cdp}, set in a bounded domain $\Omega$ with a positive initial datum $u_0\in W^{s,p}_0(\Omega)$. First we prove that any such solutions has a finite time of extinction, estimated in terms of the initial datum and possibly the domain's volume:
\vskip4pt
\noindent
{\em Proof of Theorem \ref{fte}.} First we consider case \ref{fte1}, i.e., $1<p<p_c$. Then, an elementary calculation shows that
\[q = \frac{N(2-p)}{ps} > 2.\]
The first step consists in testing problem \eqref{cdp} with $u^{q-1}$, which of course is forbidden by the lack of time-differentiability of $u$. Choosing a convenient test function in Definition \ref{scp} and applying a mollification procedure (see Lemma \ref{mot} \ref{mot1}), we can prove that there exists $\gamma>0$ depending on $N,p,s$ s.t.\ for all $\psi\in W^{1,2}_0(0,T)$, $\psi\ge 0$ in $(0,T)$
\beq\label{fte3}
\int_0^T\Big[-\|u(\cdot,\tau)\|_{L^q(\Omega)}^q\psi'(\tau)+\frac{C_1}{\gamma}\|u(\cdot,\tau)\|_{L^q(\Omega)}^{p+q-2}\psi(\tau)\Big]\,d\tau \le 0,
\eeq
where $C_1$ is as in $(K_2)$. We rephrase \eqref{fte3} as follows. Set for all $t\in(0,T)$
\[U(t) = \|u(\cdot,t)\|_{L^q(\Omega)}^q.\] Since the Sobolev weak time-derivative of $t \rightarrow \|u(\cdot, t)\|_{L^q(\Omega)}^q$ is bounded by \eqref{fte3}, the function $U$ is absolutely continuous on $[0,T]$ (up to the choice of a representative). 
\noindent Hence, from \eqref{fte3} we see that $U$ satisfies the following ordinary differential inequality:
\beq\label{fte4}
\begin{cases}
\displaystyle U'(t)+\frac{C_1}{\gamma}U^\frac{p+q-2}{q}(t) \le 0 & \text{a.e.\ in $(0,T)$} \\
U(0) = \|u_0\|_{L^q(\Omega)}^q.
\end{cases}
\eeq
The mapping
\[t\mapsto U^\frac{2-p}{q}(t) = U^\frac{ps}{N}(t)\]
is as well a.e.\ differentiable. So, integrating on \eqref{fte4}, we have for all $t\in(0,T)$
\begin{align*}
U^\frac{2-p}{q}(t)-U^\frac{2-p}{q}(0) &= \frac{2-p}{q}\int_0^t U^\frac{2-p-q}{q}(\tau)U'(\tau)\,d\tau \\
&\le -\frac{2-p}{q}\frac{C_1}{\gamma}\int_0^1 1\,d\tau = -\frac{C_1t}{\gamma_*},
\end{align*}
with $\gamma_*>0$ depending on $N,p,s$. Equivalently, we have
\[\|u(\cdot,t)\|_{L^q(\Omega)}^{2-p} \le \|u_0\|_{L^q(\Omega)}^{2-p}-\frac{C_1t}{\gamma_*},\]
which implies that $u(\cdot,t)$ vanishes a.e.\ in $\Omega$ as soon as
\[t \ge \frac{\gamma_*}{C_1}\|u_0\|_{L^q(\Omega)}^{2-p}.\]
Recalling that $u$ vanishes in $\Omega^c$ at any time, we obtain \ref{fte1}.
\vskip2pt
\noindent
We now turn to case \ref{fte2}. The proof goes essentially as in the previous case, except we replace the exponent $q$ with $2$. Note, in this connection, that $p\ge p_c$ implies both $\lambda_2\ge 0$ and $p^*_s\ge 2$. By H\"older's inequality, then, we have for all $t\in(0,T)$
\[\|u(\cdot,t)\|_{L^2(\Omega)} \le |\Omega|^\frac{\lambda_2}{2Np}\|u(\cdot,t)\|_{L^{p^*_s}(\Omega)}.\]
Through a convenient testing procedure (see Lemma \ref{mot} \ref{mot2}), we find $\gamma>0$ depending on $N,p,s$ s.t.\ for all $\psi\in W^{1,2}_0(0,T)$, $\psi\ge 0$ in $(0,T)$
\beq\label{fte5}
\int_0^T\Big[-\|u(\cdot,\tau)\|_{L^2(\Omega)}^2\psi'(\tau)+\frac{C_1}{\gamma}|\Omega|^{-\frac{\lambda_2}{2N}}\|u(\cdot,\tau)\|_{L^2(\Omega)}^p\psi(\tau)\Big]\,d\tau \le 0.
\eeq
Reasoning as in the previous case, we set for all $t\in(0,T)$
\[U(t) = \|u(\cdot,t)\|_{L^2(\Omega)}^2,\]
so again $U$ is absolutely continuous in $[0,T]$ a.e.\ differentiable, and satisfies the ordinary differential inequality
\beq\label{fte6}
\begin{cases}
\displaystyle U'(t)+\frac{C_1}{\gamma}|\Omega|^{-\frac{\lambda_2}{2N}}U^\frac{p}{2}(t) \le 0 & \text{a.e.\ in $(0,T)$} \\
U(0) = \|u_0\|_{L^2(\Omega)}^2.
\end{cases}
\eeq
In turn, integrating on \eqref{fte6}, we get for all $t\in(0,T)$
\begin{align*}
U^\frac{2-p}{2}(t)-U^\frac{2-p}{2}(0) &= \frac{2-p}{2}\int_0^t U^{-\frac{p}{2}}(\tau)U'(\tau)\,d\tau \\
&\le -\frac{2-p}{2}\frac{C_1}{\gamma}|\Omega|^{-\frac{\lambda_2}{2N}}\int_0^t 1\,d\tau = -\frac{C_1 t}{\gamma_*}|\Omega|^{-\frac{\lambda_2}{2N}},
\end{align*}
with $\gamma_*>0$ depending on $N,p,s$. The latter inequality rephrases as
\[\|u(\cdot,t)\|_{L^2(\Omega)}^{2-p} \le \|u_0\|_{L^2(\Omega)}^{2-p}-\frac{C_1 t}{\gamma_*}|\Omega|^{-\frac{\lambda_2}{2N}}.\]
Finally, we have $u(\cdot,t)=0$ in $\Omega$ as soon as
\[t \ge \frac{\gamma_*}{C_1}\|u_0\|_{L^2(\Omega)}^{2-p}|\Omega|^\frac{\lambda_2}{2N},\]
which yields the conclusion since $u$ vanishes in $\Omega^c$ at any time. \qed
\vskip4pt
\noindent
Finally, we estimate the rate of extinction of $u$ as time approaches $T_*$.
\vskip4pt
\noindent
{\em Proof of Theorem \ref{dec}.} We may see $u$ as a non-negative, local weak solution of \eqref{eq} in $\Omega_{T_*}$, which in addition satisfies $u(\cdot,t)=0$ in $\R^N$ for all $t\ge T_*$. Up to a time translation, we apply Corollary \ref{pos} \ref{pos1} to $u$ in the time interval $(t,T_*)$. Therefore, there exists $\gamma>0$ depending on the data s.t.\
\begin{align*}
\int_{B_\rho(x_0)}u(x,t)\,dx &\le \sup_{t<\tau<T_*}\,\int_{B_\rho(x_0)}u(x,\tau)\,dx \\
&\le \gamma\Big[\inf_{t<\tau<T_*}\,\int_{B_{2\rho}(x_0)}u(x,\tau)\,dx\Big]+\gamma\Big(\frac{T_*-t}{\rho^{\lambda_1}}\Big)^\frac{1}{2-p} \\
&= \gamma\Big(\frac{T_*-t}{\rho^{\lambda_1}}\Big)^\frac{1}{2-p},
\end{align*}
which proves \ref{dec1}.
\vskip2pt
\noindent
Now assume $p_*<p<2$, in particular $\lambda_1>0$. Set
\[P = \max\Big\{1,\,\Big(\frac{T_*-t}{\rho^{ps}}\Big)^\frac{1-p}{2-p}{\rm Tail}\Big(u,x_0,\frac{\rho}{2},t,T_*\Big)^{p-1}\Big\} > 0.\]
By Corollary \ref{pos} \ref{pos2}, again applied in the time interval $(t,T_*)$, and Theorem \ref{fte}, there exists $\gamma>0$ depending on the data s.t.\
\begin{align*}
\sup_{B_{\rho/2}(x_0)}\,u\Big(\cdot,\frac{t+T_*}{2}\Big) &\le \sup_{B_{\rho/2}(x_0)\times((t+T_*)/2,T_*)}\,u \\
&\le \gamma\Big[\inf_{t<\tau<T_*}\int_{B_{2\rho}(x_0)}u(x,\tau)\,dx\Big]^\frac{ps}{\lambda_1}(T_*-t)^{-\frac{N}{\lambda_1}}+\gamma\Big(\frac{T_*-t}{\rho^{ps}}\Big)^\frac{1}{2-p}P \\
&= \gamma\Big(\frac{T_*-t}{\rho^{ps}}\Big)^\frac{1}{2-p}P,
\end{align*}
which proves \ref{dec2}. Thus, the proof is concluded. \qed

\begin{remark}\label{dec1inf}
Clearly, the estimate of Theorem \ref{dec} \ref{dec2} is sharper than that in \ref{dec1}, up to a rescaling of the cylinder $B_\rho(x_0)\times(t,T_*)$ and under appropriate conditions on the tail (recall that $u$ vanishes in $\Omega^c$). Nevertheless, the fist estimate holds even in the subcritical regime $1<p\le p_*$ and does not require additional assumptions.
\end{remark}

\appendix

\section{Time mollifications}\label{app}

\noindent
At several steps in our proofs, namely in the proofs of Lemma \ref{tst}, Lemma \ref{bl}, and Theorem \ref{fte}, we have introduced test functions that are not admissible according to Definitions \ref{lws} and \ref{scp}, respectively. Such procedure is widely known in parabolic regularity theory, and in our case it is formally justified by the following mollification argument.
\vskip2pt
\noindent
For all $v:\Omega_T\to\R$ and all $h\in(0,1)$, we define two finite convolutions of $v$ with exponential weight functions by setting for all $(x,t)\in\Omega_T$
\[v_h(x,t) = \int_0^t\frac{ v(x,\tau)}{h}e^\frac{\tau-t}{h}\,d\tau, \ v_{\bar h}(x,t) = \int_t^T \frac{v(x,\tau)}{h}e^\frac{t-\tau}{h}\,d\tau.\]
We refer to \cite[Appendix B]{L} and \cite[Lemma 2.2]{KL} for the following properties of $v_h$, $v_{\bar h}$:

\begin{proposition}\label{con}
Let $v\in L^q(\Omega_T)$ for some $q\in[1,+\infty[$. Then, for all $h\in(0,1)$ we have:
\begin{enumroman}
\item\label{con1} $v_h,v_{\bar h}$ are differentiable in $t$ with
\[\frac{\partial v_h}{\partial t}(x,t) = \frac{v(x,t)-v_h(x,t)}{h}, \ \frac{\partial v_{\bar h}(x,t)}{\partial t} = \frac{v_{\bar h}(x,t)-v(x,t)}{h},\]
in particular for all $\varphi\in W^{1,2}(0,T,L^{q'}(\Omega))$ s.t.\ $\varphi(\cdot,0)=\varphi(\cdot,T)=0$ we have
\[\int_0^Tv_h(x,\tau)\varphi_\tau(x,\tau)\,d\tau = -\int_0^T\frac{v(x,\tau)-v_h(x,\tau)}{h}\varphi(x,\tau)\,d\tau,\]
\[\int_0^Tv_{\bar h}(x,\tau)\varphi_\tau(x,\tau)\,d\tau = -\int_0^t\frac{v_{\bar h}(x,\tau)-v(x,\tau)}{h}\varphi(x,\tau)\,d\tau;\]
\item\label{con2} $v_h,v_{\bar h}\in L^q(\Omega_T)$ and $v_h,v_{\bar h}\to v$ in $L^q(\Omega_T)$ as $h\to 0^+$, also if $v\ge 0$ in $\Omega_T$ then $0\le v_h,v_{\bar h}\le v$ in $\Omega_T$;
\item\label{con3} if $v\in L^q(0,T,W^{s,q}(\Omega'))$ for some $\Omega'\Subset\Omega$, then $v_h,v_{\bar h}\in L^q(0,T,W^{s,q}(\Omega'))$ with
\[\int_0^T\iint_{\Omega'\times\Omega'}\frac{|v_h(x,\tau)-v_h(y,\tau)|^q}{|x-y|^{N+qs}}\,dx\,dy\,d\tau \le \|v\|_{L^q(0,T,W^{s,q}(\Omega'))}^q,\]
and an analogous estimate holds for $v_{\bar h}$.
\end{enumroman}
\end{proposition}

\noindent
The following lemma provides a rigorous proof of inequality \eqref{tst1} in Lemma \ref{tst}:

\begin{lemma}\label{mol}
Let $u$, $\check B\subset\hat B\subset B$, $t$, $\nu$, $\xi$ be as in Lemma \ref{tst}. Then, there exists $\gamma>0$ depending on the data s.t.\
\begin{align*}
0 &\ge -\frac{pt^\frac{1}{p}}{2(p-1)}\int_B (u(x,t)+\nu)^\frac{2(p-1)}{p}\xi^p(x)\,dx+\frac{1}{2(p-1)}\int_0^t \tau^\frac{1-p}{p}\int_B(u(x,\tau)+\nu)^\frac{2(p-1)}{p}\xi^p(x)\,dx\,d\tau \\
\nonumber &- \gamma\int_0^t \tau^\frac{1}{p}\iint_{\R^N\times\R^N}\frac{|u(x,\tau)-u(y,\tau)|^{p-2}(u(x,\tau)-u(y,\tau))}{|x-y|^{N+ps}}\,\Big[(u(x,\tau)+\nu)^\frac{p-2}{p}\xi^p(x)-(u(y,\tau)+\nu)^\frac{p-2}{p}\xi^p(y)\Big]\,dx\,dy\,d\tau.
\end{align*}
\end{lemma}
\begin{proof}
Fix $\eps\in(0,t/2)$ and set for all $\tau\in[0,T]$
\[\psi_\eps(\tau) = \begin{cases}
\tau/\eps & \text{if $0\le\tau<\eps$} \\
1 & \text{if $\eps\le\tau<t-\eps$} \\
(t-\tau)/\eps & \text{if $t-\eps\le\tau<t$} \\
0 & \text{if $t\le\tau\le T$,}
\end{cases}\]
so $\psi_\eps:[0,T]\to[0,1]$ is a Lipschitz mapping. Also, fix $h\in(0,1)$ and set for all $(x,\tau)\in\R^N\times(0,T)$
\[\varphi^\eps_h(x,\tau) = -(u_{\bar h}(x,\tau)+\nu)^\frac{p-2}{p}\xi^p(x)\tau^\frac{1}{p}\psi_\eps(\tau).\]
Note that the negative exponent in the first term does not produce any singularity as $u(\cdot,\tau)\ge 0$ in $\hat B$. By Proposition \ref{con}, $\varphi^\eps_h\in W^{1,2}_{\rm loc}(0,T,L^2(\hat B))\cap L^p_{\rm loc}(0,T,W^{s,p}_0(\hat B))$ is a suitable test function for \eqref{eq}, in $\hat B\times(0,t)$ (see Definition \ref{lws}), so we have
\begin{align}\label{mol1}
0 &= -\int_0^t\int_{\hat B} u(x,\tau)\frac{\partial\varphi^\eps_h}{\partial\tau}(x,\tau)\,dx\,d\tau \\
\nonumber &+ \int_0^t\iint_{\R^N\times\R^N}|u(x,\tau)-u(y,\tau)|^{p-2}(u(x,\tau)-u(y,\tau))(\varphi^\eps_h(x,\tau)-\varphi^\eps_h(y,\tau))K(x,y,\tau)\,dx\,dy\,d\tau \\
\nonumber &= H_1+H_2,
\end{align}
where we have also used that $\varphi^\eps_h(\cdot,0)=\varphi^\eps_h(\cdot,t)=0$ in $\R^N$ and $\xi=0$ in $\hat B^c$. Our aim is now to estimate both $H_1$, $H_2$ from below, and then pass to the limit as $\eps,h\to 0$. We first focus on the evolutive term $H_1$:
\begin{align*}
H_1 &= -\int_0^T\int_{\hat B} u_{\bar h}(x,\tau)\frac{\partial\varphi^\eps_h}{\partial\tau}(x,\tau)\,dx\,d\tau+\int_0^T\int_{\hat B}(u_{\bar h}(x,\tau)-u(x,\tau))\frac{\partial\varphi^\eps_h}{\partial\tau}(x,\tau)\,dx\,d\tau \\
&= H_3+H_4.
\end{align*}
On $H_3$ we first apply Proposition \ref{con} \ref{con1}, then the chain rule of differentiation and the definition of $\psi_\eps:$
\begin{align*}
H_3 &= \int_0^T\int_{\hat B}\frac{\partial u_{\bar h}}{\partial\tau}(x,\tau)\varphi_h^\eps(x,\tau)\,dx\,d\tau \\
&= -\int_0^T\int_{\hat B}\frac{\partial}{\partial\tau}(u_{\bar h}(x,\tau)+\nu)(u_{\bar h}(x,\tau)+\nu)^\frac{p-2}{p}\xi^p(x)\tau^\frac{1}{p}\psi_\eps(\tau)\,dx\,d\tau \\
&= -\frac{p}{2(p-1)}\int_0^T\int_{\hat B}\frac{\partial}{\partial\tau}(u_{\bar h}(x,\tau)+\nu)^\frac{2(p-1)}{p}\xi^p(x)\tau^\frac{1}{p}\psi_\eps(\tau)\,dx\,d\tau \\
&= \frac{p}{2(p-1)}\int_0^T\int_{\hat B}(u_{\bar h}(x,\tau)+\nu)^\frac{2(p-1)}{p}\xi^p(x)\frac{\partial}{\partial\tau}(\tau^\frac{1}{p}\psi_\eps(\tau))\,dx\,d\tau \\
&= \frac{1}{2(p-1)}\int_0^T\int_{\hat B}(u_{\bar h}(x,\tau)+\nu)^\frac{2(p-1)}{p}\xi^p(x)\tau^\frac{1-p}{p}\psi_\eps(\tau)\,dx\,d\tau \\
&+ \frac{p}{2(p-1)}\int_0^T\int_{\hat B}(u_{\bar h}(x,\tau)+\nu)^\frac{2(p-1)}{p}\xi^p(x)\tau^\frac{1}{p}\psi'_\eps(\tau)\,dx\,d\tau \\
&= \frac{1}{2(p-1)}\int_0^t\int_{\hat B}(u_{\bar h}(x,\tau)+\nu)^\frac{2(p-1)}{p}\xi^p(x)\tau^\frac{1-p}{p}\psi_\eps(\tau)\,dx\,d\tau \\
&+ \frac{p}{2(p-1)}\Big[\frac{1}{\eps}\int_0^\eps\tau^\frac{1}{p}\int_{\hat B}(u_{\bar h}(x,\tau)+\nu)^\frac{2(p-1)}{p}\xi^p(x)\,dx\,d\tau-\frac{1}{\eps}\int_{t-\eps}^t\tau^\frac{1}{p}\int_{\hat B}(u_{\bar h}(x,\tau)+\nu)^\frac{2(p-1)}{p}\xi^p(x)\,dx\,d\tau\Big].
\end{align*}
First let $h\to 0$ and use Proposition \ref{con} \ref{con2} with $q=1$ (note that $2(p-1)/p<1$):
\begin{align*}
\lim_{h\to 0}\,H_3 &= \frac{1}{2(p-1)}\int_0^t\int_{\hat B}(u(x,\tau)+\nu)^\frac{2(p-1)}{p}\xi^p(x)\tau^\frac{1-p}{p}\psi_\eps(\tau)\,dx\,d\tau \\
&+ \frac{p}{2(p-1)}\Big[\frac{1}{\eps}\int_0^\eps\tau^\frac{1}{p}\int_{\hat B}(u(x,\tau)+\nu)^\frac{2(p-1)}{p}\xi^p(x)\,dx\,d\tau-\frac{1}{\eps}\int_{t-\eps}^t\tau^\frac{1}{p}\int_{\hat B}(u(x,\tau)+\nu)^\frac{2(p-1)}{p}\xi^p(x)\,dx\,d\tau\Big].
\end{align*}
Further, let $\eps\to 0$ and apply the fundamental theorem of calculus:
\begin{align*}
\lim_{\eps,h\to 0}\,H_3 &= \frac{1}{2(p-1)}\int_0^t\int_{\hat B}(u(x,\tau)+\nu)^\frac{2(p-1)}{p}\xi^p(x)\tau^\frac{1-p}{p}\,dx\,d\tau \\
&- \frac{pt^\frac{1}{p}}{2(p-1)}\int_{\hat B}(u(x,t)+\nu)^\frac{2(p-1)}{p}\xi^p(x)\,dx.
\end{align*}
Next, on $H_4$ we apply Proposition \ref{con} \ref{con1} and positivity of $u_{\bar h}$:
\begin{align*}
H_4 &= \frac{2-p}{p}\int_0^T\int_{\hat B}(u_{\bar h}(x,\tau)-u(x,\tau))(u_{\bar h}(x,\tau)+\nu)^{-\frac{2}{p}}\frac{\partial u_{\bar h}}{\partial\tau}(x,\tau)\tau^\frac{1}{p}\psi_\eps(\tau)\xi^p(x)\,dx\,d\tau \\
&- \int_0^T\int_{\hat B}(u_{\bar h}(x,\tau)-u(x,\tau))(u_{\bar h}(x,\tau)+\nu)^\frac{p-2}{p}\frac{\partial}{\partial\tau}(\tau^\frac{1}{p}\psi_\eps(\tau))\xi^p(x)\,dx\,d\tau \\
&\ge \frac{2-p}{p}\int_0^T\int_{\hat B}\frac{(u_{\bar h}(x,\tau)-u(x,\tau))^2}{h}(u_{\bar h}(x,\tau)+\nu)^{-\frac{2}{p}}\tau^\frac{1}{p}\psi_\eps(\tau)\xi^p(x)\,dx\,d\tau \\
&- \int_0^T\int_{\hat B}|u_{\bar h}(x,\tau)-u(x,\tau)|\nu^\frac{p-2}{p}\Big|\frac{\partial}{\partial\tau}(\tau^\frac{1}{p}\psi_\eps(\tau))\Big|\xi^p(x)\,dx\,d\tau.
\end{align*}
The first term above is non-negative, while the second tends to $0$ as $h\to 0$, uniformly for all $\eps\in(0,t/2)$, by Proposition \ref{con} \ref{con2} (with $q=1$), so we have
\[\liminf_{\eps,h\to 0}\, H_4 \ge 0.\]
By the previous estimates, and recalling that $\xi=0$ in $\hat B^c$, we have
\begin{align}\label{mol2}
\liminf_{\eps,h\to 0}\,H_1 &\ge \frac{1}{2(p-1)}\int_0^t\tau^\frac{1-p}{p}\int_B (u(x,\tau)+\nu)^\frac{2(p-1)}{p}\xi^p(x)\,dx\,d\tau \\
\nonumber &- \frac{pt^\frac{1}{p}}{2(p-1)}\int_B (u(x,t)+\nu)^\frac{2(p-1)}{p}\xi^p(x)\,dx.
\end{align}
We now turn to the diffusive term $H_2$. For all $(x,y,\tau)\in\R^N\times\R^N\times(0,T)$ set
\[G(x,y,\tau) = |u(x,\tau)-u(y,\tau)|^{p-2}(u(x,\tau)-u(y,\tau)),\]
\[F_h(x,y,\tau) =\tau^\frac{1}{p}\Big[(u_{\bar h}(x,\tau)+\nu)^\frac{p-2}{p}\xi^p(x)-(u_{\bar h}(y,\tau)+\nu)^\frac{p-2}{p}\xi^p(y)\Big],\]
\[F(x,y,\tau) =\tau^\frac{1}{p}\Big[(u(x,\tau)+\nu)^\frac{p-2}{p}\xi^p(x)-(u(y,\tau)+\nu)^\frac{p-2}{p}\xi^p(y)\Big].\]
First note that
\beq\label{mol3}
\int_0^t\iint_{B\times B}\frac{|G(x,y,\tau)|^\frac{p}{p-1}}{|x-y|^{N+ps}}\,dx\,dy\,d\tau \le \|u\|_{L^p(0,t,W^{s,p}(B))}^p.
\eeq
By Proposition \ref{con} \ref{con2} we have for a.e.\ $(x,y,\tau)\in\R^N\times\R^N\times(0,T)$
\[\lim_{h\to 0}\,F_h(x,y,\tau) = F(x,y,\tau).\]
We first note that, letting $\eps\to 0$ and using the properties of the kernel $K$, we have
\begin{align*}
\lim_{\eps\to 0}\,H_2 &= -\int_0^t\iint_{\R^N\times\R^N}G(x,y,\tau)F_h(x,y,\tau)K(x,y,\tau)\,dx\,dy\,d\tau \\
&\ge -\gamma\int_0^t\iint_{\R^N\times\R^N}\frac{G(x,y,\tau)F_h(x,y,\tau)}{|x-y|^{N+ps}}\,dx\,dy\,d\tau.
\end{align*}
We now claim that
\beq\label{mol4}
\lim_{h\to 0}\,\int_0^t\iint_{\R^N\times\R^N}\frac{G(x,y,\tau)F_h(x,y,\tau)}{|x-y|^{N+ps}}\,dx\,dy\,d\tau = \int_0^t\iint_{\R^N\times\R^N}\frac{G(x,y,\tau)F(x,y,\tau)}{|x-y|^{N+ps}}\,dx\,dy\,d\tau.
\eeq
Exploiting symmetry and recalling that $\xi=0$ in $\hat B^c$, we can split the difference as follows:
\begin{align*}
& \int_0^t\iint_{\R^N\times\R^N}\frac{G(x,y,\tau)}{|x-y|^{N+ps}}\big[F_h(x,y,\tau)-F(x,y,\tau)\big]\,dx\,dy\,d\tau \\
&= \int_0^t\iint_{B\times B}\frac{G(x,y,\tau)}{|x-y|^{N+ps}}\big[F_h(x,y,\tau)-F(x,y,\tau)\big]\,dx\,dy\,d\tau \\
&+ 2\int_0^t \tau^\frac{1}{p}\iint_{\hat B\times B^c}\frac{G(x,y,\tau)}{|x-y|^{N+ps}}\big[(u_{\bar h}(x,\tau)+\nu)^\frac{p-2}{p}-(u(x,\tau)+\nu)^\frac{p-2}{p}\big]\xi^p(x)\,dx\,dy\,d\tau \\
&= H_5+H_6.
\end{align*}
We show separately that both $H_5$, $H_6$ tend to $0$ as $h\to 0$, up to a subsequence. First we deal with $H_5$. By subadditivity and $0\le\xi\le 1$ in $\R^N$ we have
\begin{align*}
\int_0^t\iint_{B\times B}\frac{|F_h(x,y,\tau)|^p}{|x-y|^{N+ps}}\,dx\,dy\,d\tau &\le \gamma\int_0^t\tau\iint_{B\times B}\frac{\big|(u_{\bar h}(x,\tau)+\nu)^\frac{p-2}{p}-(u_{\bar h}(y,\tau)+\nu)^\frac{p-2}{p}\big|^p}{|x-y|^{N+ps}}\,dx\,dy\,d\tau \\
&+ \gamma\int_0^t\tau\iint_{B\times B}(u_{\bar h}(y,\tau)+\nu)^{p-2}\frac{|\xi^p(x)-\xi^p(y)|^p}{|x-y|^{N+ps}}\,dx\,dy\,d\tau = H_7+H_8.
\end{align*}
To estimate $H_7$ we apply Lagrange's rule, inequality $\tau\le t$, and positivity of $u_{\bar h}$, then Proposition \ref{con} \ref{con3}:
\begin{align*}
H_7 &\le \gamma t\int_0^t\iint_{B\times B}\max\big\{(u_{\bar h}(x,\tau)+\nu)^{-\frac{2}{p}},\,(u_{\bar h}(x,\tau)+\nu)^{-\frac{2}{p}}\big\}^p\frac{|u_{\bar h}(x,\tau)-u_{\bar h}(x,\tau)|^p}{|x-y|^{N+ps}}\,dx\,dy\,d\tau \\
&\le \frac{\gamma t}{\nu^2}\int_0^t\iint_{B\times B}\frac{|u_{\bar h}(x,\tau)-u_{\bar h}(x,\tau)|^p}{|x-y|^{N+ps}}\,dx\,dy\,d\tau \\
&\le \frac{\gamma t}{\nu^2}\|u\|_{L^p(0,t,W^{s,p}(B))}^p.
\end{align*}
For $H_8$ we argue similarly, using Lagrange's rule on the increment of $\xi^p$ and Proposition \ref{con} \ref{con2} (taking this time $q=p$):
\begin{align*}
H_8 &\le \gamma t\int_0^t\iint_{B\times B}(u_{\bar h}(y,\tau)+\nu)^{p-2}\max\big\{\xi^{p-1}(x),\,\xi^{p-1}(y)\big\}^p\frac{|\xi(x)-\xi(y)|^p}{|x-y|^{N+ps}}\,dx\,dy\,d\tau \\
&\le \frac{\gamma t}{\nu^2\rho^p}\Big[\int_0^t\int_B(u_{\bar h}(y,\tau)+\nu)^p\,dy\,d\tau\Big]\Big[\sup_{y\in B}\,\int_B\frac{dx}{|x-y|^{N+ps-p}}\Big] \\
&\le \frac{\gamma t}{\nu^2\rho^p}\Big[\int_0^t\int_B u_{\bar h}^p(y,\tau)\,dy\,d\tau+\nu^p\rho^Nt\Big]\Big[\int_{B_{2\rho}(0)}\frac{dz}{|z|^{N+ps-p}}\Big] \\
&\le \frac{\gamma t}{\nu^2\rho^{ps}}\Big[\|u\|_{L^p(B\times(0,t))}^p+\nu^p \rho^N t\Big].
\end{align*}
Thus, we have found a constant $\gamma>0$ depending on $u$, $\nu$, $\rho$, and $t$, but not on $h$, s.t.\ for all $h\in(0,1)$
\[\int_0^t\iint_{B\times B}\frac{|F_h(x,y,\tau)|^p}{|x-y|^{N+ps}}\,dx\,dy\,d\tau \le \gamma.\]
A standard argument based on reflexivity, and exploiting \eqref{mol3}, shows that up to passing to a sequence of $h$'s converging to $0$, we have
\beq\label{mol5}
\lim_{h\to 0}\, H_5 = 0.
\eeq
Next we focus on $H_6$. First note that for all $x\in\hat B$, $y\in B^c$
\[|x-y| \ge |y-x_0|-|x-x_0| \ge (1-\sigma')|y-x_0|.\]
Recalling that $0\le\xi\le 1$, we estimate $H_6$ as follows:
\begin{align*}
|H_6| &\le \frac{\gamma t^\frac{1}{p}}{(1-\sigma')^{N+ps}}\int_0^t\iint_{\hat B\times B^c}\frac{u^{p-1}(x,\tau)}{|y-x_0|^{N+ps}}\Big|(u_{\bar h}(x,\tau)+\nu)^\frac{p-2}{p}-(u(x,\tau)+\nu)^\frac{p-2}{p}\Big|\,dx\,dy\,d\tau \\
&+ \frac{\gamma t^\frac{1}{p}}{(1-\sigma')^{N+ps}}\int_0^t\iint_{\hat B\times B^c}\frac{|u(y,\tau)|^{p-1}}{|y-x_0|^{N+ps}}\Big|(u_{\bar h}(x,\tau)+\nu)^\frac{p-2}{p}-(u(x,\tau)+\nu)^\frac{p-2}{p}\Big|\,dx\,dy\,d\tau \\
&= \frac{\gamma t^\frac{1}{p}}{(1-\sigma')^{N+ps}}(H_9+H_{10}).
\end{align*}
We deal with $H_9$ by H\"older's inequality and Lagrange's rule:
\begin{align*}
H_9 &\le \Big[\int_0^t\int_{\hat B}u^p(x,\tau)\,dx\,d\tau\Big]^\frac{p-1}{p}\Big[\int_0^t\int_{\hat B}\Big|(u_{\bar h}(x,\tau)+\nu)^\frac{p-2}{p}-(u(x,\tau)+\nu)^\frac{p-2}{p}\Big|^p\,dx\,d\tau\Big]^\frac{1}{p}\int_{B^c}\frac{dy}{|y-x_0|^{N+ps}} \\
&\le \frac{\gamma}{\rho^{ps}}\|u\|_{L^p(\hat B\times(0,t))}^{p-1}\Big[\int_0^t\int_{\hat B}\max\big\{(u_{\bar h}(x,\tau)+\nu)^{-\frac{2}{p}},\,(u(x,\tau)+\nu)^{-\frac{2}{p}}\big\}^p|u_{\bar h}(x,\tau)-u(x,\tau)|^p\,dx\,d\tau\Big]^\frac{1}{p} \\
&\le \frac{\gamma}{\nu^\frac{2}{p}\rho^{ps}}\|u\|_{L^p(\hat B\times(0,t))}^{p-1}\|u_{\bar h}-u\|_{L^1(\hat B\times(0,t))}.
\end{align*}
The first multiplier above is bounded, while the second tends to $0$ as $h\to 0$ by Proposition \ref{con} \ref{con2} (with $q=1$). So we have
\[\lim_{h\to 0}\,H_9 = 0.\]
For $H_{10}$ we use a different approach, separating the space integration variables. For all $\tau\in(0,t)$ we have
\begin{align*}
& \iint_{\hat B\times B^c}\frac{|u(y,\tau)|^{p-1}}{|y-x_0|^{N+ps}}\Big|(u_{\bar h}(x,\tau)+\nu)^\frac{p-2}{p}-(u(x,\tau)+\nu)^\frac{p-2}{p}\Big|\,dx\,dy \\
&\le \gamma\Big[\int_{\hat B}\max\big\{(u_{\bar h}(x,\tau)+\nu)^{-\frac{2}{p}},\,(u(x,\tau)+\nu)^{-\frac{2}{p}}\big\}|u_{\bar h}(x,\tau)-u(x,\tau)|\,dx\Big]\int_{B^c}\frac{|u(y,\tau)|^{p-1}}{|y-x_0|^{N+ps}}\,dy \\
&\le \frac{\gamma}{\nu^\frac{2}{p}\rho^{ps}}\Big[\int_{\hat B}|u_{\bar h}(x,\tau)-u(x,\tau)|\,dx\Big]\Big[\rho^{ps}\int_{B^c}\frac{|u(y,\tau)|^{p-1}}{|y-x_0|^{N+ps}}\,dy\Big].
\end{align*}
Now integrate in time and recall \eqref{tail}:
\[H_{10} \le \frac{\gamma}{\nu^\frac{2}{p}\rho^{ps}}\|u_{\bar h}-u\|_{L^1(\hat B\times(0,t))}{\rm Tail}(u,x_0,\rho,0,t)^{p-1}.\]
As soon as $h\to 0$, the first term tends to $0$ by Proposition \ref{con} \ref{con2} ($q=1$) and the second is bounded, hence
\[\lim_{h\to 0}\,H_{10} = 0.\]
Therefore, we have
\beq\label{mol6}
\lim_{h\to 0}\,H_6 = 0.
\eeq
Now \eqref{mol5} and \eqref{mol6} imply \eqref{mol4}. This, in turn, leads to
\beq\label{mol7}
\limsup_{\eps,h\to 0}\,H_2 \ge - \gamma\int_0^t\iint_{\R^N\times\R^N}\frac{G(x,y,\tau)F(x,y,\tau)}{|x-y|^{N+ps}}\,dx\,dy\,d\tau.
\eeq
Finally, we use both \eqref{mol2} and \eqref{mol7} and pass to the limit in \eqref{mol1} to conclude.
\end{proof}

\noindent
Next we give a proof of inequality \eqref{bl1} in Lemma \ref{bl}:

\begin{lemma}\label{mob}
Let $u$, $r$, $t_0$, $\sigma$, $\xi$, $B$, $\tilde B$, $\hat B$ be as in Lemma \ref{bl}. Then, there exists $\gamma>1$ depending on the data and $r$, s.t.\
\begin{align*}
0 &\ge \frac{1}{r}\int_{\tilde B}u^r(x,\tau)\,dx\Big|_0^{t_0} \\
&+ \frac{1}{\gamma}\int_0^{t_0}\iint_{\R^N\times\R^N}\frac{|u(x,\tau)-u(y,\tau)|^{p-2}(u(x,\tau)-u(y,\tau))}{|x-y|^{N+ps}}\Big[u^{r-1}(x,\tau)\xi^p(x)-u^{r-1}(y,\tau)\xi^p(y)\Big]\,dx\,dy\,d\tau.
\end{align*}
\end{lemma}
\begin{proof}
The argument closely follows that of Lemma \ref{mol}, so we will omit some steps. Fix $\eps\in(0,t_0/2)$ and define the Lipschitz mapping $\psi_\eps:[0,T]\to\R$ as in Lemma \ref{mol}, with $t$ replaced by $t_0$. Also fix $h\in(0,1)$, $\nu>0$, and set for all $(x,\tau)\in\R^N\times(0,T)$
\[\varphi_h^\eps(x,\tau) = (u_{\bar h}(x,\tau)+\nu)^{r-1}\xi^p(x)\psi_\eps(\tau),\]
where the parameter $\nu$ is included in order to avoid singularities in the case $1<r<2$. By Proposition \ref{con}, the function $\varphi_h^\eps\in W^{1,2}_0(0,t_0,L^2(\tilde B))\cap L^p(0,t_0,W^{s,p}_0(\tilde B))$ can be chosen as a test for \eqref{eq} in $\tilde B\times(0,t_0)$ (see Definition \ref{lws}), producing
\begin{align}\label{mob1}
0 &= -\int_0^{t_0}\int_{\tilde B}u(x,\tau)\frac{\partial\varphi_h^\eps}{\partial\tau}(x,\tau)\,dx\,d\tau \\
\nonumber &+ \int_0^{t_0}\iint_{\R^N\times\R^N}|u(x,\tau)-u(y,\tau)|^{p-2}(u(x,\tau)-u(y,\tau))(\varphi_h^\eps(x,\tau)-\varphi_h^\eps(y,\tau))K(x,y,\tau)\,dx\,dy\,d\tau \\
\nonumber &= H_1+H_2.
\end{align}
First we split $H_1$ as follows:
\begin{align*}
H_1 &= -\int_0^{t_0}\int_{\tilde B}u_{\bar h}(x,\tau)\frac{\partial\varphi_h^\eps}{\partial\tau}(x,\tau)\,dx\,d\tau+\int_0^{t_0}\int_{\tilde B}(u_{\bar h}(x,\tau)-u(x,\tau))\frac{\partial\varphi_h^\eps}{\partial\tau}(x,\tau)\,dx\,d\tau \\
&= H_3+H_4.
\end{align*}
Repeated integration by parts and the definition of $\psi_\eps$ allows for the following rephrasing of $H_3$:
\[H_3 = \frac{1}{r\eps}\int_{t_0-\eps}^{t_0}\int_{\tilde B}(u_{\bar h}(x,\tau)+\nu)^r\xi^p(x)\,dx\,d\tau-\frac{1}{r\eps}\int_0^\eps\int_{\tilde B}(u_{\bar h}(x,\tau)+\nu)^r\xi^p(x)\,dx\,d\tau.\]
Letting $\eps,h,\nu\to 0$ gives
\[\lim_{\eps,h,\nu\to 0}\,H_3 = \frac{1}{r}\int_{\tilde B}u^r(x,t_0)\xi^p(x)\,dx-\frac{1}{r}\int_{\tilde B}u^r(x,0)\xi^p(x)\,dx.\]
To estimate $H_4$, we apply Proposition \ref{con} \ref{con1}:
\begin{align*}
H_4 &= \int_0^{t_0}\int_{\tilde B}(u_{\bar h}(x,\tau)-u(x,\tau))\frac{\partial}{\partial\tau}(u_{\bar h}(x,\tau)+\nu)^{r-1}\xi^p(x)\psi_\eps(\tau)\,dx\,d\tau \\
&+ \int_0^{t_0}\int_{\tilde B}(u_{\bar h}(x,\tau)-u(x,\tau))(u_{\bar h}(x,\tau)+\nu)^{r-1}\xi^p(x)\psi'_\eps(\tau)\,dx\,d\tau \\
&= (r-1)\int_0^{t_0}\int_{\tilde B}\frac{(u_{\bar h}(x,\tau)-u(x,\tau))^2}{h}(u_{\bar h}(x,\tau)+\nu)^{r-2}\xi^p(x)\psi_\eps(\tau)\,dx\,d\tau \\
&+ \frac{1}{\eps}\int_0^\eps\int_{\tilde B}(u_{\bar h}(x,\tau)-u(x,\tau))(u_{\bar h}(x,\tau)+\nu)^{r-1}\xi^p(x)\,dx\,d\tau \\
&- \frac{1}{\eps}\int_{t_0-\eps}^{t_0}\int_{\tilde B}(u_{\bar h}(x,\tau)-u(x,\tau))(u_{\bar h}(x,\tau)+\nu)^{r-1}\xi^p(x)\,dx\,d\tau.
\end{align*}
The first integral above is non-negative, while the others tend to $0$ as $h\to 0$, due to convergence of $u_{\bar h}$ to $u$ in $L^1(\tilde B\times(0,t_0))$ (Proposition \ref{con} \ref{con2} with $q=1$). In conclusion we have
\[\liminf_{\eps,h,\nu\to 0}\,H_4 \ge 0,\]
which along with the previous relation gives
\beq\label{mob2}
\limsup_{\eps,h,\nu\to 0}\,H_1 \ge \frac{1}{r}\int_{\tilde B}u^r(x,\tau)\xi^p(x)\,dx\Big|_0^{t_0}.
\eeq
Now for the diffusive term $H_2$, with a similar argument as in Lemma \ref{mol}. First we define $G$ as in Lemma \ref{mol}, then we set for all $(x,y,\tau)\in\R^N\times\R^N\times(0,T)$
\[F_h(x,y,\tau) = (u_{\bar h}(x,\tau)+\nu)^{r-1}\xi^p(x)-(u_{\bar h}(x,\tau)+\nu)^{r-1}\xi^p(y),\]
\[F(x,y,\tau) = (u(x,\tau)+\nu)^{r-1}\xi^p(x)-(u(x,\tau)+\nu)^{r-1}\xi^p(y).\]
Clearly
\[\int_0^{t_0}\iint_{\hat B\times\hat B}\frac{|G(x,y,\tau)|^\frac{p}{p-1}}{|x-y|^{N+ps}}\,dx\,dy\,d\tau \le \|u\|_{L^p(0,t_0,W^{s,p}(\hat B))}^p,\]
while by Proposition \ref{con} we have a.e.\ in $\R^N\times\R^N\times(0,T)$
\[\lim_{h\to 0}\,F_h(x,y,\tau) = F(x,y,\tau).\]
Also, by $(K_2)$, we have for some $\gamma>1$ depending on the data
\[\liminf_{\eps\to 0}H_2 \ge \frac{1}{\gamma}\int_0^{t_0}\iint_{\R^N\times\R^N}\frac{G(x,y,\tau)F_h(x,y,\tau)}{|x-y|^{N+ps}}\,dx\,dy\,d\tau.\]
By symmetry and the properties of $\xi$ we have
\begin{align*}
& \int_0^{t_0}\iint_{\R^N\times\R^N}\frac{G(x,y,\tau)}{|x-y|^{N+ps}}\big[F_h(x,y,\tau)-F(x,y,\tau)\big]\,dx\,dy\,d\tau \\
&= \int_0^{t_0}\iint_{\hat B\times\hat B}\frac{G(x,y,\tau)}{|x-y|^{N+ps}}\big[F_h(x,y,\tau)-F(x,y,\tau)\big]\,dx\,dy\,d\tau \\
&+ 2\int_0^{t_0}\iint_{\tilde B\times\hat B^c}\frac{G(x,y,\tau)}{|x-y|^{N+ps}}\big[(u_{\bar h}(x,\tau)+\nu)^{r-1}-(u(x,\tau)+\nu)^{r-1}\big]\xi^p(x)\,dx\,dy\,d\tau \\
&= H_5+H_6.
\end{align*}
We first consider $H_5$, noting that by boundedness of $\xi$
\begin{align*}
\int_0^{t_0}\iint_{\hat B\times\hat B}\frac{|F_h(x,y,\tau)|^p}{|x-y|^{N+ps}}\,dx\,dy\,d\tau &\le \gamma\int_0^{t_0}\iint_{\hat B\times\hat B}\frac{|(u_{\bar h}(x,\tau)+\nu)^{r-1}-(u_{\bar h}(y,\tau)+\nu)^{r-1}|^p}{|x-y|^{N+ps}}\,dx\,dy\,d\tau \\
&+ \gamma\int_0^{t_0}\iint_{\hat B\times\hat B}(u_{\bar h}(y,\tau)+\nu)^{p(r-1)}\frac{|\xi^p(x)-\xi^p(y)|^p}{|x-y|^{N+ps}}\,dx\,dy\,d\tau \\
&= H_7+H_8.
\end{align*}
Recalling that $u$ is locally bounded in the present framework, we set
\beq\label{knu}
K_\nu = \begin{cases}
\nu^{r-2} & \text{if $1<r<2$} \\
\big(\|u\|_{L^\infty(\hat B\times(0,t_0))}+\nu\big)^{r-2} & \text{if $r\ge 2$.}
\end{cases}
\eeq
To estimate $H_7$, we apply Lagrange's rule and Proposition \ref{con} \ref{con3}:
\begin{align*}
H_7 &\le \gamma\int_0^{t_0}\iint_{\hat B\times\hat B}\max\big\{(u_{\bar h}(x,\tau)+\nu)^{r-2},\,(u_{\bar h}(y,\tau)+\nu)^{r-2}\big\}^p\frac{|u_{\bar h}(x,\tau)-u_{\bar h}(y,\tau)|^p}{|x-y|^{N+ps}}\,dx\,dy\,d\tau \\
&\le \gamma K_\nu^p\int_0^{t_0}\iint_{\hat B\times\hat B}\frac{|u_{\bar h}(x,\tau)-u_{\bar h}(y,\tau)|^p}{|x-y|^{N+ps}}\,dx\,dy\,d\tau \\
&\le \gamma K_\nu^p\|u\|_{L^p(0,t_0,W^{s,p}(\hat B))}^p.
\end{align*}
For $H_8$, we act mainly on $\xi$:
\begin{align*}
H_8 &\le \gamma\int_0^{t_0}\iint_{\hat B\times\hat B}(u_{\bar h}(y,\tau)+\nu)^{p(r-1)}\max\big\{\xi^{p-1}(x),\,\xi^{p-1}(y)\big\}^p\frac{|\xi(x)-\xi(y)|^p}{|x-y|^{N+ps}}\,dx\,dy\,d\tau \\
&\le \frac{\gamma}{\sigma^p\rho^p}\int_0^{t_0}\iint_{\hat B\times\hat B}\frac{(u_{\bar h}(y,\tau)+\nu)^{p(r-1)}}{|x-y|^{N+ps-p}}\,dx\,dy\,d\tau \\
&\le \frac{\gamma}{\sigma^p\rho^p}\Big[\int_0^{t_0}\int_{\hat B}(u_{\bar h}(y,\tau)+\nu)^{p(r-1)}\,dy\,d\tau\Big]\Big[\sup_{y\in\hat B}\,\int_{\hat B}\frac{dx}{|x-y|^{N+ps-p}}\Big] \\
&\le \frac{\gamma t_0\rho^{N-ps}}{\sigma^p}\big[\|u\|_{L^\infty(\hat B\times(0,t_0))}+\nu\big]^{p(r-1)}.
\end{align*}
Combining the previous bounds, we can find $\gamma>0$ depending on $\nu$, but not on $h$, s.t.\
\[\int_0^{t_0}\iint_{\hat B\times\hat B}\frac{|F_h(x,y,\tau)|^p}{|x-y|^{N+ps}}\,dx\,dy\,d\tau \le \gamma.\]
By reflexivity, passing to a sequence $h_n\to 0$, we get
\beq\label{mob3}
\limsup_{h\to 0}\,H_5 \ge 0.
\eeq
To estimate $H_6$, we first note that for all $(x,y)\in\tilde B\times\hat B^c$
\[|x-y| \ge \frac{\sigma}{2(1+\sigma)}|y-x_0|.\]
Using Lagrange's rule, subadditivity and the constant $K_\nu$ defined in \eqref{knu}, we have:
\begin{align*}
|H_6| &\le \gamma\int_0^{t_0}\iint_{\tilde B\times\hat B^c}\frac{|G(x,y,\tau)|}{|x-y|^{N+ps}}\Big|(u_{\bar h}(x,\tau)+\nu)^{r-1}-(u(x,\tau)+\nu)^{r-1}\Big|\,dx\,dy\,d\tau \\
&\le \gamma K_\nu\Big[\int_0^{t_0}\int_{\tilde B}|u_{\bar h}(x,\tau)-u(x,\tau)|\,dx\,d\tau\Big]\Big[\sup_{(x,\tau)\in\tilde B\times(0,t_0)}\,\int_{\hat B^c}\frac{|u(x,\tau)-u(y,\tau)|^{p-1}}{|x-y|^{N+ps}}\,dy\Big] \\
&\le \frac{\gamma K_\nu}{\sigma^{N+ps}}\|u_{\bar h}-u\|_{L^1(\tilde B\times(0,t_0))}\Big[\|u\|_{L^\infty(\tilde B\times(0,t_0))}^{p-1}\int_{\hat B^c}\frac{dy}{|y-x_0|^{N+ps}}+\sup_{0<\tau<t_0}\,\int_{\hat B^c}\frac{|u(y,\tau)|^{p-1}}{|y-x_0|^{N+ps}}\,dy\Big] \\
&\le \frac{\gamma K_\nu}{\sigma^{N+ps}\rho^{ps}}\|u_{\bar h}-u\|_{L^1(\tilde B\times(0,t_0))}\Big[\|u\|_{L^\infty(\tilde B\times(0,t_0))}^{p-1}+{\rm Tail}(u,x_0,(1+\sigma)\rho,0,t_0)^{p-1}\Big]
\end{align*}
(where again we have used boundedness of $u$). By Proposition \ref{con} \ref{con2} (with $q=1$) the latter tends to $0$ as $h\to 0$, hence
\beq\label{mob4}
\lim_{h\to 0}\,H_6 = 0.
\eeq
By \eqref{mob3} and \eqref{mob4} we obtain
\[\limsup_{h\to 0}\int_0^{t_0}\iint_{\R^N\times\R^N}\frac{G(x,y,\tau)}{|x-y|^{N+ps}}\big[F_h(x,y,\tau)-F(x,y,\tau)\big]\,dx\,dy\,d\tau \ge 0,\]
which in turn implies
\begin{align}\label{mob5}
& \limsup_{\eps,h,\nu\to 0}\,H_2 \\
\nonumber &\ge \frac{1}{\gamma}\int_0^{t_0}\iint_{\R^N\times\R^N}\frac{|u(x,\tau)-u(y,\tau)|^{p-2}(u(x,\tau)-u(y,\tau))}{|x-y|^{N+ps}}\Big[u^{r-1}(x,\tau)\xi^p(x)-u^{r-1}(y,\tau)\xi^p(y)\Big]\,dx\,dy\,d\tau.
\end{align}
Passing to the limit in \eqref{mob1} and using \eqref{mob2} and \eqref{mob5}, we get
\begin{align*}
0 &\ge \frac{1}{r}\int_{\tilde B}u^r(x,\tau)\xi^p(x)\,dx\Big|_0^{t_0} \\
&+ \frac{1}{\gamma}\int_0^{t_0}\iint_{\R^N\times\R^N}\frac{|u(x,\tau)-u(y,\tau)|^{p-2}(u(x,\tau)-u(y,\tau))}{|x-y|^{N+ps}}\Big[u^{r-1}(x,\tau)\xi^p(x)-u^{r-1}(y,\tau)\xi^p(y)\Big]\,dx\,dy\,d\tau,
\end{align*}
which concludes the proof.
\end{proof}

\noindent
We finally include the formal derivation of inequalities \eqref{fte3} and \eqref{fte5} in the proof of Theorem \ref{fte}:

\begin{lemma}\label{mot}
Let $\Omega$, $u_0$, $u$, and $q$ be as in Theorem \ref{fte}:
\begin{enumroman}
\item\label{mot1} if $1<p<p_c$, then there exists $\gamma>0$ depending on $N,p,s$ s.t.\ for all $\psi\in W^{1,2}_0(0,T)$, $\psi\ge 0$ in $(0,T)$
\[\int_0^T\Big[-\|u(\cdot,\tau)\|_{L^q(\Omega)}^q\psi'(\tau)+\frac{C_1}{\gamma}\|u(\cdot,\tau)\|_{L^q(\Omega)}^{p+q-2}\psi(\tau)\Big]\,d\tau \le 0;\]
\item\label{mot2} if $p_c\le p<2$, then there exists $\gamma>0$ depending on $N,p,s$ s.t.\ for all $\psi\in W^{1,2}_0(0,T)$, $\psi\ge 0$ in $(0,T)$
\[\int_0^T\Big[-\|u(\cdot,\tau)\|_{L^2(\Omega)}^2\psi'(\tau)+\frac{C_1}{\gamma}|\Omega|^{-\frac{\lambda_2}{2N}}\|u(\cdot,\tau)\|_{L^2(\Omega)}^p\psi(\tau)\Big]\,d\tau \le 0.\]
\end{enumroman}
\end{lemma}
\begin{proof}
First recall that $u$ is locally essentially bounded in $\Omega_T$. We consider \ref{mot1}. We already know that $q>2$. Fix $h\in(0,1)$ and set for all $(x,\tau)\in\R^N\times(0,T)$
\[\varphi_h(x,\tau) = u_{\bar h}^{q-1}(x,\tau)\psi(\tau),\]
so that $\varphi_h\in W^{1,2}_0(0,T,L^2(\Omega))\cap L^p(0,T,W^{s,p}_0(\Omega))$ can be used as a test in Definition \ref{scp} \ref{scp1}, giving
\begin{align}\label{mot3}
0 &= \int_\Omega u(x,\tau)\varphi_h(x,\tau)\,dx\Big|_0^T-\int_0^T\int_\Omega u(x,\tau)\frac{\partial\varphi_h}{\partial\tau}(x,\tau)\,dx\,d\tau \\
\nonumber &+ \int_0^T\iint_{\R^N\times\R^N}|u(x,\tau)-u(y,\tau)|^{p-2}(u(x,\tau)-u(y,\tau))(\varphi_h(x,\tau)-\varphi_h(y,\tau))K(x,y,\tau)\,dx\,dy\,d\tau \\
\nonumber &= H_1+H_2.
\end{align}
We focus on the evolutive term $H_1$. Noting that $\varphi_h(\cdot,0)=\varphi_h(\cdot,T)=0$, we have
\begin{align*}
H_1 &= -\int_0^T\int_\Omega u_{\bar h}(x,\tau)\frac{\partial\varphi_h}{\partial\tau}(x,\tau)\,dx\,d\tau+\int_0^T\int_\Omega (u_{\bar h}(x,\tau)-u(x,\tau))\frac{\partial\varphi_h}{\partial\tau}(x,\tau)\,dx\,d\tau \\
&= H_3+H_4.
\end{align*}
Using Proposition \ref{con} \ref{con1} and the integration by parts formula (twice), we have
\begin{align*}
H_3 &= \int_0^T\int_\Omega\frac{\partial u_{\bar h}}{\partial\tau}(x,\tau)u_{\bar h}^{q-1}(x,\tau)\psi(\tau)\,dx\,d\tau \\
&= \frac{1}{q}\int_0^T\int_\Omega\frac{\partial u_{\bar h}^q}{\partial\tau}(x,\tau)\psi(\tau)\,dx\,d\tau \\
&= -\frac{1}{q}\int_0^T\int_\Omega u_{\bar h}^q(x,\tau)\psi'(\tau)\,dx\,d\tau.
\end{align*}
Therefore, by Proposition \ref{con} \ref{con2} we have
\[\lim_{h\to 0}\,H_3 = -\frac{1}{q}\int_0^T\int_\Omega u^q(x,\tau)\psi'(\tau)\,dx\,d\tau.\]
Besides, recalling that $q>2$ and applying Proposition \ref{con} \ref{con1}, we get
\begin{align*}
H_4 &= (q-1)\int_0^T\int_\Omega(u_{\bar h}(x,\tau)-u(x,\tau))u_{\bar h}^{q-2}(x,\tau)\frac{\partial u_{\bar h}}{\partial\tau}(x,\tau)\psi(\tau)\,dx\,d\tau \\
&+ \int_0^T\int_\Omega(u_{\bar h}(x,\tau)-u(x,\tau))u_{\bar h}^{q-1}(x,\tau)\psi'(\tau)\,dx\,d\tau \\
&= (q-1)\int_0^T\int_\Omega\frac{(u_{\bar h}(x,\tau)-u(x,\tau))^2}{h}u_{\bar h}^{q-2}(x,\tau)\psi(\tau)\,dx\,d\tau \\
&+ \int_0^T\int_\Omega u_{\bar h}^q(x,\tau)\psi'(\tau)\,dx\,d\tau-\int_0^T\int_\Omega u_{\bar h}^{q-1}(x,\tau)u(x,\tau)\psi'(\tau)\,dx\,d\tau.
\end{align*}
The first integral is non-negative, while the second and third integrals converge to the same limit as $h\to 0$ by Proposition \ref{con} \ref{con2}, so we have
\[\liminf_{h\to 0}\,H_4 \ge 0.\]
Combining the estimates above, we get
\beq\label{mot4}
\liminf_{h\to 0}\,H_1 \ge -\frac{1}{q}\int_0^T\|u(\cdot,\tau)\|_{L^q(\Omega)}^q\psi'(\tau)\,d\tau.
\eeq
We deal with the diffusive term $H_2$ as in Lemma \ref{mol} (this case is in fact simpler). Define $G$ as in Lemma \ref{mol}, and set for all $(x,y,t)\in\R^N\times\R^N\times(0,T)$
\[F_h(x,y,t) = (u_{\bar h}^{q-1}(x,t)-u_{\bar h}^{q-1}(y,t))\psi(t),\]
\[F(x,y,t) = (u^{q-1}(x,t)-u^{q-1}(y,t))\psi(t).\]
As usual we have $F_h(x,y,t)\to F(x,y,t)$ a.e.\ in $\R^N\times\R^N\times(0,T)$, as $h\to 0$. Also, by positivity and $(K_2)$, we have
\[H_2 \ge C_1\int_0^T\iint_{\R^N\times\R^N}\frac{G(x,y,\tau)F_h(x,y,\tau)}{|x-y|^{N+ps}}\,dx\,dy\,d\tau.\]
The fractional Poincar\'e inequality implies
\[\int_0^T\iint_{\R^N\times\R^N}\frac{|G(x,y,t)|^\frac{p}{p-1}}{|x-y|^{N+ps}}\,dx\,dy\,d\tau \le \|u\|_{L^p(0,T,W^{s,p}_0(\Omega))}^p.\]
We use boundedness of $\psi$, $q>2$, and Proposition \ref{con} \ref{con2} \ref{con3} to get
\begin{align*}
\int_0^T\iint_{\R^N\times\R^N}\frac{|F_h(x,y,\tau)|^p}{|x-y|^{N+ps}}\,dx\,dy\,d\tau & \le \gamma\int_0^T\iint_{\R^N\times\R^N}\frac{|u_{\bar h}^{q-1}(x,\tau)-u_{\bar h}^{q-1}(y,\tau)|^p}{|x-y|^{N+ps}}\,dx\,dy\,d\tau \\
&\le \gamma\int_0^T\iint_{\R^N\times\R^N}\max\big\{u_{\bar h}^{q-2}(x,\tau),\,u_{\bar h}^{q-2}(y,\tau)\big\}^p\frac{|u_{\bar h}(x,\tau)-u_{\bar h}(y,\tau)|^p}{|x-y|^{N+ps}}\,dx\,dy\,d\tau \\
&\le \gamma\|u\|_{L^\infty(\Omega_T)}^{p(q-2)}\|u\|_{L^p(0,T,W^{s,p}_0(\Omega))}^p.
\end{align*}
By reflexivity, along a sequence $h_n\to 0$ we have
\[\lim_n\,\int_0^T\iint_{\R^N\times\R^N}\frac{G(x,y,\tau)F_{h_n}(x,y,\tau)}{|x-y|^{N+ps}}\,dx\,dy\,d\tau = \int_0^T\iint_{\R^N\times\R^N}\frac{G(x,y,\tau)F(x,y,\tau)}{|x-y|^{N+ps}}\,dx\,dy\,d\tau.\]
Passing to the limit in $H_2$, using Lemma \ref{alg} \ref{alg1}, and Sobolev's inequality \eqref{sob}, we get
\begin{align}\label{mot5}
\limsup_{h\to 0}\,H_2 &\ge C_1\int_0^T\iint_{\R^N\times\R^N}\frac{|u(x,\tau)-u(y,\tau)|^{p-2}(u(x,\tau)-u(y,\tau))}{|x-y|^{N+ps}}(u^{q-1}(x,\tau)-u^{q-1}(y,\tau))\psi(\tau)\,dx\,dy\,d\tau \\
\nonumber &\ge \frac{C_1}{\gamma}\int_0^T\iint_{\R^N\times\R^N}\Big|u^\frac{p+q-2}{p}(x,\tau)-u^\frac{p+q-2}{p}(y,\tau)\Big|^p\,\frac{dx\,dy}{|x-y|^{N+ps}}\psi(\tau)d\tau \\
\nonumber &\ge \frac{C_1}{\gamma}\int_0^T\Big[\int_\Omega u^{\frac{p+q-2}{p}p^*_s}(x,\tau)\,dx\Big]^\frac{p}{p^*_s}\psi(\tau)\,d\tau \\
\nonumber &= \frac{C_1}{\gamma}\int_0^T\|u(\cdot,\tau)\|_{L^q(\Omega)}^{p+q-2}\psi(\tau)\,d\tau,
\end{align}
where we have used the relation
\[\bigg( \frac{p+q-2}{p}\bigg) p^*_s = q.\]
Passing to the limit as $h\to 0$ in \eqref{mot3} and using \eqref{mot4} and \eqref{mot5}, we arrive at
\[0 \ge -\frac{1}{q}\int_0^T\|u(\cdot,\tau)\|_{L^q(\Omega)}^q\psi'(\tau)\,dx\,d\tau+\frac{C_1}{\gamma}\int_0^T\|u(\cdot,\tau)\|_{L^q(\Omega)}^{p+q-2}\psi(\tau)\,d\tau,\]
which yields \ref{mot1} with a new constant $\gamma>0$ depending on $N,p,s$. The argument for \ref{mot2} is entirely analogous, just replacing $q$ with $2$ and taking care of the measure-type multiplier.
\end{proof}

\vskip4pt
\noindent
{\bf Acknowledgement.} The authors are members of GNAMPA (Gruppo Nazionale per l'Analisi Matematica, la Probabilit\`a e le loro Applicazioni) of INdAM (Istituto Nazionale di Alta Matematica 'Francesco Severi'), and are supported by the research project {\em Regolarit\`a ed esistenza per operatori anisotropi} (GNAMPA, CUP E5324001950001). In addition, F.C.\ and A.I.\ are supported by the research project {\em Partial Differential Equations and their role in understanding natural phenomena} (Fondazione di Sardegna 2023 CUP F23C25000080007), while S.C.\ acknowledges the foundings of PNR (MUR) 2021-2027 and the Department of Mathematics of the University of Bologna. Finally we acknowledge the helpful suggestions and constructive critism of the referees, that has helped to improve the quality of the present work.

\end{document}